%% file: cluster-triangle-alexander.tex
\documentclass[leqno]{amsart}

\usepackage{amsmath}
\usepackage{amssymb}
\usepackage{amsthm}
\usepackage{color}
\usepackage[dvips]{graphicx}
\usepackage{amsfonts}
\usepackage{float}

\theoremstyle{definition}

  \newtheorem{thm}{Theorem}[section]
  
  \newtheorem{lem}[thm]{Lemma}
  \newtheorem{cor}[thm]{Corollary}
  \newtheorem{prop}[thm]{Proposition}
  
  \newtheorem{ex}[thm]{Example}
  
  \newtheorem{defi}[thm]{Definition}

\theoremstyle{remark}
\newtheorem{rem}{Remark}

\newcommand{\Z}{\mathbb{Z}}

\def\Z{{\Bbb{Z}}}

\def\>{{\geq }}
\def\<{{\leq }}

\usepackage[all]{xy}
\def\cv{X} 
\def\cf#1#2{[a_1,a_2,\dots,a_{#1}#2]} 
\def\fan{\Delta} 
\def\path{\gamma}
\def\tri{T}
\def\AT#1{A(#1)} 
\def\w{\operatorname{wt}} 
\def\Trop{\operatorname{Trop}} 
\def\tb{K} 
\def\alexander#1{\Delta_{#1}} 
\def\alex#1{\alexander{#1}} 
\def\trisign#1#2{e_{#1}{#2}} 
\def\Fp#1{F_{#1}} 
\def\SFp#1{\mathcal{F}_{#1}} 
\def\alexsign#1{\epsilon_{#1}}
\def\alexexp#1{d_{#1}}
\def\Falex#1{\alexsign{#1} t^{\alexexp{#1}} \SFp{#1}}

\title{Cluster variables, ancestral triangles and Alexander polynomials}

\author{Wataru~Nagai}%
\address{Graduate School of Information Science and Engineering, 
	Tokyo Institute of Technology, 
	2-12-1 Ookayama, Meguro-ku, Tokyo 152-8550, Japan.}
\email{nagai.wataru.aa@gmail.com}

\author{Yuji~Terashima}%
\address{Graduate School of Science, Tohoku University,
	6-3, Aramaki Aza-Aoba, Aoba-ku, Sendai 980-8578, Japan.}
\email{yujiterashima@tohoku.ac.jp}

\begin{document}
	\maketitle
	\section{Introduction}
	Cluster algebras which were introduced by Sergey Fomin 
	and Andrei Zelevinsky \cite{FZ} 
	has found surprising relations with many branches of mathematics and
	mathematical physics, such as representation theory, 
	Donaldson-Thomas theory, low dimensional topology, quantum field theories.
	
	In this paper, we give a new relation between cluster algebras and Alexander polynomials for $2$-bridge knots. 
	A key tool is an ancestral triangle which appeared in both quantum topology \cite{Y, KW} and hyperbolic geometry \cite{HO, HT, H} in different ways. 
	First, we give a combinatorial formula to express 
	cluster variables associated with ancestral triangles as 
	generating polynomials over paths with weights on ancestral triangles. Second, we find specializations of the cluster variables which are identified with Alexander polynomials for $2$-bridge knots. 

	Kyungyong Lee and Ralf Schiffler \cite{LS} give a very interesting formula to express Jones polynomials for $2$-bridge knots as specializations for 
	cluster variables. Therefore, with our results, 
	it implies that cluster variables know 
	both Alexander polynomials and Jones polynomials for $2$-bridge knots through 
	different specializations.
	It would be interesting to generalize the results to any knots. 
	We remark that they use even continued fractions, though we use positive continued fractions. 
	
	The paper is organized as follows. In section \ref{section:ancestral_triangle}, we construct ancestral triangles and define path and its weight. We recall cluster variables in section \ref{section:cluster_variable} and show the cluster expansion formula using paths. We also show some recursions of $F$-polynomial as corollaries of the formula. In section \ref{section:alexander_polynomial}, we review two-bridge links and Alexander polynomials. Then we show some recursions for the Alexander polynomials of two-bridge links. Finally we state and prove the main theorem. 
	
	{\it Acknowledgments.} We would like to thank R. Inoue, S. Kano, A. Kato, Y. Mizuno, T. Tsuda and M. Wakui for valuable discussions. We are grateful to M. Wakui for careful reading of a draft of the paper and valuable comments. This work is partially supported by JSPS KAKENHI Grant Number 17K05243 and by JST CREST Grant Number JPMJCR14D6, Japan.      

\section{Ancestral triangle}
\label{section:ancestral_triangle}
\subsection{Construction of ancestral triangle}
In this section, we construct an ancestral triangle using continued fraction expansion.
Ancestral triangles are introduced by Yamada to compute the Kauffman bracket polynomial of two-bridge knots \cite{Y,KW}. It is also known as Farey diagrams which appear in hyperbolic geometry \cite{HO, HT, H}.

Let $p/q$ be an irreducible fraction between $0$ and $1$. Let $a_1,a_2,\dots,a_n$ be positive integers, we denote a continued fraction
\[
\cfrac{1}{a_1 + \cfrac{1}{a_2 + \cfrac{1}{\ddots + \cfrac{1}{a_n}}}}
\]
by $\cf{n}$.
We say that a continued fraction $\cf{n}$ is a continued fraction expansion of $p/q$ if $p/q = \cf{n}$. Note that all fractions have exactly two continued fraction expansions $\cf{n}{} = \cf{n}{-1, 1}$.

We construct the \emph{ancestral triangle} of $p/q$ by stacking 
triangles. We say that a triangle is right or left as shown in Figure \ref{fig:right_and_left_triangle}. Let $\cf{n}$ be one of continued fraction expansion of $p/q$. First, we stack $a_1 - 1$ right triangles. Then, we stack $a_2$ left triangles, $a_3$ right triangles, $\dots$, and $a_n$ right (resp. left) triangles if $n$ is odd (resp. even). 
We label all vertices with fractions according to following rules. The left vertex of the bottom edge of ancestral triangle is labeled $0/1$, and right of that is $1/1$. If two vertices in a triangle are labeled $p/q$ and $r/s$, then the other is labeled $(p+r)/(q+s)$. We denote the ancestral triangle of $p/q$ by $\AT{p/q}$.

\begin{rem}
	The ancestral triangles constructed above are the mirror images of that of Yamada's definition or Farey diagrams.
\end{rem}

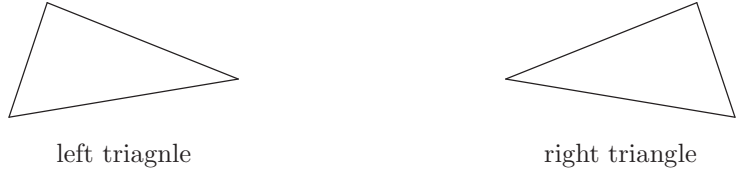
\begin{figure}
\centering
\input{right_and_left_triangle.tex}
\caption{Two kind of triangles.}
\label{fig:right_and_left_triangle}
\end{figure}

We use some notations throughout this paper. For a continued fraction $\cf{n}{}$, let $l_k$ = $\sum_{i=1}^{k}a_i$. We denote the first stacked $a_1 - 1$ right triangles by $\fan_1$ and $a_i$ triangles stacked $i$-th by $\fan_i$ for $2 \le i \le n$. For $i = 2,3,\dots,n$, the bottom edge of $\fan_i$ means the edge shared by $\fan_{i}$ and $\fan_{i-1}$. We denote the $i$-th triangle from the bottom by $\tri_i$. 
\begin{ex}
The ancestral triangle of $7/19 = [2,1,2,2]$ is shown on the left of Figure \ref{fig:example_of_ancestral_triangle_construction} and The ancestral triangle of $3/5 = [1,1,2]$ is shown on the right.
\begin{figure}
\begin{center}
	\input{example_of_ancestral_triangle_construction.tex}
	\caption{Example of ancestral triangles; $\AT{7/19}$ on the left, $\AT{3/5}$ on the right. $\AT{7/19}$ consists of one right triangle $\tri_1$, one left triangle $\tri_2$, two right triangles $\tri_3$ and $\tri_4$, and two left triangles $\tri_5$ and $\tri_6$. $\AT{3/5}$ consists of zero right triangle, one left triangle $\tri_1$, and two right triangles $\tri_2$ and $\tri_3$.}
	\label{fig:example_of_ancestral_triangle_construction}
\end{center}
\end{figure}
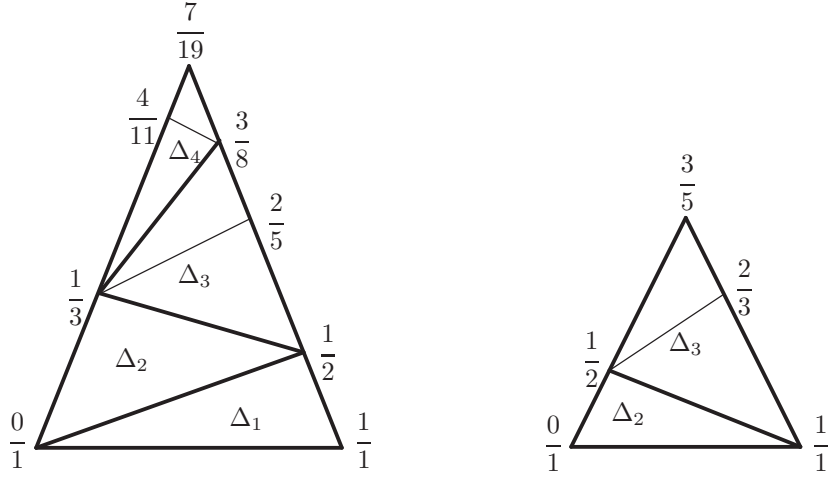
\end{ex}

\begin{prop}
\label{prop:ancestral_triangle}
	 Let $p/q$ is an irreducible fraction between $0$ and $1$. Then
	\begin{enumerate}
		\item The top vertex of $\AT{p/q}$ is labeled $p/q$.
		\item $\AT{(q - p) / q}$ is the mirror image of $\AT{p/q}$.
		\item The vertex of $\AT{(q-p)/p}$ corresponding to the vertex of $\AT{p/q}$ labeled $r/s$ is labeled $(s-r)/s$.
	\end{enumerate}
\end{prop}

\begin{proof}
	$(1)$ follows from a well known recursion: let $p_i / q_i = \cf{i}{}$,
	\[
	\cfrac{p_i}{q_i} = \cfrac{a_ip_{i-1} + p_{i-2}}{a_iq_{i-1} + q_{i - 2}}.
	\]
	We prove $(2)$.The continued fraction expansion of $p/q$ and $(q-p)/q$ are $\cf{n}{}$ with $a_1 > 1$ and $[1,a_1-1,a_2,\dots,a_n]$. The ancestral triangle of $\cf{n}{}$ consists of $a_1 - 1$ right triangles, $a_2$ left triangles, $\dots$, and that of $[1,a_1-1,a_2,\dots,a_n]$ consists of zero right triangle, $a_1 - 1$ left triangles, $a_2$ right triangles, $\dots$. These are mirror images of each other. $(3)$ follows $(2)$ and the fact that if two vertices is labeled $(q-p)/q$ and $(s-r)/s$, then the other vertex is labeled $((q + s) - (p + r))/ (q+s)$.
\end{proof}

\subsection{paths in ancestral triangle}
In the ancestral triangle of $p/q$, we call a sequence of edges $\path = (e_1, e_2,\dots, e_m)$ a \emph{path} if following conditions hold:
\begin{enumerate}
\item The starting point of $e_i$ is the endpoint of $e_{i-1}$.
\item The starting point of $e_1$ is $p/q$ and the endpoint of $e_m$ is $0/1$ or $1/1$.
\item The denominator of the starting point of $e_i$ is greater than the denominator of the endpoint of $e_i$.
\end{enumerate}

Paths divide the ancestral triangle into two parts. For a path $\path$, let $S_\path$ be the set of triangles on the left side of $\path$.

\begin{defi}
\label{def:weight_of_path}
Let $x_i$'s and $y_i$'s be variables. We define the \emph{weight} of a path $\path$, denoted by $\w(\path)$, to be
\begin{equation*}
\w(\path) = \prod_{\tri_i \in S_\path}\w(\tri_i),
\end{equation*}
where
\begin{equation*}
\w(T_1) = \begin{cases}
y_1x_{2} & \text{if $\tri_1$ is a right triangle;}\\
y_1/x_{2} & \text{if $\tri_1$ is a left triangle,}
\end{cases}
\end{equation*}
and for $2 \le i \le l_n - 1$,
\begin{equation*}
\w(\tri_i) = \begin{cases}
y_ix_{i+1}/x_{i-1} & \text{if $\tri_i$ is a right triangle and $\tri_{i-1}$ is a right triangle;}\\
y_ix_{i-1}x_{i+1} & \text{if $\tri_i$ is a right triangle and $\tri_{i-1}$ is a left triangle;}\\
y_i/x_{i-1}x_{i+1} & \text{if $\tri_i$ is a left triangle and $\tri_{i-1}$ is a right triangle;}\\
y_ix_{i-1}/x_{i+1} & \text{if $\tri_i$ is a left triangle and $\tri_{i-1}$ is a left triangle;}
\end{cases}
\end{equation*}
with the convention that $x_{l_n} = 1$.
\end{defi}

\begin{ex}
An example of path and not path is shown in Figure  \ref{fig:example_of_path}. Let $\path$ be a path shown on the left in that figure. The weight of triangles are as follows:
\[
\begin{array}{rclrclrcl}
\w(\tri_1) & = & y_1x_2,&
\w(\tri_2) & = & y_2/x_1x_3,&
\w(\tri_3) & = & y_3x_2x_4,\\
\w(\tri_4) & = & y_4x_5/x_3,&
\w(\tri_5) & = & y_5/x_4x_6,&
\w(\tri_6) & = & y_6x_5.
\end{array}
\]
Since $S_\path = \{\tri_1, \tri_2,\tri_5\}$, the weight of $\path$ is
\[
\w(\path) = \w(\tri_1)\w(\tri_2)\w(\tri_5) = \cfrac{y_1y_2y_5x_2}{x_1x_3x_4x_6}.
\]
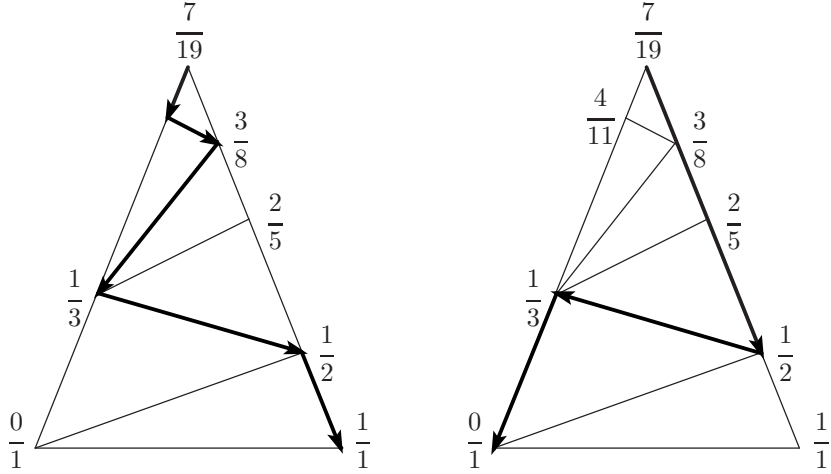
\begin{figure}
	\centering
	\input{example_of_path.tex}
	\caption{Bold arrows on the left is a path and on the right is not a path because the arrow from $1/2$ to $1/3$ does not satisfy the condition $(3)$.}
	\label{fig:example_of_path}
\end{figure}
\end{ex}

\section{Cluster variable}
\label{section:cluster_variable}
In this section, first we recall the definition of cluster variables with principal coefficients. Then we define the cluster variables from the ancestral triangles and show the cluster expansion formula using paths.

\subsection{Cluster variables with principal coefficients}
We will consider cluster algebras with principal coefficients which are defined over tropical semifields.

Let $N$ be a positive integer. Let $\Trop(y_1,y_2,\dots,y_N)$ be an abelian group freely generated by the elements $y_1,y_2\dots,y_N$. We define the multiplication $\cdot$ in $\Trop(y_1,y_2,\dots,y_N)$ as the usual multiplication of polynomials and define the addition $\oplus$ in $\Trop(y_1,y_2,\dots,y_N)$ by
\[
\prod_{j=1}^{N}y_j^{a_j} \oplus \prod_{j=1}^{N}y_j^{b_j} = \prod_{j=1}^{N}y_j^{\min(a_j,b_j)}.
\]
We call $(\Trop(y_1,y_2,\dots,y_N), \oplus, \cdot)$ a tropical semifield. 

We use the notation that $[1,N] = \{1, 2, \dots, N\}$ and $[x]_+ = \max(x, 0)$. Let $B = (b_{ij})$ be an $N \times N$ skew-symmetrizable matrix. Let $\mathbf{y} = (u_1,u_2,\dots,u_N)$ be an $N$-tuple of elements of tropical semifield $(\Trop(y_1,y_2,\dots,y_N), \oplus, \cdot)$ and $\mathbf{x} = (X_1,X_2,\dots,X_N)$ an $N$-tuple of rational functions in $2N$ independent variables $x_1,\dots,x_N,y_1,\dots,y_N$. We call a triple $(\mathbf{x},\mathbf{y},B)$ a \emph{seed}. For $k \in [1,N]$, the \emph{mutation} $\mu_k$ in derection $k$ transforms $(\mathbf{x},\mathbf{y},B)$ into $\mu_k(\mathbf{x},\mathbf{y},B) = (\mathbf{x}',\mathbf{y}',B')$ defined as follows:
\begin{itemize}
\item The entries of $B' = (b_{ij}')$ are given by
\[
b_{ij}' = \begin{cases}
-b_{ij} & \text{if $i = k$ or $j = k$}\\
b_{ij} + \operatorname{sgn}(b_{ik})[b_{ik}b_{kj}]_+ & \text{otherwise.}
\end{cases}
\]

\item $\mathbf{y}' = (u'_1,u'_2,\dots,u'_N)$ is given by
\[
u'_j = \begin{cases}
u_k^{-1} & \text{if $j = k$;}\\
u_ju_k^{[b_{kj}]_+}(u_k \oplus 1)^{-b_{kj}} & \text{if $j \ne k$}.
\end{cases}
\]
\item $\mathbf{x}' = (X'_1,X'_2,\dots,X'_N)$ is given by
\[
X'_j = \begin{cases}
\cfrac{\prod X_i^{[-b_{ik}]_+} + u_k\prod X_i^{[b_{ik}]_+}}{(u_k \oplus 1)X_k} & \text{if $j = k$;}\\
X_j & \text{if $j \ne k$.}
\end{cases}
\]
\end{itemize}

Let $Q$ be a finite quiver without $1$-loops and $2$-cycles with vertex set $[1,N]$. The $N \times N$ skew-symmetric matrix $B^Q = (b_{ij}^Q)$ corresponding to $Q$ is defined by $b_{ij}^Q = Q_{ij} - Q_ {ji}$ where $Q_{ij}$ is the number of arrows from $i$ to $j$ in $Q$. Therefore we regard a triple $(\mathbf{x},\mathbf{y},Q)$ as a seed $(\mathbf{x},\mathbf{y},B^Q)$.

We call a seed $((x_1,x_2,\dots,x_N), (y_1,y_2,\dots,y_N),Q)$ an \emph{initial seed} and denote it by $(\mathbf{x}_0,\mathbf{y}_0,Q_0)$. We say $X$ is a \emph{cluster variable} if there exists a seed $(\mathbf{x},\mathbf{y},Q) = \mu_{i_k} \circ \cdots \circ \mu_{i_1}(\mathbf{x}_0,\mathbf{y}_0,Q_0)$ for some mutations $\mu_{i_1},\dots,\mu_{i_k}$ and $X$ is an element of $\mathbf{x}$. It is known that all cluster variables are elements of $\Z[x_1^{\pm1},\dots,x_N^{\pm1},y_1,\dots,y_N]$. The \emph{$F$-polynomial} is a polynomial in $\Z[y_1,\dots,y_N]$ obtained by substituting $1$ for all $x_i$.

\begin{ex}
\label{exmp:cluster_mutation}
Let $Q_0$ be a quiver $\xymatrix{1 \ar[r] & 2 & \ar[l] 3}$ and $(\mathbf{x}_i,\mathbf{y}_i,Q_i) = \mu_i \circ \cdots \circ \mu_1(\mathbf{x}_0,\mathbf{y}_0,Q_0)$. Seeds $(\mathbf{x}_0,\mathbf{y}_0,Q_0)$, $(\mathbf{x}_1,\mathbf{y}_1,Q_1)$, $(\mathbf{x}_2,\mathbf{y}_2,Q_2)$, and $(\mathbf{x}_3,\mathbf{y}_3,Q_3)$ are as follows:
\begin{center}
\begin{tabular}{|c|c|c|c|}
\hline
i & $\mathbf{x}_i$ & $\mathbf{y}_i$ & $Q_i$\\
\hline
$0$ & $(x_1,x_2,x_3)$ & $(y_1,y_2,y_3)$  & $\xymatrix{1 \ar[r] & 2 & \ar[l] 3}$ \\
\hline
\vrule width 0pt height 20pt depth 12pt $1$ & $\left(\cfrac{x_2 + y_1}{x_1},x_2,x_3\right)$ & $\left(\cfrac{1}{y_1},y_1y_2,y_3\right)$ & $\xymatrix{1 & \ar[l] 2 & \ar[l] 3}$\\
\hline
\vrule width 0pt height 20pt depth 12pt $2$ & $\left(\cfrac{x_2 + y_1}{x_1},\cfrac{x_2 + y_1 + y_1y_2x_1x_3}{x_1x_2},x_3\right)$ & $\left(y_2,\cfrac{1}{y_1y_2},y_3\right)$ & $\xymatrix{1 \ar[r] & 2 \ar[r] & 3 \ar@/^3mm/[ll]}$\\
\hline
\vrule width 0pt height 20pt depth 12pt $3$ & $\left(\cfrac{x_2 + y_1}{x_1},\cfrac{x_2 + y_1 + y_1y_2x_1x_3}{x_1x_2},X_3'\right)$ & $\left(y_2y_3,\cfrac{1}{y_1y_2},\cfrac{1}{y_3}\right)$ & $\xymatrix{1 \ar@/^-3mm/[rr] & 2 & \ar[l] 3 }$\\
\hline
\end{tabular}
\end{center}

where
\[
X_3' = \cfrac{x_{2}^2 + y_1x_2 + y_3x_2 + y_1y_3 + y_1y_2y_3x_1x_3}{x_1x_2x_3}.
\]
\end{ex}
\subsection{Cluster variables from ancestral triangles}
Let $p/q$ be an irreducible fraction between $0$ and $1$, and  $\cf{n}{}$ its continued fraction expansion. Let $N = l_n - 1$. To define the cluster variable from the ancestral triangle $\AT{p/q}$, we construct an initial quiver $Q_0$ as follows:
\begin{itemize}
\item The vertices set of $Q_0$ is $[1,N]$.
\item For $i = 1,2,\dots,N-1$, we put an arrow from $i+1$ to $i$ if $\tri_i$ is a right triangle, and put an arrow from $i$ to $i+1$ if $T_i$ is a left triangle.
\end{itemize}

\begin{ex}
The initial quiver of $\AT{7/19}$ has $6$ vertices. According to Figure \ref{fig:example_of_ancestral_triangle_construction}, $\tri_1$, $\tri_3$, and $\tri_4$ are right triangles and $\tri_2$ and $\tri_5$ are left triangles. Thus, the initial quiver is
\[
\xymatrix{1 & \ar[l] 2 \ar[r] & 3 & \ar[l] 4 & \ar[l] 5 \ar[r] & 6}.
\]
\end{ex}

Let $(\mathbf{x},\mathbf{y},Q) = \mu_N\circ\mu_{N-1}\circ\cdots\circ\mu_2\circ\mu_1(\mathbf{x}_0,\mathbf{y}_0,Q_0)$. We call the $N$-th element of $\mathbf{x}$ the cluster variable from $\AT{p/q}$ and denote it by $\cv_{p/q}$. 

\begin{thm}
\label{thm:cluster_expansion_formula}
Let $\Gamma_{p/q}$ be the set of all paths in $\AT{p/q}$. Then
\begin{equation}
\label{eqn:cluster_expansion_formula}
\cv_{p/q} = \cfrac{D}{x_1x_2\cdots x_N}\sum_{\path \in \Gamma_{p/q}}\w(\path),
\end{equation}
where 
\[
D = \prod_{i=1}^N (\text{the denominator of $\w(\tri_i)$})
\]
In particular, let $F_{p/q}$ be the $F$-polynomial obtained from $\cv_{p/q}$,
\[
F_{p/q} = \sum_{\path \in \Gamma_{p/q}} \prod_{\tri_i \in S_{\path}} y_i.
\]
\end{thm}

\begin{proof}
Before we prove the cluster expansion formula (\ref{eqn:cluster_expansion_formula}), we only consider about quiver. To simplify the proof, let $\tilde{Q}$ be a quiver obtained from $Q$ by adding three vertices $-1$, $0$, and $N + 1$, an arrow from $1$ to $-1$, an arrow from $0$ to $1$, and an arrow from $N$ to $N+1$. We set $x_{-1} = x_0 = x_{N+1} = 1$. Since three additional variables are equal to $1$, these changes do not affect cluster variables. We focus on the vertex $i$ of $\mu_{i-1}\circ\cdots\circ\mu_1(\tilde{Q})$ for $1 \le i \le N$. It is easy to check that three arrows are incident with $i$ as follows:
\begin{center}
\begin{tabular}{|c|c|c|}
\hline
$i$ & from & to\\
\hline
$l_{2k} + 1 \le i \le l_{2k + 1} - 1$ & $i+1$, $i-1$ & $l_{2k} - 1$\\
\hline
$l_{2k+1}$ & $i-1$ & $i+1$, $l_{2k} - 1$\\
\hline
$l_{2k+1} + 1 \le i \le l_{2k+2} - 1$ & $l_{2k+1} - 1$ & $i+1$, $i-1$\\
\hline
$l_{2k}$ & $l_{2k-1} - 1$, $i+1$ & $i - 1$\\
\hline
\end{tabular}
\end{center}
Then, we prove (\ref{eqn:cluster_expansion_formula}) by induction on $i$. If $i = 1$, then the cluster variable is
\[
X_{1}' = \begin{cases}
\cfrac{1 + y_1x_2}{x_1} & \text{if }\xymatrix{1 & \ar[l] 2};\\
\cfrac{y_1 + x_2}{x_1} & \text{if }\xymatrix{1 \ar[r] & 2}.
\end{cases}
\]
On the other hand, the RHS of (\ref{eqn:cluster_expansion_formula}) is
\[
\begin{cases}
\cfrac{1}{x_1}(1 + \w(\tri_1)) = \cfrac{1}{x_1}(1 + y_1x_2) & \text{if $\tri_1$ is a right triangle;}\\
\cfrac{x_2}{x_1}\left(1 + \w(\tri_1)\right) = \cfrac{x_2}{x_1}\left(1 + \cfrac{y_1}{x_2}\right) & \text{if $\tri_1$ is a left triangle;}
\end{cases}
\]
and Theorem \ref{thm:cluster_expansion_formula} holds. Suppose that $i > 1$. Let $p_i/q_i$ be a fraction which has the greatest denominator of labels of vertices of $\tri_i$. Let $\Gamma_{p/q}^{(i)}$ be the set of paths from $p_i/q_i$ and $D_i = \prod_{k = 1}^{i}(\text{the denominator of $\w(\tri_k)$})$. By induction, we may assume
\[
X_{k}' = \cfrac{D_k}{x_1\cdots x_k}\left(\sum_{\path \in \Gamma_{p/q}^{(k)}}\w(\path)\right)
\]
for $k < i$.

\begin{figure}
	\centering
	\scalebox{0.8}{\input{cluster_expansion_formula.tex}}
	\caption{Triangles under $\tri_i$; $l_{2k} + 1 \le i \le l_{2k + 1} - 1$ on the upper left, $i = l_{2k+1}$ on the upper right, $l_{2k+1} + 1 \le i \le l_{2k+2} - 1$ on bottom left, $i = l_{2k}$ on bottom right.}
	\label{fig:cluster_expansion_formula}
\end{figure}
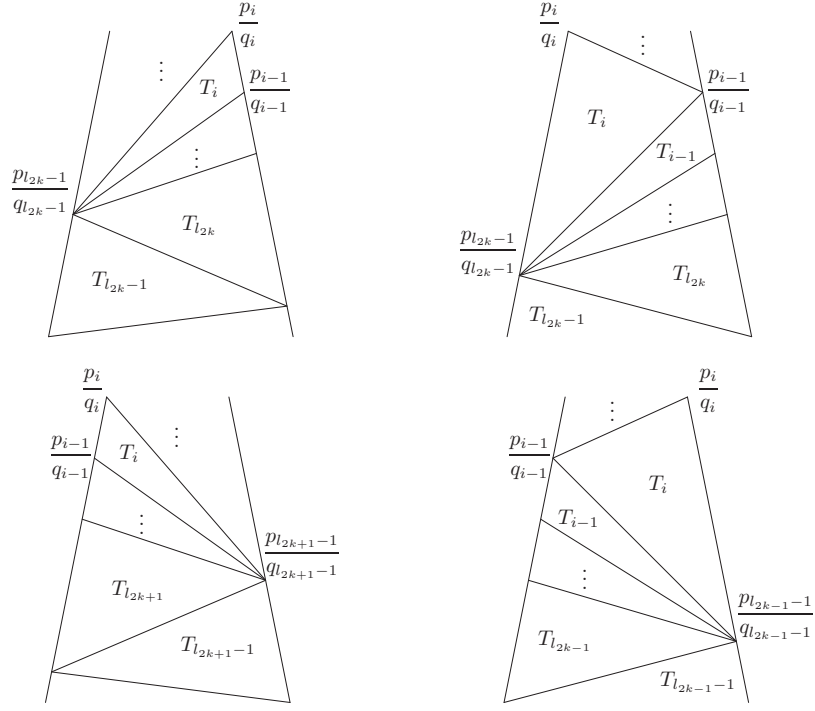

If $l_{2k} + 1 \le i \le l_{2k + 1} - 1$, then the cluster variable is
\[
X_i' = \cfrac{y_ix_{i+1}X_{i-1}' + X_{l_{2k} - 1}}{x_i}.
\]
According to on the upper left of Figure \ref{fig:cluster_expansion_formula}, $\tri_i$ is a right triangle. The left vertex of $\tri_i$ is labeled $p_{l_{2k}-1}/q_{l_{2k}-1}$ and the right vertex is labeled $p_{i-1}/q_{i-1}$. Any $\path \in \Gamma_{p/q}^{(i)}$ is obtained by adding an edge $(p_i/q_i, p_{i-1}/q_{i-1})$ to an element of $\Gamma_{p/q}^{(i-1)}$ or adding an edge $(p_i/q_i, p_{l_{2k}-1}/q_{l_{2k}-1})$ to an element of $\Gamma_{p/q}^{(l_{2k}-1)}$. For a path obtained by the former way, since $T_i$ is on the left side of the path, the weight of the path is the product of the original weight and $\w(\tri_i)$. For a path obtained by the latter way, the weight of path is equal to the original weight. Since $\tri_{l_{2k}-1}$ is a left triangle and $\tri_{l_{2k}}, \tri_{l_{2k} + 1}, \dots, \tri_{i}$ are right triangles, we have $D_i = x_{i-1}D_{i-1} = x_{l_{2k}}x_{l_{2k}+1}\cdots x_{i-1}D_{l_{2k}-1}$. Thus, the RHS of (\ref{eqn:cluster_expansion_formula}) is
\[
\begin{array}{rcl}
\displaystyle
\cfrac{D_i}{x_1x_2\cdots x_i}\sum_{\path \in \Gamma_{p/q}^{(i)}} \w(\path) &=& \cfrac{D_i}{x_1x_2\cdots x_i}\left(\sum_{\path \in \Gamma_{p/q}^{(i-1)}} \w(\tri_i)\w(\path) + \sum_{\path \in \Gamma_{p/q}^{(l_{2k}-1)}}\w(\path)\right)\\
&=& \cfrac{x_{i-1}D_{i-1}y_ix_{i+1}\sum_{\path \in \Gamma_{p/q}^{(i-1)}}\w(\path)}{x_1x_2\cdots x_ix_{i-1}}\\
& & + ~ \cfrac{x_{l_{2k}}x_{l_{2k}+1}\cdots x_{i-1}D_{l_{2k}-1}\sum_{\path \in \Gamma_{p/q}^{(l_{2k}-1)}}\w(\path)}{x_1x_2\cdots x_i}\\
&=& \cfrac{y_ix_{i+1}X_{i-1}'}{x_i} + \cfrac{X_{l_{2k}-1}'}{x_i},
\end{array}
\]
which is equal to $X_{i}'$.

If $i = l_{2k+1}$, then the cluster variable is
\[
X_{i}' = \cfrac{y_iX_{i-1}' + x_{i+1}X_{l_{2k-1}}'}{x_i}.
\]
Since $\tri_i$ and $\tri_{l_{2k}-1}$ are left triangles and $\tri_{l_{2k}},\dots,\tri_{i-1}$ are right triangles, we have
\[
\sum_{\path \in \Gamma_{p/q}^{(i)}} \w(\path) = \sum_{\path \in \Gamma_{p/q}^{(i-1)}} \w(\tri_i)\w(\path) + \sum_{\path \in \Gamma_{p/q}^{(l_{2k}-1)}} \w(\path);\]
\[
D_i = x_{i-1}x_{i+1}D_{i-1} = x_{l_{2k}}x_{l_{2k}+1}\cdots x_{i-2}x_{i-1}x_{i+1}D_{l_{2k}-1}.
\]
Thus, 
\[
\begin{array}{rcl}
\displaystyle
\cfrac{D_i}{x_1x_2\cdots x_i}\sum_{\path \in \Gamma_{p/q}^{(i)}} \w(\path) &=& \cfrac{D_i}{x_1x_2\cdots x_i}\left(\sum_{\path \in \Gamma_{p/q}^{(i-1)}} \w(\tri_i)\w(\path) + \sum_{\path \in \Gamma_{p/q}^{(l_{2k}-1)}} \w(\path)\right)\\
&=& \cfrac{x_{i-1}x_{i+1}D_{i-1}y_i\sum_{\path \in \Gamma_{p/q}^{(i-1)}}\w(\path)}{x_1x_2\cdots x_ix_{i-1}x_{i+1}}\\
& & + ~ \cfrac{x_{l_{2k}}x_{l_{2k}+1}\cdots x_{i-2}x_{i-1}x_{i+1}D_{l_{2k}-1}\sum_{\path \in \Gamma_{p/q}^{(l_{2k}-1)}}\w(\path)}{x_1x_2\cdots x_i}\\
&=& \cfrac{y_iX_{i-1}'}{x_i} + \cfrac{x_{i+1}X_{l_{2k}-1}'}{x_i} ~=~ X_{i}'.
\end{array}
\]

If $l_{2k+1} + 1 \le i \le l_{2k+2} - 1$, then the cluster variable is
\[
X_{i}' = \cfrac{\prod_{k=l_{2k+1}}^{i}y_k \cdot X_{l_{2k+1}-1}' + x_{i+1}X_{i-1}'}{x_i}.
\]
Since $\tri_{l_{2k+1}},\dots,\tri_{i}$ are left triangles and $\tri_{l_{2k+1}-1}$ is a right triangle, we have
\[
\sum_{\path \in \Gamma_{p/q}^{(i)}} \w(\path) = \sum_{\path \in \Gamma_{p/q}^{(l_{2k+1}-1)}} \w(\path)\cdot\prod_{k = l_{2k+1}}^{i}\w(\tri_k) + \sum_{\path \in \Gamma_{p/q}^{(i-1)}} \w(\path);\]
\[
D_i = x_{i+1}D_{i-1} = x_{l_{2k+1}+2}x_{l_{2k+1}+3}\cdots x_{i+1}x_{l_{2k+1}-1}x_{l_{2k+1}+1}D_{l_{2k+1}-1}.
\]
Thus, 
\[
\begin{array}{rl}
&\cfrac{D_i}{x_1x_2\cdots x_i}\sum_{\path \in \Gamma_{p/q}^{(i)}} \w(\path)\\
=& \cfrac{D_i}{x_1x_2\cdots x_i}\left(\sum_{\path \in \Gamma_{p/q}^{(l_{2k+1}-1)}} \w(\path)\cdot\prod_{k = l_{2k+1}}^{i}\w(\tri_k) + \sum_{\path \in \Gamma_{p/q}^{(i-1)}} \w(\path)\right)\\
=& \cfrac{x_{l_{2k+1}+2}\cdots x_{i+1}x_{l_{2k+1}-1}x_{l_{2k+1}+1}D_{l_{2k+1}-1}x_{l_{2k+1}}x_{l_{2k+1}+1}\prod_{k = l_{2k+1}}^{i}y_i\sum_{\path \in \Gamma_{p/q}^{(l_{2k-1}-1)}}\w(\path)}{x_1x_2\cdots x_ix_{i}{x_{i+1}}x_{l_{2k+1}-1}x_{l_{2k+1}+1}}\\
 & + ~ \cfrac{x_{i+1}D_{i-1}\sum_{\path \in \Gamma_{p/q}^{(i-1)}}\w(\path)}{x_1x_2\cdots x_i}\\
=& \cfrac{\prod_{k = l_{2k+1}}^{i}y_k\cdot X_{l_{2k+1}-1}'}{x_i} + \cfrac{x_{i+1}X_{i-1}'}{x_i} ~=~ X_{i}'.
\end{array}
\]

If $i = l_{2k}$, then the cluster variable is
\[
X_{i}' = \cfrac{\prod_{k=l_{2k-1}}^{i}y_k \cdot x_{i+1}X_{l_{2k-1}-1}' + X_{i-1}'}{x_i}.
\]
Since $\tri_{l_{2k-1}},\dots,\tri_{i-1}$ are left triangles and $\tri_{l_{2k-1}-1}$ and $\tri_{i}$ are right triangles, we have
\[
\sum_{\path \in \Gamma_{p/q}^{(i)}} \w(\path) = \sum_{\path \in \Gamma_{p/q}^{(l_{2k-1}-1)}} \w(\path)\cdot\prod_{k = l_{2k-1}}^{i}\w(\tri_k) + \sum_{\path \in \Gamma_{p/q}^{(i-1)}} \w(\path);\]
\[
D_i = D_{i-1} = x_{l_{2k-1}+2}x_{l_{2k-1}+3}\cdots x_{i}x_{l_{2k-1}-1}x_{l_{2k-1}+1}D_{l_{2k-1}-1}.
\]
Thus, 
\[
\begin{array}{rl}
&\cfrac{D_i}{x_1x_2\cdots x_i}\sum_{\path \in \Gamma_{p/q}^{(i)}} \w(\path)\\
=& \cfrac{D_i}{x_1x_2\cdots x_i}\left(\sum_{\path \in \Gamma_{p/q}^{(l_{2k-1}-1)}} \w(\path)\cdot\prod_{k = l_{2k-1}}^{i}\w(\tri_k) + \sum_{\path \in \Gamma_{p/q}^{(i-1)}} \w(\path)\right)\\
=& \cfrac{x_{l_{2k-1}+2}\cdots x_{i}x_{l_{2k-1}-1}x_{l_{2k-1}+1}D_{l_{2k-1}-1}x_{l_{2k-1}}x_{l_{2k-1}+1}x_{i-1}x_{i+1} \prod_{k = l_{2k+1}}^{i}y_i\sum_{\path \in \Gamma_{p/q}^{(l_{2k-1}-1)}}\w(\path)}{x_1x_2\cdots x_ix_{i-1}x_{i}x_{l_{2k-1}-1}x_{l_{2k-1}+1}}\\
& + ~ \cfrac{D_{i-1}\sum_{\path \in \Gamma_{p/q}^{(i-1)}}\w(\path)}{x_1x_2\cdots x_i}\\
=& \cfrac{\prod_{k = l_{2k-1}}^{i}y_k\cdot x_{i+1}X_{l_{2k+1}-1}'}{x_i} + \cfrac{x_{i+1}X_{i-1}'}{x_i} ~=~ X_{i}'.
\end{array}
\]

This complete the induction. We substitute $N$ for $k$ in the assumption:
\[
X_{N}' = \cfrac{D_N}{x_1\cdots x_N}\left(\sum_{\path \in \Gamma_{p/q}^{(N)}}\w(\path)\right).
\]
By definition, $X_{N}' = X_{p/q}$, $D_N = D$, and $\Gamma_{p/q}^{(N)} = \Gamma_{p/q}$. Therefore (\ref{eqn:cluster_expansion_formula}) holds.
\end{proof}

\begin{rem}
	In \cite{LS}, they denote cluster variable corresponding to continued fraction $\cf{n}{}$ by $x\cf{n}{}$ and $F$-polynomial obtained from $x\cf{n}{}$ by $F\cf{n}{}$. This correspondence consists of two parts, correspondence between cluster variables from triangulations and snake graphs \cite{MS} and correspondence between snake graphs and continued fractions \cite{CS}. We can regard $\AT{p/q}$ as the triangulation of $(N+3)$-gon by adding a vertex $v$ and two edges between $v$ and $0/1$ and between $v$ and $1/1$ and the quiver correponding to the triangulation is the same as $Q_0$ defined above. Since the snake graph corresponding to cluster variable of that corresponds to $q/p$, we can regard $\cv_{p/q}$ as $x\cf{n}{}$ and $F_{p/q}$ as $F\cf{n}{}$ where $p/q = \cf{n}{}$. Note that we and they use the same notation to denote continued fractions, but their definition is the inverse of ours.
\end{rem}

\begin{ex}
The ancestral triangle $\AT{3/5}$ is shown on the right of Figure \ref{fig:example_of_ancestral_triangle_construction}. The weights of triangles are 
$\w(\tri_1) = y_1/x_2$, $\w(\tri_2) = y_2x_1x_3$, and $\w(\tri_3) = y_3/x_2$. Paths in $\AT{p/q}$ and their weight are as follows:
\begin{center}
\begin{tabular}{|c|c|c|}
\hline
$\path$ & $S_\path$ & $\w(\path)$\\
\hline
\vrule width 0pt height 20pt depth 12pt $\cfrac{3}{5} \rightarrow \cfrac{1}{2} \rightarrow \cfrac{0}{1}$ & $\emptyset$ & $1$\\
\hline
\vrule width 0pt height 20pt depth 12pt$\cfrac{3}{5} \rightarrow \cfrac{1}{2} \rightarrow \cfrac{1}{1}$ & $\{\tri_1\}$ & $\cfrac{y_1}{x_2}$\\
\hline
\vrule width 0pt height 20pt depth 12pt$\cfrac{3}{5} \rightarrow \cfrac{2}{3} \rightarrow \cfrac{1}{2} \rightarrow \cfrac{0}{1}$ & $\{\tri_3\}$ & $\cfrac{y_3}{x_2}$\\
\hline
\vrule width 0pt height 20pt depth 12pt$\cfrac{3}{5} \rightarrow \cfrac{2}{3} \rightarrow \cfrac{1}{2} \rightarrow \cfrac{1}{1}$ & $\{\tri_1, \tri_3\}$ & $\cfrac{y_1y_3}{x_{2}^2}$\\
\hline
\vrule width 0pt height 20pt depth 12pt$\cfrac{3}{5} \rightarrow \cfrac{2}{3} \rightarrow \cfrac{1}{1}$ & $\{\tri_1, \tri_2, \tri_3\}$ & $\cfrac{y_1y_2y_3x_1x_3}{x_{2}^2}$\\
\hline
\end{tabular}
\end{center}
By Theorem \ref{thm:cluster_expansion_formula}, we have
\[
\begin{array}{rcl}

\vrule width 0pt height 20pt depth 12pt X_{3/5} &=& \cfrac{x_{2}^2}{x_1x_2x_3}\left(1 + \cfrac{y_1}{x_2} + \cfrac{y_3}{x_2} + \cfrac{y_1y_3}{x_{2}^2} + \cfrac{y_1y_2y_3x_1x_3}{x_{2}^2}\right)\\
\vrule width 0pt height 20pt depth 12pt&=& \cfrac{x_{2}^2}{x_1x_2x_3} + \cfrac{y_1x_{2}}{x_1x_2x_3} + \cfrac{y_3x_{2}}{x_1x_2x_3} + \cfrac{y_1y_3}{x_1x_2x_3} + \cfrac{y_1y_2y_3x_1x_3}{x_1x_2x_3}\\
\vrule width 0pt height 20pt depth 12pt&=& \cfrac{x_{2}^2 + y_1x_2 + y_3x_2 + y_1y_3 + y_1y_2y_3x_1x_3}{x_1x_2x_3}.\\
\end{array}
\]

On the other hand, the initial quiver of $\AT{3/5}$ is
\[
Q_0 = \xymatrix{1 \ar[r] & 2 & \ar[l] 3}
\]
and $X_{3/5}$ is the third element of $\mathbf{x}$ in the seed $(\mathbf{x},\mathbf{y},Q) = \mu_3\circ\mu_2\circ\mu_1(\mathbf{x}_0,\mathbf{y}_0,Q_0)$. Since the initial seed is same as Example \ref{exmp:cluster_mutation}, $(\mathbf{x},\mathbf{y},Q)$ is equal to $(\mathbf{x}_3,\mathbf{y}_3,Q_3)$ in Example \ref{exmp:cluster_mutation}. In particular,
\[
X_{3/5} = X_3' = \cfrac{x_{2}^2 + y_1x_2 + y_3x_2 + y_1y_3 + y_1y_2y_3x_1x_3}{x_1x_2x_3}.
\]
\end{ex}

\begin{cor}
\label{cor:mirror_F-polynomial}
$F$-polynomials $\Fp{p/q}$ and $\Fp{(q-p)/q}$ satisfy the following equation:
\[
\Fp{(q-p)/q} = y_1y_2\cdots y_N \overline{\Fp{p/q}}
\]
where $\overline{\Fp{p/q}}$ is obtained by substituting $y_i^{-1}$ for $y_i$ for $i = 1,2,\dots,N$.
\end{cor}

\begin{proof}
For a path $\path \in \Gamma_{p/q}$, let $\overline{\path} \in \Gamma_{(q-p)/q}$ be the mirror image of $\path$, that is, when $\path$ goes an edge $(r/s,t/u)$ of $\AT{p/q}$, $\overline{\path}$ goes the edge $((s-r)/s, (u-t)/u)$ of $\AT{(q-p)/q}$. Since this correspondence is bijective, a triangle is on the left side of ${\overline{\path}}$ in $\AT{(q-p)/q}$ if and only if the corresponding triangle is on the right side of ${{\path}}$ in $\AT{p/q}$. Thus we have $S_{\overline{\path}} = \{\tri_1,\tri_2,\dots,\tri_N\} \backslash S_\path$ and
\[
\begin{array}{rcl}
\Fp{(q-p)/q} &=& \sum_{{\overline{\path}} \in \Gamma_{(q-p)/q}} \prod_{\tri_i \in S_{\overline{\path}}} y_i\\
&=& \sum_{{\path} \in \Gamma_{p/q}} \prod_{\tri_i \in \{\tri_1,\tri_2,\dots,\tri_N\} \backslash S_\path} y_i\\
&=&  \sum_{{\path} \in \Gamma_{p/q}} y_1y_2\cdots y_N \prod_{\tri_i \in S_\path} y_i^{-1}\\
&=& y_1y_2\cdots y_N \overline{\Fp{p/q}}.
\end{array}
\]
\end{proof}

\begin{cor}
\label{cor:F-polynomial_recursion}
Let $\Fp{\cf{n}{}} = F_{p/q}$ if $p/q = \cf{n}{}$. $F$-polynomials satisfy the following equations:
\begin{equation}
\label{eqn:F-polynomial_identity}
\Fp{\cf{n}{-1,1}} = \Fp{\cf{n}{}};
\end{equation}
and if $a_n \ge 2$, $\Fp{\cf{n}{}} = $
\begin{equation}
\label{eqn:F-polynomial_recursion}
\begin{cases}
\Fp{\cf{n-1}{}} + y_N\Fp{\cf{n}{-1}} & \text{if $n$ is odd;}\\
\displaystyle \left(\prod_{i = l_{n-1}}^{N}y_i\right)\Fp{\cf{n-1}{}} + \Fp{\cf{n}{-1}} & \text{if $n$ is even.}
\end{cases}
\end{equation}
\end{cor}
\begin{proof}
(\ref{eqn:F-polynomial_identity}) is clear. We prove (\ref{eqn:F-polynomial_recursion}). Let $\alpha = \cf{n}{-1}$ and $\beta = \cf{n-1}{}$. All $\path \in \Gamma_{p/q}$ has the form $((p/q,\alpha),\path')$ or $((p/q,\beta),\path'')$ where $\path'$ is a path from $\alpha$ and $\path''$ is a path from $\beta$. If $n$ is odd,  triangles on the left side of $\gamma$ are same as those of $\gamma'$ in the former case and are those of $\gamma''$ and a triangle $\tri_N$ in the latter case. If $n$ is even,  triangles on the left side of $\gamma$ are same as those of $\gamma''$ in the latter case and are those of $\gamma'$ and triangles $\tri_{l_{n-1}},\tri_{l_{n-1}+1},\dots,\tri_{N}$ in the former case. 
\end{proof}

\begin{cor}
\label{cor:F-polynomial_2_recursion}
With the notation in Corollary \ref{cor:F-polynomial_recursion}, $F$-polynomials satisfy the following recursions with the convention that $\Fp{\cf{0}{}} = 1$ and $\Fp{\cf{n-1}{,0}} = \Fp{\cf{n-2}{}}$:

If $n$ is odd, then 
\begin{equation}
\label{eqn:F-polynomial_rec_odd2_n12}
\Fp{\cf{n-1}{,2}} = (1 + y_N)\Fp{\cf{n-1}{}} + \left(\prod_{i=l_{n-2}}^{N}y_i\right)\Fp{\cf{n-2}{}};
\end{equation}
\begin{equation}
\label{eqn:F-polynomial_rec_odd2_12}
\Fp{\cf{n-1}{,2}} = (1 + y_N)\Fp{\cf{n-1}{,1}} - \left(\prod_{i=l_{n-2}}^{N-1}y_i\right)\Fp{\cf{n-2}{}};
\end{equation}

and for $a_i \ge 3$, 
\begin{align}
\label{eqn:F-polynomial_rec_odd_12}
\Fp{\cf{n}{}} &= (1 + y_N)\Fp{\cf{n}{-1}} - y_{N-1}\Fp{\cf{n}{-2}};\\
\label{eqn:F-polynomial_rec_odd_2n1}
\Fp{\cf{n}{}} &= (1 + y_N)\Fp{\cf{n-1}{}} + y_{N-1}y_N\Fp{\cf{n}{-2}}.
\end{align}

If $n$ is even and $a_n \ge 2$, then 
\begin{align}
\label{eqn:F-polynomial_rec_even_12}
\Fp{\cf{n}{}} &= (1 + y_N)\Fp{\cf{n}{-1}} - y_N\Fp{\cf{n}{-2}};\\
\label{eqn:F-polynomial_rec_even_2n1}
\Fp{\cf{n}{}} &= \displaystyle\left(\prod_{i=l_{n-1}}^{N-1}y_i\right)(1 + y_N)\Fp{\cf{n-1}{}} + \Fp{\cf{n}{-2}}.
\end{align}
\end{cor}

\begin{proof}
Applying Corollary \ref{cor:F-polynomial_recursion} twice yields these recursions.
\end{proof}

\section{Alexander polynomial}
\label{section:alexander_polynomial}
\subsection{Two-bridge links}
In this subsection, we review a notation of two-bridge links. We denote a $2$-strand braid with $m$ crossings by
\[
\raisebox{-3mm}{\begin{xy}
(4,4)*{m}
\ar@{-} (0,0);(8,0)
\ar@{-} (8,0);(8,8)
\ar@{-} (8,8);(0,8)
\ar@{-} (0,8);(0,0)
\end{xy}} ~=~ 
\begin{cases}
\overbrace{
	\begin{xy}
	(0,5)="A",(5,5)="B",(10,5)="C",(15,5)="D",(20,5)="E",
	(0,0)="F",(5,0)="G",(10,0)="H",(15,0)="I",(20,0)="J",
	(12.5,2.5)*{\cdots},
	\ar@{-} "A";"G"
	\ar@{-} "B";"H"
	\ar@{-} "D";"J"
	\ar@{-} "F";"B"|!{"A";"G"}\hole
	\ar@{-} "G";"C"|!{"B";"H"}\hole
	\ar@{-} "I";"E"|!{"D";"J"}\hole
	\end{xy}
}^{m} & \raisebox{7pt}{\text{if $m > 0$}}\\
\overbrace{
	\begin{xy}
	(0,5)="A",(5,5)="B",(10,5)="C",(15,5)="D",(20,5)="E",
	(0,0)="F",(5,0)="G",(10,0)="H",(15,0)="I",(20,0)="J",
	(12.5,2.5)*{\cdots},
	\ar@{-} "F";"B"
	\ar@{-} "G";"C"
	\ar@{-} "I";"E"
	\ar@{-} "A";"G"|!{"F";"B"}\hole
	\ar@{-} "B";"H"|!{"G";"C"}\hole
	\ar@{-} "D";"J"|!{"I";"E"}\hole
	\end{xy}
}^{m} & \raisebox{7pt}{\text{if $m < 0$}}
\end{cases}
\]

For a continued fraction $\cf{n}{}$, a link who has the link diagram of the  form in Figure \ref{fig:two-brigde_knot} is called a \emph{two-bridge link}. We denote it by $\tb(\cf{n}{})$. It is also denoted by $\tb(p/q)$ if $\cf{n}{}$ is a continued fraction expansion of $p/q$. $\tb(p/q)$ is well-defined because of Proposition \ref{prop:two-bridge_knot_prop} (3). 
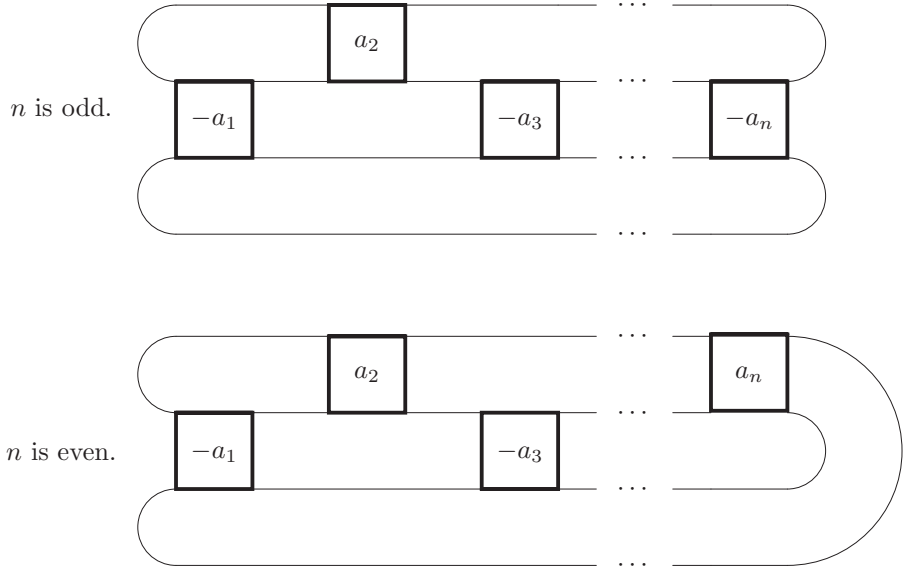
\begin{figure}
	\centering
	\input{two-bridge_knot.tex}
	\caption{Two-bridge link diagrams. Each box represents a $2$-strand braid.}
	\label{fig:two-brigde_knot}
\end{figure}

We list a few facts on two-bridge links.
\begin{prop}
\label{prop:two-bridge_knot_prop}
\begin{enumerate}
\item $\tb((q-p)/q)$ is the mirror image of $\tb(p/q)$.
\item $\tb(p/q)$ is a knot if $q$ is odd and a $2$-component link if $q$ is even. 
\item $\tb(\cf{n}{})$ and $\tb(\cf{n}{-1,1})$ are equivalent.
\end{enumerate}
\end{prop}

\begin{ex}
The two-bridge link $\tb(7/19) = \tb([2,1,2,2])$ is shown in Figure \ref{fig:example_of_two-brigde_knot}.
\begin{figure}
	\centering
	\input{example_of_two-bridge_knot.tex}
	\caption{The link diagram of the two-bridge knot $\tb(7/19)$.}
	\label{fig:example_of_two-brigde_knot}
\end{figure}
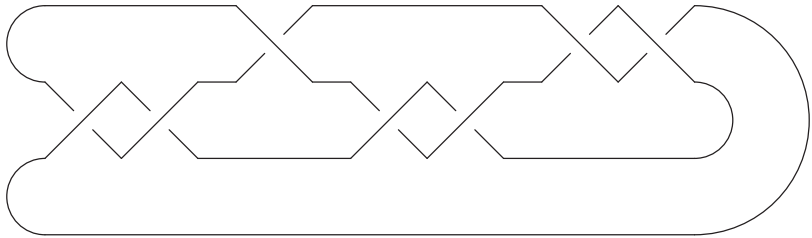
\end{ex}

\subsection{Orientation}
We say that a link is \emph{oriented} if each component of the link has been given an orientation. For an oriented link, we define a sign on each crossing as
\begin{center}
\raisebox{5pt}{a sign on }
\begin{xy}
(0,0)="A",(5,0)="B",(5,5)="C",(0,5)="D",
\ar "A";"C"
\ar "B";"D"|!{"A";"C"}\hole
\end{xy}
\raisebox{5pt}{ is $+1$ and a sign on }
\begin{xy}
	(0,0)="A",(5,0)="B",(5,5)="C",(0,5)="D",
	\ar "B";"D"
	\ar "A";"C"|!{"B";"D"}\hole
\end{xy}
\raisebox{5pt}{ is $-1$. }
\end{center}

We fix orientations of two-bridge links throughout this paper as follows: the orientation on the leftmost crossing in the $2$-strand braid corresponding to $-a_1$ is
\begin{xy}
	(0,0)="A",(5,0)="B",(5,5)="C",(0,5)="D",
	\ar "A";"C"
	\ar@{-} "B";"D"|!{"A";"C"}\hole
\end{xy}
if the two-bridge link is a knot and 
\begin{xy}
	(0,0)="A",(5,0)="B",(5,5)="C",(0,5)="D",
	\ar "A";"C"
	\ar "B";"D"|!{"A";"C"}\hole
\end{xy}
if the two-bridge link is a $2$-component link. From now on, $\tb(\cf{n}{})$ and $\tb(p/q)$ denote the two-bridge link oriented as above. The crossings in each $2$-strand braid have the same sign. We define the sign of $2$-strand braid to be the sign of a crossing in it. 

\begin{ex}
The oriented two-bridge link $\tb(7/19)$ is shown in Figure \ref{fig:example_of_oriented_two-brigde_knot}.

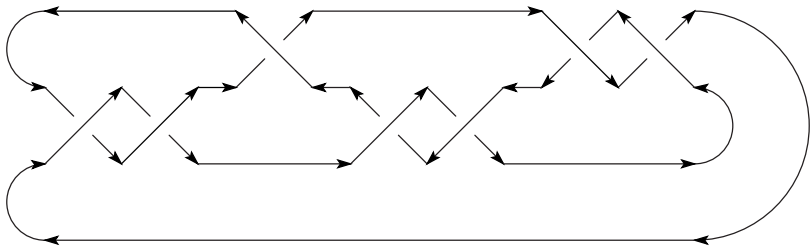
\begin{figure}
	\centering
	\input{example_of_oriented_two-bridge_knot.tex}
	\caption{The link diagram of the oriented two-bridge knot $\tb(7/19)$. The sign of $2$-strand braid corresponding to $-a_1$ is $-1$, $a_2$ is $-1$, $-a_3$ is $+1$ and $a_4$ is $-1$}
	\label{fig:example_of_oriented_two-brigde_knot}
\end{figure}
\end{ex}

%

We will show the theorem to compute signs of $2$-strand braids using the corresponding ancestral triangle. To state the theorem, we define the \emph{Seifert path} which uniquely exists in each ancestral triangle. We write $a/b \equiv c/d \mod 2$ if $a \equiv c \mod 2$ and $b \equiv d \mod 2$. We take the $\mod 2$ values in all labels $r/s$ of $\AT{p/q}$. Therefore three vertices of each triangle are labeled $0/1$, $1/1$, and $1/0$. The Seifert path is a path in $\AT{p/q}$ which satisfies following conditions:
\begin{enumerate}
\item It does not go along two sides of the same triangle.
\item If $p/q \equiv 1/1 \mod 2$, then the path use only edges whose one endpoint is $1/1$ and the other is $1/0$. Otherwise the path use only edges whose one endpoint is $0/1$ and the other is $1/0$.
\end{enumerate}
It is clear that all triangles in $\fan_i$ are on the right side or the left side of the Seifert path because of the condition $(1)$.
\begin{rem}
In \cite{HO}, paths satisfying condition (1) are called edgepaths. Hatcher and Oertel showed that edgepaths correspond to essential surfaces of $\tb(p/q)$. The Seifert path corresponds to Seifert surface.
\end{rem}

\begin{thm}
\label{thm:sign_of_crossing}
For $\AT{p/q}$, we define $t_i$ by following recursion:
\[
\begin{array}{l}

t_1 = \begin{cases}
+1 & \text{if $p/q \equiv 1/0 \text{ or } 0/1 \mod 2$;}\\
-1 & \text{if $p/q \equiv 1/1 \mod 2$.}
\end{cases}\\
t_{i} = \begin{cases}
-t_{i-1} & \text{if the bottom edge of $\fan_i$ is in the Seifert path;}\\
t_{i-1} & \text{otherwise.}
\end{cases}
\end{array}
\]
Then, the sign of $i$-th $2$-strand braid in $\tb(p/q)$ is $t_i$.
\end{thm}

\begin{proof}
Throughout the proof, $t_{i}^{(p/q)}$ denotes $t_i$ determined according to $\AT{p/q}$ and $\epsilon_{i}^{(p/q)}$ denotes the sign of $i$-th $2$-strand braid in $\tb(p/q)$. 

Assume that $p/q = \cf{n}{}$ and $a_i \ge 3$ for some $i$. Let $r/s = [a_1,\dots,a_{i-1},a_i - 2, a_{i+1},\dots,a_n]$. Both of $\fan_k$ of $\AT{p/q}$ and $\fan_k$ of $\AT{r/s}$ are on the left side or right side of each Seifert path for all $k$ because of the condition $(1)$ of the Seifert path. Thus we have $t_k^{(p/q)} = t_k^{(r/s)}$ for all $k$. On the other hand, since $\tb(r/s)$ is obtained by crossing change at the rightmost crossing in $i$-th $2$-strand braid of $\tb(p/q)$, we have $\epsilon_k^{(p/q)} = \epsilon_k^{(r/s)}$. Therefore it is sufficient to prove the theorem when  $a_i = 1 \text{ or } 2$ for all $i$.

At first we prove the theorem for $p/q \equiv 1/1 \text{ or } 0/1$ by induction on the length of continued fraction expansion. We also prove the following formula:
\begin{equation}
\label{eqn:sign_formula}
t_n = \begin{cases}
1 & \text{if the left vertex of the top triangle is labeled $1/0$ modulo $2$;}\\
-1 & \text{if the right vertex of the top triangle is labeled $1/0$ modulo $2$.}
\end{cases}
\end{equation}
If $n = 1$, then $a_1 = 1$ because $[a_1] = 1/a_1 \equiv 1/1 \mod 2$ and we have $t_1^{(1/1)} = -1$. Since $\tb(1/1)$ is a knot, we have $\epsilon_i^{(1/1)} = -1$ by definition. 

If $n = 2$, then since $[a_1,a_2] = a_2/(a_1a_2 + 1) \equiv 1/1 \text{ or } 0/1 \mod 2$, at least one of $a_1$ and $a_2$ is even. Therefore we need only consider three cases: $p/q = [2,2] \text{, } [2,1] \text{, or } [1,2]$. In the case when $p/q = [2,1]$, it is proved above because $[2,1] = [3]$. If $p/q = [2,2] \text{ or } [1,2]$, then $\AT{p/q}$ and $\tb(p/q)$ is shown in Figure \ref{fig:proof_of_sign_l2}. According to that, we have $t_1^{([2,2])} = -1$, $t_2^{([2,2])} = +1$, $t_1^{([1,2])} = t_2^{([1,2])} = -1$, $\epsilon_1^{([2,2])} = -1$, $\epsilon_2^{([2,2])} = +1$, and $\epsilon_1^{([1,2])} = \epsilon_2^{([1,2])} = -1$. Therefore the theorem and the formula hold.

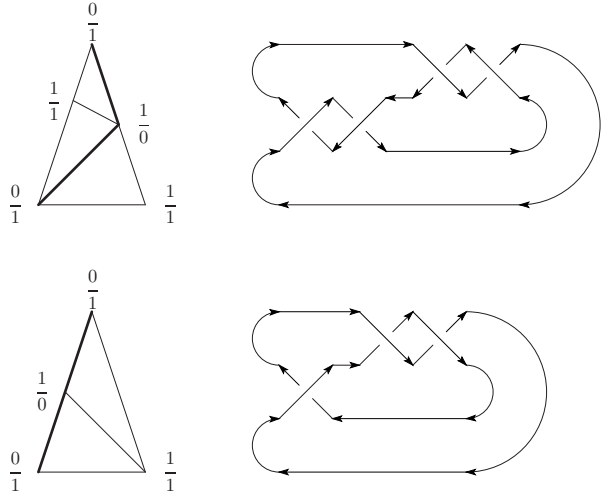
\begin{figure}
	\centering
	\scalebox{0.7}{\input{proof_of_sign_l2.tex}}
	\caption{$\AT{p/q}$ and $\tb(p/q)$; $p/q = [2,2]$ on the top; $p/q = [1,2]$ on the bottom; Each label is replaced by its modulo $2$. The bold line in each $\AT{\alpha}$ represents the Seifert path.}
	\label{fig:proof_of_sign_l2}
\end{figure}

Now suppose that $n > 2$. Since $\cf{n}{-1, 1} = \cf{n}$, we need only consider the case when $a_n = 2$. Let $\alpha = \cf{n}{}$, $\beta = \cf{n-1}{}$, and $\gamma = \cf{n-2}{}$. Then, we have $\alpha \equiv \gamma \mod 2$ and $\alpha \not\equiv \beta \mod 2$.

If $n$ is odd and $\beta \equiv 1/0$, then $\AT{\alpha}$, $\tb(\gamma)$, and $\tb(\alpha)$ are shown in Figure \ref{fig:proof_of_sign_odd_10}. The Seifert path of $\AT{\alpha}$ is obtained by adding two edge $(\alpha,\beta)$ and $(\beta,\gamma)$ to the Seifert path of $\AT{\gamma}$. Thus we have $t_k^{(\alpha)} = t_k^{(\gamma)}$ for $1 \le k \le n-2$ and
\[
\begin{array}{cl}
t_{n-2}^{(\alpha)} = t_{n-1}^{(\alpha)} = -t_n^{(\alpha)} & \text{if $a_{n-1}$ is odd;}\\
t_{n-2}^{(\alpha)} = -t_{n-1}^{(\alpha)} = t_n^{(\alpha)} & \text{if $a_{n-1}$ is even.}
\end{array}
\]
Since $t_{n-2}^{(\gamma)} = t_{n-2}^{(\alpha)}$ is $-1$ if $a_{n-1}$ is odd and $+1$ if $a_{n-1}$ is even by (\ref{eqn:sign_formula}), we have
\[
\begin{array}{cccl}
t_{n-2}^{(\alpha)} = -1, & t_{n-1}^{(\alpha)} = -1, & t_n^{(\alpha)} = +1 & \text{if $a_{n-1}$ is odd;}\\
t_{n-2}^{(\alpha)} = +1, & t_{n-1}^{(\alpha)} = -1, & t_n^{(\alpha)} = +1 & \text{if $a_{n-1}$ is even}
\end{array}
\]
and (\ref{eqn:sign_formula}) holds because the vertex labeled $1/0$ is on the left of the top triangle of $\AT{\alpha}$. On the other hand, by induction and (\ref{eqn:sign_formula}), the sign of the rightmost crossing of $\AT{\gamma}$ is $-1$ if $a_{n-1}$ is odd and $+1$ if $a_{n-1}$ is even. Then, the rightmost crossing is
\begin{xy}
	(0,0)="A",(5,0)="B",(5,5)="C",(0,5)="D",
	\ar "A";"C"
	\ar "D";"B"|!{"A";"C"}\hole
\end{xy}
 or 
\begin{xy}
 	(0,0)="A",(5,0)="B",(5,5)="C",(0,5)="D",
 	\ar "C";"A"
 	\ar "B";"D"|!{"A";"C"}\hole
\end{xy} 
if $a_{n-1}$ is odd and 
\begin{xy}
	(0,0)="A",(5,0)="B",(5,5)="C",(0,5)="D",
	\ar "A";"C"
	\ar "B";"D"|!{"A";"C"}\hole
\end{xy}
 or
\begin{xy}
	(0,0)="A",(5,0)="B",(5,5)="C",(0,5)="D",
	\ar "C";"A"
	\ar "D";"B"|!{"A";"C"}\hole
\end{xy}
if $a_{n-1}$ is even. Since the sign of each crossing does not depend on orientation, reversing orientation if necessary, without loss of generality we can assume the rightmost crossing is
\begin{xy}
	(0,0)="A",(5,0)="B",(5,5)="C",(0,5)="D",
	\ar "A";"C"
	\ar "D";"B"|!{"A";"C"}\hole
\end{xy}
if $a_{n-1}$ is odd and 
\begin{xy}
	(0,0)="A",(5,0)="B",(5,5)="C",(0,5)="D",
	\ar "C";"A"
	\ar "D";"B"|!{"A";"C"}\hole
\end{xy}
if $a_{n-1}$ is even. According to Figure \ref{fig:proof_of_sign_odd_10}, the common parts of $\tb(\gamma)$ and $\tb(\alpha)$ have same orientation. Thus, we have $\epsilon_k^{(\alpha)} = \epsilon_k^{(\gamma)}$ for $1 \le k \le n-2$ and 
\[
\begin{array}{ccl}
\epsilon_{n-1}^{(\alpha)} = -1 & \epsilon_n^{(\alpha)} = +1 & \text{if $a_{n-1}$ is odd;}\\
\epsilon_{n-1}^{(\alpha)} = -1 & \epsilon_n^{(\alpha)} = +1 & \text{if $a_{n-1}$ is even.}
\end{array}
\]
Therefore the theorem holds.
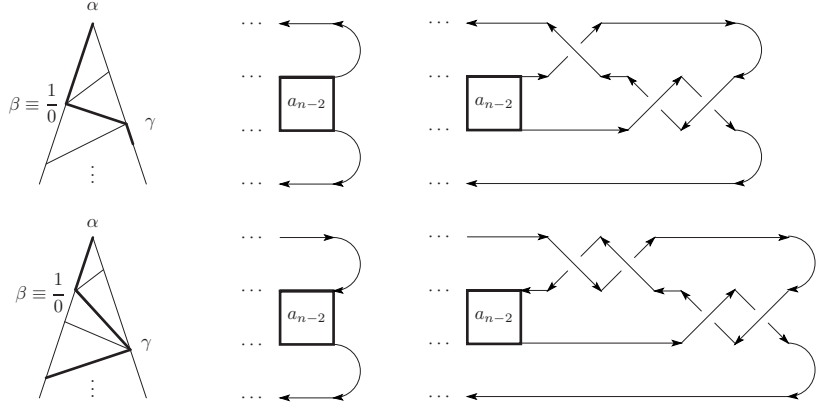
\begin{figure}
	\centering
	\scalebox{0.7}{\input{proof_of_sign_odd_10.tex}}
	\caption{$\AT{\alpha}$, $\tb(\gamma)$, and $\tb(\alpha)$ when $n$ is odd and $\beta \equiv 1/0$; $a_{n-1}$ is odd on the top; $a_{n-1}$ is even on the bottom; The bold line in each $\AT{\alpha}$ is a part of the Seifert path.}
	\label{fig:proof_of_sign_odd_10}
\end{figure}

If $n$ is odd and $\beta \not\equiv 1/0$, then $\AT{\alpha}$, $\tb(\gamma)$, and $\tb(\alpha)$ are shown in Figure \ref{fig:proof_of_sign_odd_not10}. In the same manner we can see
\[
\begin{array}{cccl}
t_{n-2}^{(\gamma)} = t_{n-2}^{(\alpha)} = +1, & t_{n-1}^{(\alpha)} = -1, & t_n^{(\alpha)} = -1 & \text{if $a_{n-1}$ is odd;}\\
t_{n-2}^{(\gamma)} = t_{n-2}^{(\alpha)} = -1, & t_{n-1}^{(\alpha)} = -1, & t_n^{(\alpha)} = -1 & \text{if $a_{n-1}$ is even,}\\
&&&\\
\epsilon_{n-2}^{(\gamma)} = \epsilon_{n-2}^{(\alpha)} = +1 & \epsilon_{n-1}^{(\alpha)} = -1 & \epsilon_n^{(\alpha)} = -1 & \text{if $a_{n-1}$ is odd;}\\
\epsilon_{n-2}^{(\gamma)} = \epsilon_{n-1}^{(\alpha)} = -1 & \epsilon_{n-1}^{(\alpha)} = -1 & \epsilon_n^{(\alpha)} = -1 & \text{if $a_{n-1}$ is even.}
\end{array}
\]
\begin{figure}
	\centering
	\scalebox{0.7}{\input{proof_of_sign_odd_not10.tex}}
	\caption{$\AT{\alpha}$, $\tb(\gamma)$, and $\tb(\alpha)$ when $n$ is odd and $\beta \not\equiv 1/0$; $a_{n-1}$ is odd on the top; $a_{n-1}$ is even on the bottom; The bold line in each $\AT{\alpha}$ is a part of the Seifert path.}
	\label{fig:proof_of_sign_odd_not10}
\end{figure}
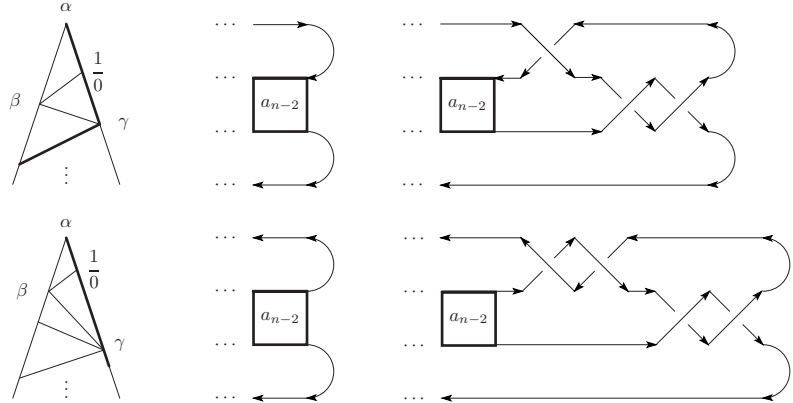

If $n$ is even and $\beta \equiv 1/0$, then $\AT{\alpha}$, $\tb(\gamma)$, and $\tb(\alpha)$ are shown in Figure \ref{fig:proof_of_sign_even_10}. By the same way as above, we have $t_k^{(\alpha)} = t_k^{(\gamma)}$ for $1 \le k \le n-2$ and
\[
\begin{array}{cccl}
t_{n-2}^{(\gamma)} = t_{n-2}^{(\alpha)} = +1, & t_{n-1}^{(\alpha)} = +1, & t_n^{(\alpha)} = -1 & \text{if $a_{n-1}$ is odd;}\\
t_{n-2}^{(\gamma)} = t_{n-2}^{(\alpha)} = -1, & t_{n-1}^{(\alpha)} = +1, & t_n^{(\alpha)} = -1 & \text{if $a_{n-1}$ is even.}
\end{array}
\]
On the other hand, the sign of the rightmost crossing of $\AT{\gamma}$ is $+1$ if $a_{n-1}$ is odd and $-1$ if $a_{n-1}$ is even by (\ref{eqn:sign_formula}). Without loss of generality we can assume the rightmost crossing is
\begin{xy}
	(0,0)="A",(5,0)="B",(5,5)="C",(0,5)="D",
	\ar "D";"B"
	\ar "C";"A"|!{"D";"B"}\hole
\end{xy}
if $a_{n-1}$ is odd and 
\begin{xy}
	(0,0)="A",(5,0)="B",(5,5)="C",(0,5)="D",
	\ar "D";"B"
	\ar "A";"C"|!{"D";"B"}\hole
\end{xy}
if $a_{n-1}$ is even. According to Figure \ref{fig:proof_of_sign_even_10}, the common parts of $\tb(\gamma)$ and $\tb(\alpha)$ have same orientation. Thus we have $\epsilon_k^{(\alpha)} = \epsilon_k^{(\gamma)}$ for $1 \le k \le n-2$ and 
\[
\begin{array}{ccl}
\epsilon_{n-1}^{(\alpha)} = +1 & \epsilon_n^{(\alpha)} = -1 & \text{if $a_{n-1}$ is odd;}\\
\epsilon_{n-1}^{(\alpha)} = +1 & \epsilon_n^{(\alpha)} = -1 & \text{if $a_{n-1}$ is even.}
\end{array}
\]
Therefore the theorem and (\ref{eqn:sign_formula}) holds.
\begin{figure}
	\centering
	\scalebox{0.7}{\input{proof_of_sign_even_10.tex}}
	\caption{$\AT{\alpha}$, $\tb(\gamma)$, and $\tb(\alpha)$ when $n$ is even and $\beta \equiv 1/0$; $a_{n-1}$ is odd on the top; $a_{n-1}$ is even on the bottom; The bold line in each $\AT{\alpha}$ is a part of the Seifert path.}
	\label{fig:proof_of_sign_even_10}
\end{figure}
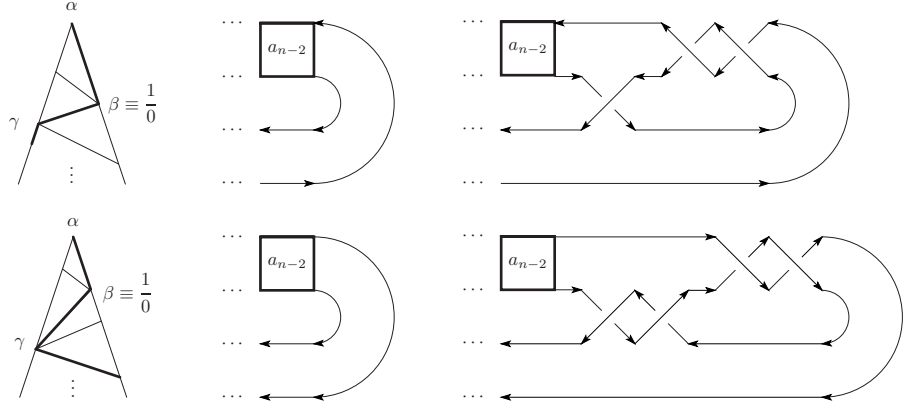

If $n$ is even and $\beta \not\equiv 1/0$, then $\AT{\alpha}$, $\tb(\gamma)$, and $\tb(\alpha)$ are shown in Figure \ref{fig:proof_of_sign_even_not10}. In the same manner we can show that
\[
\begin{array}{cccl}
t_{n-2}^{(\gamma)} = t_{n-2}^{(\alpha)} = -1, & t_{n-1}^{(\alpha)} = +1, & t_n^{(\alpha)} = +1 & \text{if $a_{n-1}$ is odd;}\\
t_{n-2}^{(\gamma)} = t_{n-2}^{(\alpha)} = +1, & t_{n-1}^{(\alpha)} = +1, & t_n^{(\alpha)} = +1 & \text{if $a_{n-1}$ is even,}\\
&&&\\
\epsilon_{n-2}^{(\gamma)} = \epsilon_{n-2}^{(\alpha)} = -1 & \epsilon_{n-1}^{(\alpha)} = +1 & \epsilon_n^{(\alpha)} = +1 & \text{if $a_{n-1}$ is odd;}\\
\epsilon_{n-2}^{(\gamma)} = \epsilon_{n-1}^{(\alpha)} = +1 & \epsilon_{n-1}^{(\alpha)} = +1 & \epsilon_n^{(\alpha)} = +1 & \text{if $a_{n-1}$ is even,}
\end{array}
\]
and the proof when $p/q \equiv 1/1 \text{ or } 0/1 \mod 2$ is complete.
\begin{figure}
	\centering
	\scalebox{0.7}{\input{proof_of_sign_even_not10.tex}}
	\caption{$\AT{\alpha}$, $\tb(\gamma)$, and $\tb(\alpha)$ when $n$ is even and $\beta \not\equiv 1/0$; $a_{n-1}$ is odd on the top; $a_{n-1}$ is even on the bottom; The bold line in each $\AT{\alpha}$ is a part of the Seifert path.}
	\label{fig:proof_of_sign_even_not10}
\end{figure}

Then, we prove the theorem for $p/q \equiv 1/0 \mod 2$ by induction on the length of continued fraction expansion. If $n=1$, then $p/q = [2]$ and the theorem holds by definition. If $n=2$, then $p/q = [1,1] = [2]$. Now suppose $n > 2$. Since $\cf{n}{-1,1} = \cf{n}{}$, we need only consider the case of $a_n = 2$. Let $\beta = \cf{n-1}{,1}$ and $\gamma = \cf{n-1}{}$. If $\beta \equiv 0/1$, then the Seifert path of $\AT{p/q}$ is obtained by adding an edge $(p/q,\beta)$ to that of $\AT{\beta}$. Thus we have $t_k^{(p/q)} = t_k^{(\beta)}$ for $1 \le k \le n-1$ and $t_n^{(p/q)} = t_{n-1}^{(\beta)}$. On the other hand, by (\ref{eqn:sign_formula}), the rightmost crossing of $\tb(\beta)$ is 
\begin{xy}
	(0,0)="A",(5,0)="B",(5,5)="C",(0,5)="D",
	\ar "A";"C"
	\ar "D";"B"|!{"A";"C"}\hole
\end{xy}
or 
\begin{xy}
	(0,0)="A",(5,0)="B",(5,5)="C",(0,5)="D",
	\ar "C";"A"
	\ar "B";"D"|!{"A";"C"}\hole
\end{xy}
. Since $\tb(p/q)$ is obtained by adding one crossing to the rightmost $2$-strand braid of $\tb(\beta)$, the rightmost crossing of $\tb(p/q)$ is same as that of $\tb(\beta)$. Thus the orientation of $\tb(p/q)$ other than the $n$-th $2$-strand braid is same as that of $\tb(\beta)$ and we have $\epsilon_k^{(p/q)} = \epsilon_k^{(\beta)}$ for $1\le k \le n$. Therefore the theorem holds. If $\beta \not\equiv 1/0$, then we have  $\gamma \equiv 1/0$. The two edges from top of the Seifert path of $\AT{p/q}$ are $(p/q,\gamma)$ and $(\gamma,\cf{n-2}{})$. Thus we have $t_k^{(p/q)} = t_k^{(\gamma)}$ for $1 \le k \le n-1$ and $t_n^{(p/q)} = -t_{n-1}^{(\gamma)}$. On the other hand, according to Figure \ref{fig:proof_of_sign_10}, we have $\epsilon_k^{(p/q)} = \epsilon_k^{(\gamma)}$ for $1\le k \le n-1$ and $\epsilon_k^{(p/q)} = -\epsilon_k^{(\gamma)}$ if $n$ is odd. Similarly, we have same equation if $n$ is even. Therefore the theorem holds.
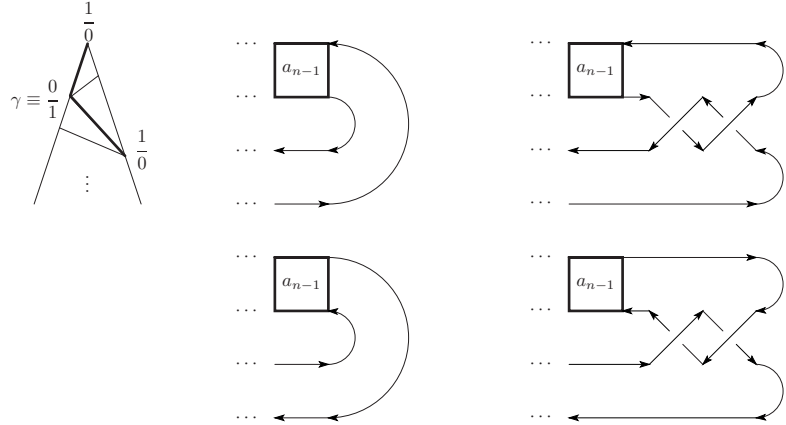
\begin{figure}
	\centering
	\scalebox{0.7}{\input{proof_of_sign_10.tex}}
	\caption{$\AT{p/q}$, $\tb(\gamma)$, and $\tb(p/q)$ when $n$ is odd. The bold line in $\AT{p/q}$ is a part of the Seifert path.}
	\label{fig:proof_of_sign_10}
\end{figure}
\end{proof}

\begin{cor}
\label{cor:sign_formula}
Let $p/q = \cf{n}{} \equiv 1/1 \text{ or } 0/1 \mod 2$. Then
\[
t_n = \begin{cases}
1 & \text{if the left vertex of the top triangle is labeled $1/0$ modulo $2$;}\\
-1 & \text{if the right vertex of the top triangle is labeled $1/0$ modulo $2$.}
\end{cases}
\]
\end{cor}

\begin{cor}
\label{cor:sign_of_mirror}
If $q$ is odd, then the sign of $i$-th $2$-strand braid of $\tb(p/q)$ and that of $(i+1)$-th $2$-strand braid of $\tb((q-p)/q)$ are different.
\end{cor}
\begin{proof}
Since Proposition \ref{prop:ancestral_triangle} and the definition of the Seifert path, the Seifert path of $\AT{(q-p)/q}$ is the mirror image of that of $\AT{p/q}$. Therefore $t_i^{(p/q)}$'s and $t_{i+1}^{((q-p)/q)}$'s satisfy the same recursion. Since the initial values are different, $t_i^{(p/q)}$ and $t_{i+1}^{((q-p)/q)}$ are different.
\end{proof}

\begin{ex}
\label{exmp:sign_of_crossing}
\begin{figure}
	\centering
	\scalebox{0.7}{\input{example_of_Seifert_path.tex}}
	\caption{$\AT{7/19}$ and its Seifert path (bold line).}
	\label{fig:example_of_Seifert_path}
\end{figure}
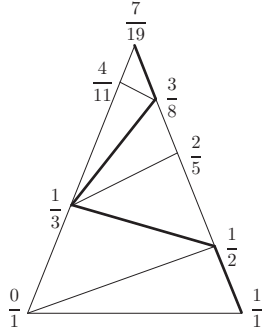
The $\AT{7/19}$ and its Seifert path is shown in Figure \ref{fig:example_of_Seifert_path}. According to that, we have $t_2 = t_1$, $t_3 = -t_2$, and $t_4 = -t_3$. Since $t_1 = -1$ by definition, $t_1 = -1$, $t_2 = -1$, $t_3 = +1$, and $t_4 = -1$. On the other hand, according to Figure \ref{fig:example_of_oriented_two-brigde_knot}, the sign of the first $2$-strand braid of $\tb(7/19)$ is $-1$, the second is $-1$, the third is $+1$, and the fourth is $-1$. 
\end{ex}

\subsection{Alexander polynomials}
 We denote the Alexander polynomial of an oriented link $L$ by $\alexander{L}(t)$. Here we use the following normalization: if $L$ is the unknot, then the Alexander polynomial $\alexander{\text{unknot}}(t)$ is $1$. Alexander polynomials satisfy the following recursion called Skein relation:
\begin{equation}
\alexander{L_+}(t) - \alexander{L_-}(t) = (t^{1/2} - t^{-1/2})\alexander{L_0}(t)
\end{equation}
where $L_+$, $L_-$, and $L_0$ are oriented links which only differ in a ball as Figure \ref{fig:skein_relation}.
\begin{figure}
	\centering
	\input{skein_relation.tex}
	\caption{Skein diagrams}
	\label{fig:skein_relation}
\end{figure}
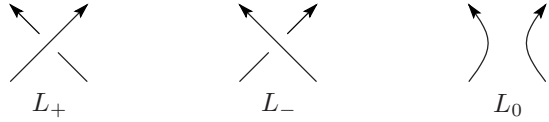

For a knot $K$, it is known that Alexander polynomial is symmetric, namely $\alexander{K}(t) = \alexander{K}(t^{-1})$.

\begin{ex}
\label{exmp:alexander_polynomial}
\begin{figure}
	\centering
	\scalebox{0.8}{\input{example_of_alexander_polynomial.tex}}
	\caption{Skein diagrams for Example \ref{exmp:alexander_polynomial}.}
	\label{fig:example_of_alexander_polynomial}
\end{figure}
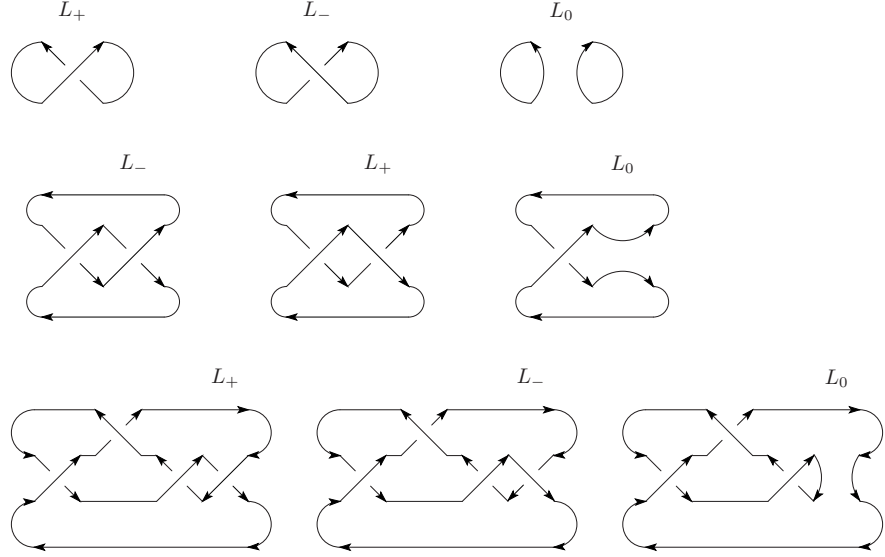

Let $\alex{\circ\circ}$ be the Alexander polynomial of two disjoint unknots. Skein diagrams of it is shown on the top of Figure \ref{fig:example_of_alexander_polynomial}. Since both of $L_+$ and $L_-$ are the unknots, we have
\[
\alex{\circ\circ} = (\alex{\text{unknot}} - \alex{\text{unknot}})/(t^{1/2} - t^{-1/2}) = 0.
\]

Let $L$ be the $2$-component link shown in the left on the middle of Figure \ref{fig:example_of_alexander_polynomial}. Since $L_+$ is two disjoint unknots and $L_0$ is the unknot, we have
\[
\alex{L} = \alex{\circ\circ} - (t^{1/2} - t^{-1/2})\alex{\text{unknot}} = t^{-1/2} - t^{-1/2}.
\]
Skein relation for $\tb(3/5)$ is shown on the bottom of Figure \ref{fig:example_of_alexander_polynomial}. Since $L_-$ is the unknot and $L_0$ is $L$, we have
\[
\alex{\tb(3/5)} = \alex{\text{unknot}} + (t^{1/2} - t^{-1/2})\alex{L} = 1 + (t^{1/2} - t^{-1/2})(t^{-1/2} - t^{1/2}) = -t^{-1} + 3 - t
\]
\end{ex}

We denote the Alexander polynomial of $\tb(\cf{n}{})$ by $\alex{\cf{n}{}}$.
\begin{lem}
\label{lem:alexander_recursion}
If $\cf{n}{} \equiv 0/1 \text{ or } 1/0$, then $\alex{\cf{n}{}}$'s satisfy the following recursions with the convention that $\cf{n-1}{,0} = \cf{n-2}{}$, $\alex{[]} = 1$, and $\alex{[0]} = 0$:
\begin{equation}
\label{eqn:alexander_identity}
\alex{\cf{n}{-1,1}} = \alex{\cf{n}{}};
\end{equation}
and if $a_n \ge 2$, then $\alex{\cf{n}{}} = $
\begin{equation}
\begin{cases}
\label{eqn:alexander_recursion}
\alex{[a_1,\dots,a_n-2]} + (t^{1/2} - t^{-1/2})\alex{[a_1,\dots,a_{n-1}]} & \text{if $n$ is odd and $t_n = +1$;}\\
\alex{[a_1,\dots,a_n-2]} - (t^{1/2} - t^{-1/2})\alex{[a_1,\dots,a_n-1]} & \text{if $n$ is odd and $t_n = -1$;}\\
\alex{[a_1,\dots,a_n-2]} + (t^{1/2} - t^{-1/2})\alex{[a_1,\dots,a_n-1]} & \text{if $n$ is even and $t_n = +1$;}\\
\alex{[a_1,\dots,a_n-2]} - (t^{1/2} - t^{-1/2})\alex{{[a_1,\dots,a_{n-1}]}} & \text{if $n$ is even and $t_n = -1$.}
\end{cases}
\end{equation}
\end{lem}

\begin{proof}
(\ref{eqn:alexander_identity}) clearly holds because Alexander polynomial is a knot invariant. We only prove (\ref{eqn:alexander_recursion}) when $n$ is odd. When $n$ is even, it is proved in the same way. 

First, we show when $a_n \ge 3$. If $t_n = +1$, then the sign of the rightmost crossing of $\tb(\cf{n})$ is $+1$. According to Figure \ref{fig:proof_of_alexander_recursion}, applying Skein relation $\alexander{L_+} = \alexander{L_-} + (t^{1/2} - t^{-1/2})\alexander{L_0}$ to the rightmost crossing yields the equation
\[
\alex{\cf{n}{}} = \alex{\cf{n}{-2}} + (t^{1/2} - t^{-1/2})\alex{{\cf{n-1}{}}}.
\]
If $t_n = -1$, then applying Skein relation $\alexander{L_-} = \alexander{L_+} - (t^{1/2} - t^{-1/2})\alexander{L_0}$ to the rightmost crossing yields the equation
\[
\alex{\cf{n}{}} = \alex{\cf{n}{-2}} - (t^{1/2} - t^{-1/2})\alex{{\cf{n}{-1}}}.
\]

Next we prove when $a_n = 2$. When we perform crossing changes at the rightmost crossing, the rightmost crossing and the crossing to the left disappear. Hence $L_+$ or $L_-$ has the form of the bottom of Figure \ref{fig:proof_of_alexander_recursion}. Since the $(n-1)$-th $2$-strand braid become trivial, it is isotopic to $\tb(\cf{n-2}{})$.
\begin{figure}
	\centering
	\scalebox{0.8}{\input{proof_of_alexander_recursion.tex}}
	\caption{Proof of Lemma \ref{lem:alexander_recursion}. $t_n = +1$ on the top; $t_n = -1$ on the middle; A diagram after applying the Skein relation when $a_n = 2$ on the bottom.}
	\label{fig:proof_of_alexander_recursion}
\end{figure}
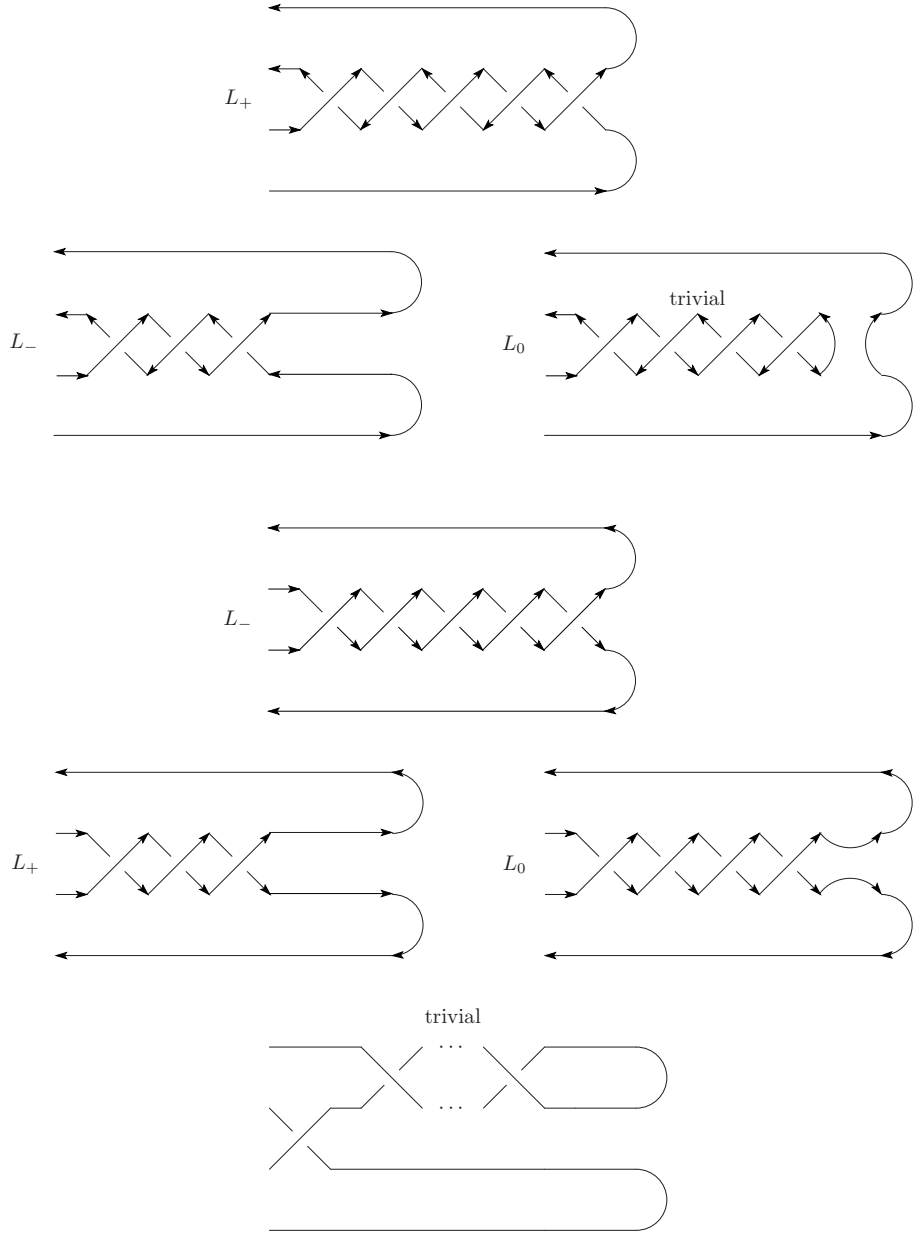
\end{proof}


\subsection{The main result}
Now we prepare to state the main theorem. Let $p/q = \cf{n}{}$ be a fraction. We denote the $i$-th triangle in $\fan_k$ from bottom by $\tri_{i}^{(k)}$, that is, $\tri_{i}^{(k)}$ is $\tri_i$ if $k = 1$ and $\tri_{l_{k-1} + i - 1}$ if $k > 1$. We define a sign of each triangle $\tri$ denoted by $\trisign{}{(\tri)}$ as follows:
\begin{itemize}
\item If the bottom edge of $\fan_k$ is in the Seifert path, then
\[
\trisign{}{(\tri_{i}^{(k)})} = (-1)^i.
\]
\item If $k$ is odd, then
\[
\trisign{}{(\tri_{i}^{(k)})} = \begin{cases}
+1 & \text{if $\fan_k$ is on the left side of the Seifert path;}\\
(-1)^{i-1} & \text{if $\fan_k$ is on the right side of the Seifert path.}\\
\end{cases}
\]
\item If $k$ is even, then
\[
\trisign{}{(\tri_{i}^{(k)})} = \begin{cases}
+1 & \text{if $\fan_k$ is on the right side of the Seifert path;}\\
(-1)^{i-1} & \text{if $\fan_k$ is on the left side of the Seifert path.}\\
\end{cases}
\]
\end{itemize}

We write $\trisign{i}{} = \trisign{}{(\tri_i)}$. We specialize $F$-polynomial $\Fp{\cf{n}{}}$ by setting $y_i = -t^{\trisign{i}{}}$ for all $i$ and denote it by $\SFp{\cf{n}{}}$.

We are ready to state the main theorem.

\begin{thm}
\label{thm:main}
With the above notation
\[
\alex{\cf{n}{}} = \alexsign{} t^{\alexexp{}}\SFp{\cf{n}{}}
\]
where
\begin{align*}
 \alexexp{} &= -\cfrac{1}{2}\sum_{k = 1}^{n}\alexexp{}(k),\\
 \alexexp{}(k) &= \begin{cases}
 a_k - 1 & \text{if $k$ is odd and $t_k = -1$ or $k$ is even and $t_k = +1$;}\\
 t_k & \text{if the bottom edge of $\fan_k$ is in the Seifert path;}\\
 1 & \text{otherwise,}
 \end{cases}\\
 \alexsign{} &= (-1)^{n + \sum_{\text{$i$:even}}a_i - pq}.
\end{align*}
\end{thm}

\begin{proof}
Throughout the proof, $\alexexp{\cf{n}{}}$ and $\alexsign{\cf{n}{}}$ denote $\alexexp{}$ and $\alexsign{}$ determined according to $\cf{n}{}$. It is clear from definitions that $\alexexp{\cf{n}{}} = \alexexp{\cf{n}{-1,1}}$ and $\alexsign{\cf{n}{}} = \alexsign{\cf{n}{-1,1}}$. First we prove this theorem when $p/q \equiv 1/0 \text{ or } 0/1$ by induction on $n$ and $a_n$. Note that continued fractions which appear in recursions are equivalent to $0/1$ or $1/0$ and the Seifert path of $\AT{r/s}$ is obtained by restricting that of $\AT{p/q}$ if $r/s$ is on the Seifert path of $\AT{p/q}$.

If $n = 1$, then $a_1$ is even. It is easily seen that
\[
\alex{[a_1]} = \cfrac{a_1}{2}(t^{1/2} - t^{-1/2}).
\]
On the other hand, since the Seifert path goes along the left boundary edge, we have $e_i = (-1)^{i-1}$ and $\alexexp{[a_1]} = -1/2$. An easy computation shows that $\Fp{[a_1]} = \sum_{k=0}^{a_1 - 1} \prod_{i = a_1 - i}^{a_1 - 1}y_i$. Therefore
\[
\Falex{[a_1]} = (-1)^1 t^{-1/2}~ \cfrac{a_1}{2}(1 - t) = \alex{[a_1]}
\]
and the theorem holds.

Now we consider in the case when $n = 2$. If $a_1$ is odd, then the Seifert path goes along the left boundary edges. Thus we have $t_1 = t_2 = +1$ and Lemma \ref{lem:alexander_recursion} implies 
\[
\alex{[a_1,a_2]} = \begin{cases}
\alex{[a_1,a_2 - 2]} + (t^{1/2} - t^{-1/2})\alex{[a_1,a_2 - 1]} & \text{if $a_2 \ge 3$;}\\
1 + (t^{1/2} - t^{-1/2})\alex{[a_1,1]} & \text{if $a_2 = 2$;}\\
\alex{[a_1 + 1]}  & \text{if $a_2 = 1$.}\\
\end{cases}
\]
On the other hand, we have
\[
\trisign{i}{} = \begin{cases}
(-1)^{i-1} & \text{if $1 \le i \le a_1 - 1$;}\\
+1 & \text{if $a_1 \le i$}
\end{cases}
\]
and (\ref{eqn:F-polynomial_identity}) and (\ref{eqn:F-polynomial_rec_even_12}) in Corollary \ref{cor:F-polynomial_2_recursion} yields
\[
\SFp{[a_1,a_2]} = \begin{cases}
(1 - t)\SFp{[a_1,a_2 - 1]} + t\SFp{[a_1,a_2-2]} & \text{if $a_2 \ge 3$;}\\
(1-t)\SFp{[a_1,1]} + t & \text{if $a_2 = 2$;}\\
\SFp{[a_1+1]} & \text{if $a_2 = 1$.}
\end{cases}
\]
Suppose that $a_2 = 1$. By induction on $n$, we have $\alex{a_1 + 1} = \alexsign{[a_1+1]}t^{\alexexp{[a_1+1]}}\SFp{[a_1+1]}$. Since recursions above, the theorem holds. Now suppose that $a_2 = 2$. Since $\alexexp{[a_1,2]} = -1$, $\alexexp{[a_1,1]} = -1/2$, $\alexsign{[a_1,2]} = +1$, and $\alexsign{[a_1,1]} = -1$, we get
\[
\begin{array}{rcl}
\Falex{[a_1,2]} &=& t^{-1}((1-t)\SFp{[a_1,1]} + t)\\
&=& -t^{-1/2}(t^{1/2} - t^{-1/2})\SFp{[a_1,1]} + 1\\
&=& (t^{1/2} - t^{-1/2})\alex{[a_1,1]} + 1 = \alex{[a_1,2]}.
\end{array}
\]
Now suppose that $a_2 \ge 3$. To simplify notation, we write $i$ instead of $[a_1,a_2-i]$ for $i = 0,1,2$ in subscript. By induction on $a_2$, we have $\Falex{i} = \alexander{i}$ for $i=1,2$. Since $\alexsign{0} = -\alexsign{1} = \alexsign{2}$ and
$\alexexp{0} = \alexexp{1} - 1/2 = \alexexp{2} - 1$, we get
\[
\begin{array}{rcl}
\Falex{0} &=& \alexsign{0}t^\alexexp{0}((1 - t)\SFp{1} + t\SFp{2})\\
&=& (t^{1/2} - t^{-1/2})\Falex{1} + \Falex{2}\\
&=& (t^{1/2} - t^{-1/2})\alexander{1} + \alexander{2}\\
&=& \alexander{0},
\end{array}
\]
and the theorem holds.
We prove when $a_1$ is even. Since $p/q \equiv 1/0 \text{ or } 0/1$, $a_2$ is also even. Then we have $t_1 = +1$ and $t_2 = -1$ because the Seifert path is $p/q \rightarrow 1/a_1 \rightarrow 0/1$. Lemma \ref{lem:alexander_recursion} implies 
\[
\alex{[a_1,a_2]} = \begin{cases}
\alex{[a_1,a_2-2]} - (t^{1/2} - t^{-1/2})\alex{[a_1]} & \text{if $a_2 \ge 4$;}\\
1 - (t^{1/2} - t^{-1/2})\alex{[a_1]} & \text{if $a_2 = 2$.}
\end{cases}
\]
On the other hand, since the bottom edge of $\fan_2$ is in the Seifert path, we have
\[
\trisign{i}{} = \begin{cases}
(-1)^{i-1} & \text{if $1 \le i \le a_1 - 1$;}\\
(-1)^{i - a_1 - 1} & \text{if $a_1 \le i$.}
\end{cases}
\]
Thus (\ref{eqn:F-polynomial_rec_even_2n1}) in Corollary \ref{cor:F-polynomial_2_recursion} yields
\[
\SFp{[a_1,a_2]} = \begin{cases}
-t^{-1}(1 - t)\SFp{[a_1]} + \SFp{[a_1,a_2-2]} & \text{if $a_2 \ge 4$;}\\
-t^{-1}(1 - t)\SFp{[a_1]} + 1 & \text{if $a_2 = 2$.}
\end{cases}
\]
Suppose that $a_2 = 2$. By induction on $n$ and definition, we have $\Falex{[a_1]} = \alex{[a_1]}$, $\alexexp{[a_1,2]} = 0$, $\alexexp{[a_1]} = -1/2$, $\alexsign{[a_1,2]} = +1$, and $\alexsign{[a_1]} = -1$. Therefore
\[
\Falex{[a_1,2]} = -t^{-1}(1 - t)\SFp{[a_1]} + 1 = 1 - (t^{1/2} - t^{-1/2})\alex{[a_1]} = \alex{[a_1,2]}
\]
and the theorem holds. Now suppose that $a_2 \ge 4$. By induction and definition, we have $\Falex{[a_1,a_2-2]} = \alexander{[a_1,a_2-2]}$, $\Falex{[a_1]} = \alex{[a_1]}$, $\alexexp{[a_1,a_2]} = \alexexp{[a_1,a_2-2]} = \alexexp{[a_1]} + 1/2$, and $\alexsign{[a_1,a_2]} = \alexsign{[a_1,a_2-2]} = -\alexsign{[a_1]}$. Therefore we get
\[
\begin{array}{rcl}
\Falex{[a_1,a_2]} &=& \alexsign{[a_1,a_2]}t^{\alexexp{[a_1,a_2]}}(-t^{-1}(1 - t)\SFp{[a_1]} + \SFp{[a_1,a_2-2]})\\
&=& \alexsign{[a_1,a_2]}t^{\alexexp{[a_1,a_2]}-1/2}(t^{1/2} - t^{-1/2})\SFp{[a_1]} + \Falex{[a_1,a_2-2]}\\
&=& \alex{[a_1,a_2-2]} - \alex{[a_1]} = \alex{[a_1,a_2]}
\end{array}
\]
and the proof when $n = 2$ is complete.

Now suppose that $n \ge 3$. By induction on $n$, we may assume that
\[
\Falex{\cf{m}{}} = \alex{\cf{m}{}}
\]
for all $m < n$. 

If $n$ is odd and $\cf{n-1}{} \equiv 1/1$, then $\cf{n}{} \equiv 0/1$ or $\cf{n}{+1} \equiv 0/1$. Thus we have $t_n = -1$ because of Corollary \ref{cor:sign_formula}. According to Figure \ref{fig:proof_of_main_odd_minus}, the Seifert path goes the right boundary edges of $\fan_n$ and we have $\trisign{}{(\tri_i^{(n)})} = +1$.

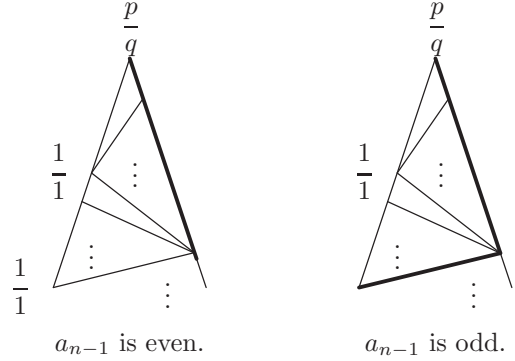
\begin{figure}
	\centering
	{\input{proof_of_main_odd_minus.tex}}
	\caption{Proof when $\cf{n-1}{} \equiv 1/1$. The bold line in each ancestral triangle represents its Seifert path.}
	\label{fig:proof_of_main_odd_minus}
\end{figure}

Now suppose that $a_n = 1$. If $a_{n-1}$ is even, then the bottom edge of $\fan_{n-1}$ is not in the Seifert path. Therefore we have $\trisign{}{(\tri_i^{(n-1)})} = (-1)^{i-1}$ and $\trisign{}{(\tri_1^{(n)})} = +1$. Since $\trisign{}{(\tri_{a_{n-1}+1}^{(n-1)})} = +1$, we have $\SFp{\cf{n-1}{+1}} = \SFp{\cf{n-1}{,1}}$. Therefore we get
\[
\begin{array}{rcl}
\alex{\cf{n-1}{,1}} &=& \alex{\cf{n-1}{+1}}\\
 &=& \Falex{\cf{n-1}{+1}}\\
 &=& \Falex{\cf{n-1}{,1}}.
\end{array}
\]
If $a_{n-1}$ is odd, then the bottom edge of $\fan_{n-1}$ is in the Seifert path. Therefore we have $\trisign{}{(\tri_i^{(n-1)})} = (-1)^{i}$ and $\trisign{}{(\tri_1^{(n)})} = +1$. In the same way, we get $\alex{\cf{n-1}{,1}} = \Falex{\cf{n-1}{,1}}$.

Now suppose that $a_n = 2$. By (\ref{eqn:F-polynomial_rec_odd2_n12}) in Corollary \ref{cor:F-polynomial_2_recursion} and Lemma \ref{lem:alexander_recursion}, we have
\[
\begin{array}{rcl}
\SFp{\cf{n-1}{,2}} &=& \begin{cases}
(1 - t)\SFp{\cf{n-1}{,1}} + t\SFp{\cf{n-2}{}} & \text{if $a_{n-1}$ is even;}\\
(1 - t)\SFp{\cf{n-1}{,1}} - \SFp{\cf{n-2}{}} & \text{if $a_{n-1}$ is odd;}\\
\end{cases}\\
\alex{\cf{n-1}{,2}} &=& \alex{\cf{n-2}{}} - (t^{1/2} - t^{-1/2})\alex{\cf{n-1}{,1}}.
\end{array}
\]
To simplify notation, we write $0$ instead of $\cf{n-1}{,2}$, $1$ instead of ${\cf{n-1}{,1}}$, and $2$ instead of $\cf{n-2}{}$ in subscript.
We have $\alexexp{0} = \alexexp{1} - 1/2 = \alexexp{2} - 1$ and $\alexsign{0} = \alexsign{1} = \alexsign{2}$ if $a_{n-1}$ is even and $\alexexp{0} = \alexexp{1} - 1/2 = \alexexp{2}$ and $\alexsign{0} = \alexsign{1} = -\alexsign{2}$ if $a_{n-1}$ is odd. Therefore
\[
\begin{array}{rcl}
\Falex{0} &=& \begin{cases}
\alexsign{0}t^{\alexexp{0}}(1 - t)\SFp{1} + \alexsign{0}t^{\alexexp{0}}t\SFp{2} & \text{if $a_{n-1}$ is even;}\\
\alexsign{0}t^{\alexexp{0}}(1 - t)\SFp{1} - \alexsign{0}t^{\alexexp{0}}\SFp{2} & \text{if $a_{n-1}$ is odd;}\\
\end{cases}\\
&=& \begin{cases}
-\alexsign{1}t^{\alexexp{1}}(t^{1/2} - t^{-1/2})\SFp{1} + \alexsign{2}t^{\alexexp{2}}\SFp{2} & \text{if $a_{n-1}$ is even;}\\
-\alexsign{1}t^{\alexexp{1}}(t^{1/2} - t^{-1/2})\SFp{1} + \alexsign{2}t^{\alexexp{2}}\SFp{2} & \text{if $a_{n-1}$ is odd;}\\
\end{cases}\\
&=& -(t^{1/2} - t^{-1/2})\alex{1} + \alex{2}\\
&=& \alex{0}
\end{array}
\]
and the theorem holds.

Now suppose that $a_n \ge 3$. To simplify notation, we write $i$ instead of $[a_1,a_2-i]$ for $i = 0,1,2$ in subscript. By induction on $a_n$, (\ref{eqn:F-polynomial_rec_odd_12}), and Lemma \ref{lem:alexander_recursion}, we have
\[
\begin{array}{rcl}
\alex{1} &=& \Falex{1},\\
\alex{2} &=& \Falex{2},\\
\SFp{0} &=& (1-t)\SFp{1} + t\SFp{2},\\
\alex{0} &=& \alex{2} - (t^{1/2} - t^{-1/2})\alex{1}.
\end{array}
\]
Since $\alexexp{0} = \alexexp{1} -1/2 = \alexexp{2} - 1$ and $\alexsign{0} = \alexsign{1} = \alexsign{2}$, we get
\[
\begin{array}{rcl}
\Falex{0} &=& \alexsign{0}t^{\alexexp{0}}((1-t)\SFp{1} + t\SFp{2})\\
&=& -\alexsign{1}t^{\alexexp{1}}(t^{1/2}-t^{-1/2})\SFp{1} + \alexsign{2}t^{\alexexp{2}}\SFp{2}\\
&=& -(t^{1/2}-t^{-1/2})\alex{1} + \alex{2}\\
&=& \alex{0}.
\end{array}
\]

If $n$ is odd and $\cf{n-1}{} \not\equiv 1/1$, then we have $t_n = +1$ because of Corollary \ref{cor:sign_formula}. According to Figure \ref{fig:proof_of_main_odd_plus}, we need to consider three cases.

\begin{figure}
	\centering
	{\input{proof_of_main_odd_plus.tex}}
	\caption{Proof when $\cf{n-1}{} \not\equiv 1/1$. The bold line in each ancestral triangle represents the Seifert path.}
	\label{fig:proof_of_main_odd_plus}
\end{figure}

When $a_{n}$ is odd, both of $\fan_{n}$ and $\fan_{n-1}$ are on the right of the Seifert path. Thus we have $\trisign{}{(\tri_i^{(n-1)})} = +1$ and $\trisign{}{(\tri_i^{(n)})} = (-1)^{i-1}$. If $a_n = 1$, then since $\trisign{}{(\tri_{a_{n-1}+1}^{(n-1)})} = \trisign{}{(\tri_1^{(n)})} = +1$, we get $\SFp{\cf{n-1}{+1}} = \SFp{\cf{n-1}{,1}}$ and the result follows. Now suppose that $a_n \ge 3$. To simplify notation, we write $0$ instead of $\cf{n}{}$, $1$ instead of ${\cf{n-1}{}}$, and $2$ instead of $\cf{n}{-2}$ in subscript.
By induction on $a_n$, (\ref{eqn:F-polynomial_rec_odd_2n1}), and Lemma \ref{lem:alexander_recursion}, we have
\[
\begin{array}{rcl}
\alex{2} &=& \Falex{2},\\
\SFp{0} &=& (1-t)\SFp{1} + \SFp{2},\\
\alex{0} &=& \alex{2} + (t^{1/2} - t^{-1/2})\alex{1}.
\end{array}
\]
Since $\alexexp{0} = \alexexp{2} = \alexexp{1} - 1/2$ and $\alexsign{0} = \alexsign{2} = -\alexsign{1}$, we get
\[
\begin{array}{rcl}
\Falex{0} &=& \alexsign{0}t^{\alexexp{0}}((1-t)\SFp{1} + \SFp{2})\\
&=& \alexsign{1}t^{\alexexp{1}}(t^{1/2}-t^{-1/2})\SFp{1} + \alexsign{2}t^{\alexexp{2}}\SFp{2}\\
&=& (t^{1/2}-t^{-1/2})\alex{1} + \alex{2} = \alex{0}
\end{array}
\]
and the theorem holds.

When $a_n$ and $a_{n-1}$ are even, both of the bottom edge of $\fan_{n-1}$ and that of $\fan_n$ is in the Seifert path. Then we have $\trisign{}{(\tri_i^{(n-1)})} = (-1)^{i}$ and $\trisign{}{(\tri_i^{(n)})} = (-1)^{i}$. Hence (\ref{eqn:F-polynomial_rec_odd2_n12}) yields
\[
\SFp{\cf{n-1}{,2}} = (1-t)\SFp{\cf{n-1}{}} + \SFp{\cf{n-2}{}}
\]
and (\ref{eqn:F-polynomial_rec_odd_2n1}) yields
\[
\SFp{\cf{n}{}} = (1-t)\SFp{\cf{n-1}{}} + \SFp{\cf{n}{-2}}.
\]
The first recursion is included by the second recursion with the convention that $\cf{n-1}{,0} = \cf{n-2}{}$. On the other hand, by Lemma \ref{lem:alexander_recursion} we have
\[
\alex{\cf{n}{}} = (1-t)\alex{\cf{n-1}{}} + \alex{\cf{n}{-2}}.
\]
with the same convention. To simplify notation, we write $0$ instead of $\cf{n}{}$, $1$ instead of ${\cf{n-1}{}}$, and $2$ instead of $\cf{n}{-2}$ in subscript. Since $\alexexp{0} = \alexexp{1} - 1/2 = \alexexp{2}$ and $\alexsign{0} = -\alexsign{1} = \alexsign{2}$, we get
\[
\begin{array}{rcl}
\Falex{0} &=& \alexsign{0}t^{\alexexp{0}}((1-t)\SFp{1} + \SFp{2})\\
&=& \alexsign{1}t^{\alexexp{1}}(t^{1/2}-t^{-1/2})\SFp{1} + \alexsign{2}t^{\alexexp{2}}\SFp{2}.
\end{array}
\]
Suppose that $a_n = 2$. Thus we have $\alex{i} = \Falex{i}$ for $i = 1,2$ by induction on $n$ and we get $\alex{0} = \Falex{0}$. Suppose that $a_n \ge 4$. Then we have $\alex{1} = \Falex{1}$ by induction on $n$ and $\alex{2} = \Falex{2}$ by induction on $a_n$. Therefore $\alex{0} = \Falex{0}$ as required.

When $a_n$ is even and $a_{n-1}$ is odd, the bottom edge of $\fan_n$ is in the Seifert path and that of $\fan_{n-1}$ is not. Since $\fan_{n-1}$ is on the left side of the Seifert path, we have $\trisign{}{(\tri_i^{(n-1)})} = (-1)^{i-1}$ and $\trisign{}{(\tri_i^{(n)})} = (-1)^{i}$. Suppose that $a_n = 2$. To simplify notation, we write $0$ instead of $\cf{n-1}{,2}$, $1$ instead of ${\cf{n-1}{}}$, and $2$ instead of $\cf{n-2}{}$ in subscript. We have $\alexexp{0} = \alexexp{1} - 1/2 = \alexexp{2} - 1$ and $\alexsign{0} = -\alexsign{1} = -\alexsign{2}$. By (\ref{eqn:F-polynomial_rec_odd2_n12}) and Lemma \ref{lem:alexander_recursion}, we have
\[
\begin{array}{rcl}
\SFp{0} &=& (1-t)\SFp{1} -t \SFp{2},\\
\alex{0} &=& \alex{2} + (t^{1/2} - t^{-1/2})\alex{1}.\\
\end{array}
\]
Therefore we get
\[
\begin{array}{rcl}
\Falex{0} &=& \alexsign{0}t^{\alexexp{0}}((1-t)\SFp{1} -t \SFp{2})\\
&=& \alexsign{1}t^{\alexexp{1}}(t^{1/2} - t^{-1/2})\SFp{1} - \alexsign{2}t^{\alexexp{2}}\SFp{2}\\
&=& (t^{1/2} - t^{-1/2})\alex{1} + \alex{2}\\
&=& \alex{0}
\end{array}
\]
and the theorem holds. If $a_n \ge 4$, then the recursion is same as the case when both of $a_n$ and $a_{n-1}$ are even. Therefore the theorem also holds for $a_n \ge 4$. 

The proof is complete when $n$ is odd. We prove when $n$ is even. If $\cf{n-1}{} \equiv 1/1$, then either $\cf{n}{} \equiv 0/1$ or $\cf{n}{+1} \equiv 0/1$ holds. Thus we have $t_n = +1$ because of Corollary \ref{cor:sign_formula}. According to Figure \ref{fig:proof_of_main_even_plus}, the Seifert path goes the left boundary edges of $\fan_n$ and $\fan_n$ is on the right side of the Seifert path. Thus we have $\trisign{}{(\tri_i^{(n)})} = +1$.

\begin{figure}
	\centering
	{\input{proof_of_main_even_plus.tex}}
	\caption{Proof when $\cf{n-1}{} \equiv 1/1$. The bold line in each ancestral triangle represents its Seifert path.}
	\label{fig:proof_of_main_even_plus}
\end{figure}
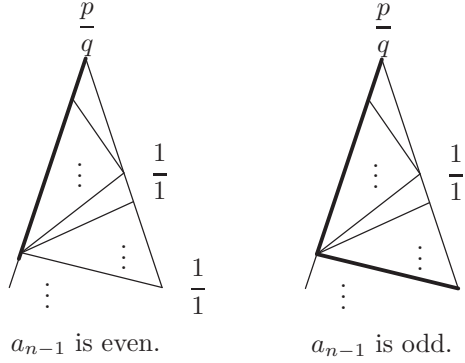

Suppose that $a_n = 1$. According to Figure \ref{fig:proof_of_main_even_plus}, the bottom edge is in the Seifert path if $a_{n-1}$ is even and not if $a_{n-1}$ is odd. Therefore we have
\[
\trisign{}{(\tri_i^{(n-1)})} = \begin{cases}
(-1)^{i-1} & \text{if $a_{n-1}$ is odd;}\\
(-1)^{i} & \text{if $a_{n-1}$ is even.}
\end{cases}
\]
In particular, $\trisign{}{(\tri_{a_{n-1}+1}^{(n-1)})}$ is $+1$. Therefore we have $\SFp{\cf{n-1}{,1}} = \SFp{\cf{n-1}{,1}}$ and $\alex{\cf{n-1}{,1}} = \alex{\cf{n-1}{,1}}$ follows.

Now suppose that $a_n \ge 2$. To simplify notation, we write $0$ instead of $\cf{n-1}{,2}$, $1$ instead of $\cf{n}{-1}$, and $2$ instead of $\cf{n}{-2}$ in subscript. It is easy to check that $\alexexp{0} = \alexexp{1} - 1/2 = \alexexp{2} - 1$ and $\alexsign{0} = -\alexsign{1} =\alexsign{2}$ for $a_n \ge 2$. By induction on $a_n$, (\ref{eqn:F-polynomial_rec_even_12}), and Lemma \ref{lem:alexander_recursion}, we have
\[
\begin{array}{rcl}
\SFp{0} &=& (1 - t)\SFp{1} + t\SFp{2};\\
\alex{0} &=& \alex{1} + (t^{1/2} - t^{-1/2})\alex{2};\\
\alex{1} &=&\Falex{1};\\
\alex{2} &=& \Falex{2};
\end{array}
\]
with the convention that $\cf{n-1}{0} = \cf{n-2}{}$. Therefore we obtain
\[
\begin{array}{rcl}
\Falex{0} &=& \alexsign{0}t^{\alexexp{0}}((1 - t)\SFp{1} + t\SFp{2})\\
&=& \alexsign{1}t^{\alexexp{1}}(t^{1/2} - t^{-1/2})\SFp{1} + \alexsign{2}t^{\alexexp{2}}\SFp{2}\\
&=& (t^{1/2} - t^{-1/2})\alex{1} + \alex{2}\\
&=& \alex{0},
\end{array}
\]
and the theorem holds.

If $\cf{n-1}{} \not\equiv 1/1$, then we have $t_n = -1$ by Corollary \ref{cor:sign_formula}. According to Figure \ref{fig:proof_of_main_even_minus}, we need to consider two cases.

\begin{figure}
	\centering
	{\input{proof_of_main_even_minus.tex}}
	\caption{Proof when $\cf{n-1}{} \not\equiv 1/1$. The bold line in each ancestral triangle represents its Seifert path.}
	\label{fig:proof_of_main_even_minus}
\end{figure}

If $a_n$ is odd, then both of $\fan_n$ and $\fan_{n-1}$ are on left side of the Seifert path. Thus we have $\trisign{}{(\tri_i^{(n-1)})} = (-1)^{i-1}$ and $\trisign{}{(\tri_i^{(n)})} = +1$. Suppose that $a_n = 1$. Since $\trisign{}{(\tri_{a_{n-1}+1}^{(n-1)})} = +1$, we get $\SFp{\cf{n-1}{+1}} = \SFp{\cf{n-1}{,1}}$ and $\alex{\cf{n-1}{,1}} = \SFp{\cf{n-1}{,1}}$ follows. Now Suppose that $a_n \ge 3$. To simplify notation, we write $0$ instead of $\cf{n}{}$, $1$ instead of ${\cf{n-1}{}}$, and $2$ instead of $\cf{n}{-2}$ in subscript. By induction on $a_n$, (\ref{eqn:F-polynomial_rec_even_2n1}), and Lemma \ref{lem:alexander_recursion} we have
\[
\begin{array}{rcl}
\alex{0} &=& \alex{2} - (t^{1/2} - t^{-1/2})\alex{1};\\
\SFp{0} &=& \SFp{2} + (1-t)\SFp{1};\\
\alex{2} &=& \Falex{2};
\end{array}
\]
and we have $\alexexp{0} = \alexexp{1} -1/2 = \alexexp{2}$ and $\alexsign{0} = \alexsign{1} = \alexsign{2}$. Therefore
\[
\begin{array}{rcl}
\Falex{0} &=& \alexsign{0}t^{\alexexp{0}}(\SFp{2} + (1-t)\SFp{1})\\
&=& \alexsign{2}t^{\alexexp{2}}\SFp{2} - \alexsign{1}t^{\alexexp{1}}(t^{1/2} - t^{-1/2})\SFp{1}\\
&=& \alex{2} - (t^{1/2} - t^{-1/2})\alex{1} = \alex{0}
\end{array}
\]
and the theorem holds.

If $a_n$ is even, then we have $\trisign{}{(\tri_i^{(n)})} = (-1)^{i}$. To simplify notation, we write $0$ instead of $\cf{n}{}$, $1$ instead of ${\cf{n-1}{}}$, and $2$ instead of $\cf{n}{-2}$ in subscript. By (\ref{eqn:F-polynomial_rec_even_2n1}), and Lemma \ref{lem:alexander_recursion}, we have
\[
\begin{array}{rcl}
\SFp{0} &=& (- t^{-1} + 1)\SFp{1} + \SFp{2};\\
\alex{0} &=& \alex{2} - (t^{1/2} - t^{-1/2})\alex{1};
\end{array}
\]
and we have $\alexexp{0}  = \alexexp{2} = \alexexp{1} + 1/2$ and $ = \alexsign{0} = \alexsign{2} = -\alexsign{1}$. Therefore we get
\[
\begin{array}{rcl}
\Falex{0} &=& \alexsign{0}t^{\alexexp{0}}((- t^{-1} + 1)\SFp{1} + \SFp{2})\\
&=& -\alexsign{1}t^{\alexexp{1}}(t^{1/2} - t^{-1/2})\SFp{1} + \alexsign{2}t^{\alexexp{2}}\SFp{2},\\
\end{array}
\]
and we can show that $\alex{0} = \Falex{0}$ in the same way when $n$ is odd, $\cf{n-1}{} \not\equiv 1/1$, and $a_n$ and $a_{n-1}$ are even. This completes the proof when $n$ is even, and the proof when $p/q \equiv 0/1 \text{ or } 1/0$ is complete.

Next we prove when $p/q \equiv 1/1$. Since $\AT{(q-p)/q}$ is the mirror image of $\AT{p/q}$ and $\fan_i$ in $\AT{p/q}$ corresponds to $\fan_{i+1}$ in $\AT{(q-p)/q}$, $\trisign{i}{}$ of $\AT{(q-p)/q}$ is equal to $\trisign{i}{}$ of $\AT{p/q}$. Let $N = \sum_{i=1}^na_i - 1$ be the number of triangles and $s_{p/q} = \sum_{i=1}^{N}\trisign{i}{}$ the sum of signs of triangles in $\AT{p/q}$.

We first prove following two equations:
\begin{align}
\label{eqn:alexexp}
\alexexp{[1,a_1-1,a_2,\dots,a_n]} + \alexexp{\cf{n}{}} &= -s_{p/q} = -s_{(q-p)/q},\\
\label{eqn:alexsign}
\alexsign{[1,a_1-1,a_2,\dots,a_n]}\alexsign{\cf{n}{}} &= (-1)^{N}.
\end{align}

The proof of (\ref{eqn:alexexp}) is in the same way as above. Since either $\cf{n}{}$ or $[1,a_1-1,a_2,\dots,a_n]$ is equivalent to $0/1$, it is sufficient to prove that (\ref{eqn:alexexp}) holds if $p/q \equiv 1/0 \text{ or } 0/1$. We can show that (\ref{eqn:alexexp}) holds in each cases above. Therefore by induction on $n$ and $a_n$, (\ref{eqn:alexexp}) holds.

We prove (\ref{eqn:alexsign}). Since $q$ is odd, we have
\[
\begin{array}{rcl}
\alexsign{[1,a_1-1,a_2,\dots,a_n]}\alexsign{\cf{n}{}} &=& (-1)^{n+1 + \sum_{i\text{:odd}} a_i - 1 - (q-p)q}(-1)^{n + \sum_{i\text{:even}}a_i - pq}\\
&=& (-1)^{n+1 + \sum_{i\text{:odd}} a_i - 1 - (q-p)q + n + \sum_{i\text{:even}}a_i - pq}\\
&=& (-1)^{\sum_{i}a_i + 2n - q^2}\\
&=& (-1)^{\sum_{i}a_i - 1} = (-1)^{N}.
\end{array}
\]

By Corollary \ref{cor:mirror_F-polynomial}, we have
\[
\SFp{[1,a_1-1,a_2,\dots, a_n]} = (-1)^Nt^{s_{p/q}} ~ \SFp{\cf{n}{}}(t^{-1}).
\]
On the other hand, it is known that Alexander polynomial of a knot is equal to that of its mirror image. Since $\tb((q-p)/q)$ is the mirror image of $\tb(p/q)$, we have $\alex{(q-p)/q}(t) = \alex{p/q}(t)$. By symmetry of Alexander polynomial, we have $\alex{(q-p)/q}(t) = \alex{(q-p)/q}(t^{-1})$. Therefore we get
\[
\begin{array}{rcl}
\Falex{p/q} &=& \alexsign{p/q}t^{\alexexp{p/q}}(-1)^Nt^{s_{p/q}} ~ \SFp{(q-p)/q}(t^{-1})\\
&=& \alexsign{(q-p)/q}t^{-\alexexp{(q-p)/q}} ~ \SFp{(q-p)/q}(t^{-1})\\
&=&\alex{(q-p)/q}(t^{-1})\\
&=& \alex{(q-p)/q}(t) = \alex{p/q}(t),
\end{array}
\]
which complete the proof.
\end{proof}

\begin{rem}
After finding the specialization of F-polynomial, it is possible to prove this
theorem using another recurrence of Alexander polynomial such as \cite{K}
instead of the skein relation.
\end{rem}

\begin{figure}
	\centering
	{\input{example_of_main_27.tex}}
	\caption{$\AT{2/7}$ and its Seifert path.}
	\label{fig:example_of_main27}
\end{figure}
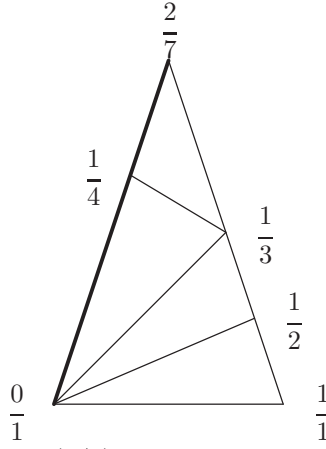

\begin{ex}
The ancestral triangle of $2/7 = [3,2]$ is shown in Figure \ref{fig:example_of_main27}. Therefore $\fan_1$ consists of $\tri_1$ and $\tri_2$ and $\fan_2$ consists of $\tri_3$ and $\tri_4$. By Theorem \ref{thm:cluster_expansion_formula}, we have
\[
\Fp{2/7} = 1 + y_3 + y_2y_3 + y_1y_2y_3 + y_3y_4 + y_2y_3y_4 + y_1y_2y_3y_4
\]
Since the bottom edge of $\fan_1$ is not in the Seifert path and $\fan_1$ is on the right side of the Seifert path, we have
\[
\begin{array}{rcccl}
\trisign{}{(\tri_1^{(1)})} &=& \trisign{}{(\tri_1)} &=& +1;\\
\trisign{}{(\tri_2^{(1)})} &=& \trisign{}{(\tri_2)} &=& -1.
\end{array}
\]
Since the bottom edge of $\fan_2$ is not in the Seifert path and $\fan_2$ is on the right side of the Seifert path, we have
\[
\begin{array}{rcccl}
\trisign{}{(\tri_1^{(2)})} &=& \trisign{}{(\tri_3)} &=& +1;\\
\trisign{}{(\tri_2^{(2)})} &=& \trisign{}{(\tri_4)} &=& +1.
\end{array}
\]
Therefore by substituting $-t$ for $y_1$, $-t^{-1}$ for $y_2$, $-t$ for $y_3$ and $-t$ for $y_4$, we get
\[
\SFp{2/7} = 1 - t + 1 - t + t^2 - t + t^2 = 2 - 3t + t^2.
\]
Next we compute $d$. Since $t_1 = +1$ and $t_2 = t_1 = +1$ by definition, we have $\alexexp{}(1) = 1$ and $\alexexp{}(2) = 2 - 1 = 1$. Therefore we get
\[
\alexexp{} = -\cfrac{1}{2}(\alexexp{}(1) + \alexexp{}(2)) = -1
\]
and
\[
\alexsign{}t^\alexexp{}\SFp{2/7} = (-1)^{2 + 2 - 2 \times 7}t^{-1}(2 - 3t + t^2) = 2t^{-1} - 3 + 2t.
\]
On the other hand, it is known that the Alexander polynomial $\alex{2/7} = 2t - 3 + 2t^{-1}$.
\end{ex}

\begin{ex}
We compute the Alexander polynomial of $\tb(3/5) = \tb([1,1,2])$. By Example \ref{exmp:cluster_mutation}, we have
\[
\Fp{3/5} = 1 + y_1 + y_3 + y_1y_3 + y_1y_2y_3.
\]
According to on the left of Figure \ref{fig:example_of_main}, $\SFp{3/5}$ is obtained by substituting $-t$ for $y_1$, $-t^{-1}$ for $y_2$, and $-t$ for $y_3$, namely
\[
\SFp{3/5} = 1 -3t +t^2.
\]
Since $t_1 = t_2 = -1$ and $t_3 = +1$, we have
$\alexexp{}(1) = 1 - 1 = 0$, $\alexexp{}(2) = 1$, and $\alexexp{}(3) = t_3 = 1$. Terefore we have $\alexexp{} = -(0 + 1 + 1)/2 = -1$ and we get
\[
\begin{array}{rcl}
\Falex{3/5} &=& (-1)^{3 + 1 - 3\times5}t^{-1}(1 -3t +t^2)\\
&=& -t^{-1} + 3 -t,
\end{array}
\]
which is equal to $\alex{3/5}$ as we computed in Example \ref{exmp:alexander_polynomial}. 
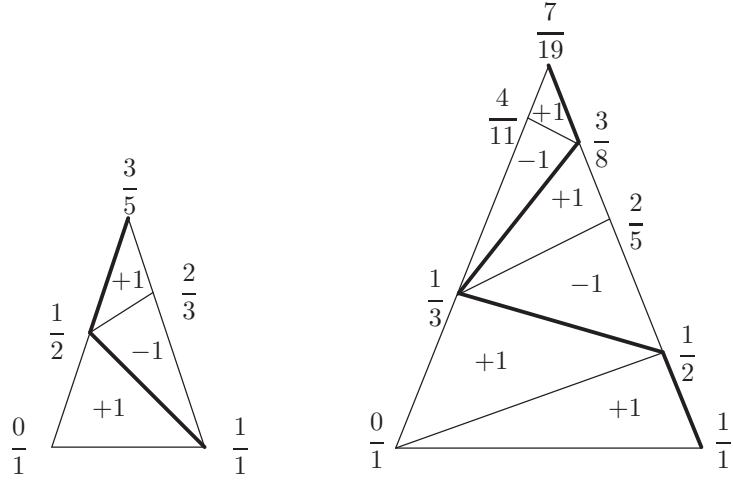
\begin{figure}
	\centering
	{\input{example_of_main5.tex}}
	\caption{$\AT{3/5}$ (left) and $\AT{7/19}$ (right). The bold lines represent the Seifert paths. The number in each triangle is the sign of it.}
	\label{fig:example_of_main}
\end{figure}
\end{ex}

\begin{ex}
By Theorem \ref{thm:cluster_expansion_formula}, we have
\[
\begin{split}
\Fp{7/19} &= 1 +  y_5  +  y_2  +  y_5y_6  +  y_4y_5  +  y_2y_5  +  y_1y_2  +  y_4y_5y_6  +  y_2y_5y_6  +  y_2y_4y_5\\
&\quad  +  y_1y_2y_5  +  y_2y_4y_5y_6  +  y_1y_2y_5y_6  +  y_2y_3y_4y_5  +  y_1y_2y_4y_5  +  y_2y_3y_4y_5y_6\\
&\quad  +  y_1y_2y_4y_5y_6  +  y_1y_2y_3y_4y_5  + y_1y_2y_3y_4y_5y_6.
\end{split}
\]
According to Figure \ref{fig:example_of_main}, $\SFp{7/19}$ is obtained by setting $y_1 = y_2 = y_4 = y_6 = -t$ and $y_3 = y_5 = -t^{-1}$, then we get
\[
\SFp{7/19} = -t^{-1} + 5 -7t + 5t^2 -t^3.
\]
Since $t_1 = t_2 = t_4 = -1$ and $t_3 = +1$ by Example \ref{exmp:sign_of_crossing}, we have $\alexexp{}(1) = 2 - 1 = 1$, $\alexexp{}(2) = 1$, $\alexexp{}(3) = t_3 = +1$, and $\alexexp{}(4) = t_4 = -1$. Therefore $d = -(1 + 1 + 1 - 1)/2 = -1$ and we get
\[
\begin{array}{rcl}
\Falex{7/19} &=& (-1)^{4 + 1 + 2 - 7\times19}t^{-1}(-t^{-1} + 5 -7t + 5t^2 -t^3)\\
&=& -t^{-2} + 5t^{-1} -7 + 5t - t^2.
\end{array}
\]
On the other hand, since it is known that $\alex{7/19} = -t^{-2} + 5t^{-1} -7 + 5t - t^2$, we have $\Falex{7/19} = \alex{7/19}$.
\end{ex}

{}
\vspace{2cm}
\end{document}

%% file: right_and_left_triangle.tex
{\unitlength 0.1in%
\begin{picture}(38.0000,7.3500)(6.0000,-13.3500)%
\put(12.0000,-14.0000){\makebox(0,0){{\color[named]{Black}{left triagnle}}}}%
\put(38.0000,-14.0000){\makebox(0,0){{\color[named]{Black}{right triangle}}}}%
%
{\color[named]{Black}{%
\special{pn 8}%
\special{pa 600 1200}%
\special{pa 1800 1000}%
\special{fp}%
\special{pa 1800 1000}%
\special{pa 800 600}%
\special{fp}%
\special{pa 800 600}%
\special{pa 600 1200}%
\special{fp}%
}}%
%
{\color[named]{Black}{%
\special{pn 8}%
\special{pa 4400 1200}%
\special{pa 3200 1000}%
\special{fp}%
\special{pa 3200 1000}%
\special{pa 4200 600}%
\special{fp}%
\special{pa 4200 600}%
\special{pa 4400 1200}%
\special{fp}%
}}%
\end{picture}}%

%% file: example_of_ancestral_triangle_construction.tex
{\unitlength 0.1in%
\begin{picture}(41.6000,22.6500)(8.4000,-26.0500)%
%
{\color[named]{Black}{%
\special{pn 8}%
\special{pa 1960 1015}%
\special{pa 1690 875}%
\special{fp}%
}}%
%
{\color[named]{Black}{%
\special{pn 8}%
\special{pa 1320 1805}%
\special{pa 2120 1405}%
\special{fp}%
}}%
%
{\color[named]{Black}{%
\special{pn 20}%
\special{pa 1800 605}%
\special{pa 1000 2605}%
\special{fp}%
}}%
%
{\color[named]{Black}{%
\special{pn 20}%
\special{pa 2600 2605}%
\special{pa 1800 605}%
\special{fp}%
}}%
\put(27.5000,-23.9500){\makebox(0,0)[rt]{{\color[named]{Black}{$\cfrac{1}{1}$}}}}%
\put(8.5000,-23.9500){\makebox(0,0)[lt]{{\color[named]{Black}{$\cfrac{0}{1}$}}}}%
\put(25.2000,-20.9500){\makebox(0,0){{\color[named]{Black}{$\cfrac{1}{2}$}}}}%
\put(12.0000,-18.0500){\makebox(0,0){{\color[named]{Black}{$\cfrac{1}{3}$}}}}%
\put(22.5000,-14.0500){\makebox(0,0){{\color[named]{Black}{$\cfrac{2}{5}$}}}}%
\put(20.7000,-9.7500){\makebox(0,0){{\color[named]{Black}{$\cfrac{3}{8}$}}}}%
\put(15.5000,-8.5500){\makebox(0,0){{\color[named]{Black}{$\cfrac{4}{11}$}}}}%
\put(18.0000,-4.0500){\makebox(0,0){{\color[named]{Black}{$\cfrac{7}{19}$}}}}%
\put(21.0000,-24.5500){\makebox(0,0){{\color[named]{Black}{$\fan_1$}}}}%
\put(15.0000,-21.4500){\makebox(0,0){{\color[named]{Black}{$\fan_2$}}}}%
\put(18.3000,-17.0500){\makebox(0,0){{\color[named]{Black}{$\fan_3$}}}}%
\put(17.8000,-10.5500){\makebox(0,0){{\color[named]{Black}{$\fan_4$}}}}%
%
{\color[named]{Black}{%
\special{pn 20}%
\special{pa 1000 2605}%
\special{pa 2600 2605}%
\special{fp}%
}}%
%
{\color[named]{Black}{%
\special{pn 20}%
\special{pa 1000 2605}%
\special{pa 2400 2105}%
\special{fp}%
}}%
%
{\color[named]{Black}{%
\special{pn 20}%
\special{pa 2400 2105}%
\special{pa 1330 1795}%
\special{fp}%
}}%
%
{\color[named]{Black}{%
\special{pn 20}%
\special{pa 1330 1795}%
\special{pa 1960 995}%
\special{fp}%
}}%
%
{\color[named]{Black}{%
\special{pn 20}%
\special{pa 5000 2600}%
\special{pa 4400 1400}%
\special{fp}%
\special{pa 4400 1400}%
\special{pa 3800 2600}%
\special{fp}%
\special{pa 3800 2600}%
\special{pa 5000 2600}%
\special{fp}%
\special{pa 5000 2600}%
\special{pa 4000 2200}%
\special{fp}%
}}%
%
{\color[named]{Black}{%
\special{pn 8}%
\special{pa 4000 2200}%
\special{pa 4600 1800}%
\special{fp}%
}}%
\put(44.0000,-12.0000){\makebox(0,0){{\color[named]{Black}{$\cfrac{3}{5}$}}}}%
\put(39.0000,-21.3000){\makebox(0,0){{\color[named]{Black}{$\cfrac{1}{2}$}}}}%
\put(47.0000,-17.5000){\makebox(0,0){{\color[named]{Black}{$\cfrac{2}{3}$}}}}%
\put(37.0000,-25.7000){\makebox(0,0){{\color[named]{Black}{$\cfrac{0}{1}$}}}}%
\put(51.0000,-25.7000){\makebox(0,0){{\color[named]{Black}{$\cfrac{1}{1}$}}}}%
\put(44.0000,-20.6000){\makebox(0,0){{\color[named]{Black}{$\fan_3$}}}}%
\put(41.0000,-24.3000){\makebox(0,0){{\color[named]{Black}{$\fan_2$}}}}%
\end{picture}}%

%% file: example_of_path.tex
{\unitlength 0.1in%
\begin{picture}(41.6000,22.7000)(8.4000,-24.0000)%
%
{\color[named]{Black}{%
\special{pn 8}%
\special{pa 1960 805}%
\special{pa 1690 665}%
\special{fp}%
}}%
%
{\color[named]{Black}{%
\special{pn 8}%
\special{pa 1320 1595}%
\special{pa 2120 1195}%
\special{fp}%
}}%
%
{\color[named]{Black}{%
\special{pn 8}%
\special{pa 1800 395}%
\special{pa 1000 2395}%
\special{fp}%
}}%
\put(27.5000,-21.8500){\makebox(0,0)[rt]{{\color[named]{Black}{$\cfrac{1}{1}$}}}}%
\put(8.5000,-21.8500){\makebox(0,0)[lt]{{\color[named]{Black}{$\cfrac{0}{1}$}}}}%
\put(25.2000,-18.8500){\makebox(0,0){{\color[named]{Black}{$\cfrac{1}{2}$}}}}%
\put(12.0000,-15.9500){\makebox(0,0){{\color[named]{Black}{$\cfrac{1}{3}$}}}}%
\put(22.5000,-11.9500){\makebox(0,0){{\color[named]{Black}{$\cfrac{2}{5}$}}}}%
\put(20.7000,-7.6500){\makebox(0,0){{\color[named]{Black}{$\cfrac{3}{8}$}}}}%
\put(18.0000,-1.9500){\makebox(0,0){{\color[named]{Black}{$\cfrac{7}{19}$}}}}%
%
{\color[named]{Black}{%
\special{pn 8}%
\special{pa 1000 2395}%
\special{pa 2600 2395}%
\special{fp}%
}}%
%
{\color[named]{Black}{%
\special{pn 8}%
\special{pa 1000 2395}%
\special{pa 2400 1895}%
\special{fp}%
}}%
%
{\color[named]{Black}{%
\special{pn 8}%
\special{pa 2400 1895}%
\special{pa 1330 1585}%
\special{fp}%
}}%
%
{\color[named]{Black}{%
\special{pn 8}%
\special{pa 1330 1585}%
\special{pa 1960 785}%
\special{fp}%
}}%
%
{\color[named]{Black}{%
\special{pn 8}%
\special{pa 4360 805}%
\special{pa 4090 665}%
\special{fp}%
}}%
%
{\color[named]{Black}{%
\special{pn 8}%
\special{pa 3720 1595}%
\special{pa 4520 1195}%
\special{fp}%
}}%
%
{\color[named]{Black}{%
\special{pn 8}%
\special{pa 4200 395}%
\special{pa 3400 2395}%
\special{fp}%
}}%
%
{\color[named]{Black}{%
\special{pn 8}%
\special{pa 5000 2395}%
\special{pa 4200 395}%
\special{fp}%
}}%
\put(51.5000,-21.8500){\makebox(0,0)[rt]{{\color[named]{Black}{$\cfrac{1}{1}$}}}}%
\put(32.5000,-21.8500){\makebox(0,0)[lt]{{\color[named]{Black}{$\cfrac{0}{1}$}}}}%
\put(49.2000,-18.8500){\makebox(0,0){{\color[named]{Black}{$\cfrac{1}{2}$}}}}%
\put(36.0000,-15.9500){\makebox(0,0){{\color[named]{Black}{$\cfrac{1}{3}$}}}}%
\put(46.5000,-11.9500){\makebox(0,0){{\color[named]{Black}{$\cfrac{2}{5}$}}}}%
\put(44.7000,-7.6500){\makebox(0,0){{\color[named]{Black}{$\cfrac{3}{8}$}}}}%
\put(39.5000,-6.4500){\makebox(0,0){{\color[named]{Black}{$\cfrac{4}{11}$}}}}%
\put(42.0000,-1.9500){\makebox(0,0){{\color[named]{Black}{$\cfrac{7}{19}$}}}}%
%
{\color[named]{Black}{%
\special{pn 8}%
\special{pa 3400 2395}%
\special{pa 5000 2395}%
\special{fp}%
}}%
%
{\color[named]{Black}{%
\special{pn 8}%
\special{pa 3400 2395}%
\special{pa 4800 1895}%
\special{fp}%
}}%
%
{\color[named]{Black}{%
\special{pn 8}%
\special{pa 4800 1895}%
\special{pa 3730 1585}%
\special{fp}%
}}%
%
{\color[named]{Black}{%
\special{pn 8}%
\special{pa 3730 1585}%
\special{pa 4360 785}%
\special{fp}%
}}%
%
{\color[named]{Black}{%
\special{pn 20}%
\special{pa 4200 400}%
\special{pa 4802 1899}%
\special{fp}%
\special{sh 1}%
\special{pa 4802 1899}%
\special{pa 4796 1830}%
\special{pa 4782 1850}%
\special{pa 4759 1845}%
\special{pa 4802 1899}%
\special{fp}%
\special{pa 4802 1899}%
\special{pa 3730 1584}%
\special{fp}%
\special{sh 1}%
\special{pa 3730 1584}%
\special{pa 3788 1622}%
\special{pa 3781 1599}%
\special{pa 3800 1584}%
\special{pa 3730 1584}%
\special{fp}%
\special{pa 3730 1584}%
\special{pa 3400 2400}%
\special{fp}%
\special{sh 1}%
\special{pa 3400 2400}%
\special{pa 3444 2346}%
\special{pa 3420 2351}%
\special{pa 3406 2331}%
\special{pa 3400 2400}%
\special{fp}%
}}%
%
{\color[named]{Black}{%
\special{pn 8}%
\special{pa 1800 400}%
\special{pa 2600 2400}%
\special{fp}%
}}%
%
{\color[named]{Black}{%
\special{pn 20}%
\special{pa 1800 400}%
\special{pa 1693 663}%
\special{fp}%
\special{sh 1}%
\special{pa 1693 663}%
\special{pa 1737 609}%
\special{pa 1713 614}%
\special{pa 1700 594}%
\special{pa 1693 663}%
\special{fp}%
\special{pa 1693 663}%
\special{pa 1955 799}%
\special{fp}%
\special{sh 1}%
\special{pa 1955 799}%
\special{pa 1905 751}%
\special{pa 1908 774}%
\special{pa 1887 786}%
\special{pa 1955 799}%
\special{fp}%
\special{pa 1955 799}%
\special{pa 1329 1585}%
\special{fp}%
\special{sh 1}%
\special{pa 1329 1585}%
\special{pa 1386 1545}%
\special{pa 1362 1543}%
\special{pa 1355 1520}%
\special{pa 1329 1585}%
\special{fp}%
\special{pa 1329 1584}%
\special{pa 2395 1892}%
\special{fp}%
\special{sh 1}%
\special{pa 2395 1892}%
\special{pa 2337 1854}%
\special{pa 2344 1877}%
\special{pa 2325 1893}%
\special{pa 2395 1892}%
\special{fp}%
\special{pa 2395 1892}%
\special{pa 2598 2394}%
\special{fp}%
\special{sh 1}%
\special{pa 2598 2394}%
\special{pa 2592 2325}%
\special{pa 2578 2345}%
\special{pa 2554 2340}%
\special{pa 2598 2394}%
\special{fp}%
\special{pa 2598 2394}%
\special{pa 2598 2394}%
\special{fp}%
}}%
\end{picture}}%

%% file: cluster_expansion_formula.tex
{\unitlength 0.1in%
\begin{picture}(51.1000,45.3000)(0.9000,-47.9500)%
%
{\color[named]{Black}{%
\special{pn 8}%
\special{pa 1000 400}%
\special{pa 600 2400}%
\special{fp}%
\special{pa 1800 400}%
\special{pa 2200 2400}%
\special{fp}%
}}%
\put(19.0000,-3.3000){\makebox(0,0){{\color[named]{Black}{$\cfrac{p_i}{q_i}$}}}}%
\put(20.4000,-7.7000){\makebox(0,0){{\color[named]{Black}{$\cfrac{p_{i-1}}{q_{i-1}}$}}}}%
\put(15.7000,-11.9000){\makebox(0,0){{\color[named]{Black}{$\vdots$}}}}%
\put(13.4000,-6.2000){\makebox(0,0){{\color[named]{Black}{$\vdots$}}}}%
\put(3.5000,-16.0000){\makebox(0,0)[lb]{{\color[named]{Black}{$\cfrac{p_{l_{2k} - 1}}{q_{l_{2k} - 1}}$}}}}%
\put(16.5000,-7.7000){\makebox(0,0){{\color[named]{Black}{$\tri_i$}}}}%
\put(10.7000,-20.4000){\makebox(0,0){{\color[named]{Black}{$\tri_{l_{2k} - 1}$}}}}%
%
{\color[named]{Black}{%
\special{pn 8}%
\special{pa 4000 400}%
\special{pa 3600 2400}%
\special{fp}%
\special{pa 4800 400}%
\special{pa 5200 2400}%
\special{fp}%
}}%
%
{\color[named]{Black}{%
\special{pn 8}%
\special{pa 1800 400}%
\special{pa 760 1600}%
\special{fp}%
\special{pa 760 1600}%
\special{pa 1880 800}%
\special{fp}%
\special{pa 760 1600}%
\special{pa 1960 1200}%
\special{fp}%
\special{pa 760 1600}%
\special{pa 2160 2200}%
\special{fp}%
\special{pa 2160 2200}%
\special{pa 600 2400}%
\special{fp}%
}}%
\put(15.9000,-16.7000){\makebox(0,0){{\color[named]{Black}{$\tri_{l_{2k}}$}}}}%
%
{\color[named]{Black}{%
\special{pn 8}%
\special{pa 3680 2000}%
\special{pa 5200 2400}%
\special{fp}%
\special{pa 3680 2000}%
\special{pa 5040 1600}%
\special{fp}%
\special{pa 4960 1200}%
\special{pa 3680 2000}%
\special{fp}%
\special{pa 3680 2000}%
\special{pa 4880 800}%
\special{fp}%
\special{pa 4880 800}%
\special{pa 4000 400}%
\special{fp}%
}}%
\put(38.7000,-3.3000){\makebox(0,0){{\color[named]{Black}{$\cfrac{p_i}{q_i}$}}}}%
\put(44.9000,-4.5000){\makebox(0,0){{\color[named]{Black}{$\vdots$}}}}%
\put(46.7000,-15.5000){\makebox(0,0){{\color[named]{Black}{$\vdots$}}}}%
\put(41.9000,-9.7000){\makebox(0,0){{\color[named]{Black}{$\tri_i$}}}}%
\put(32.8000,-20.0000){\makebox(0,0)[lb]{{\color[named]{Black}{$\cfrac{p_{l_{2k} - 1}}{q_{l_{2k} - 1}}$}}}}%
\put(50.4000,-7.7000){\makebox(0,0){{\color[named]{Black}{$\cfrac{p_{i-1}}{q_{i-1}}$}}}}%
\put(39.3000,-22.7000){\makebox(0,0){{\color[named]{Black}{$\tri_{l_{2k} - 1}$}}}}%
\put(47.1000,-12.0000){\makebox(0,0){{\color[named]{Black}{$\tri_{i-1}$}}}}%
\put(48.0000,-20.0000){\makebox(0,0){{\color[named]{Black}{$\tri_{l_{2k}}$}}}}%
%
{\color[named]{Black}{%
\special{pn 8}%
\special{pa 1780 2795}%
\special{pa 2180 4795}%
\special{fp}%
\special{pa 980 2795}%
\special{pa 580 4795}%
\special{fp}%
}}%
\put(8.8000,-27.2500){\makebox(0,0){{\color[named]{Black}{$\cfrac{p_i}{q_i}$}}}}%
\put(7.4000,-31.6500){\makebox(0,0){{\color[named]{Black}{$\cfrac{p_{i-1}}{q_{i-1}}$}}}}%
\put(12.1000,-35.8500){\makebox(0,0){{\color[named]{Black}{$\vdots$}}}}%
\put(14.4000,-30.1500){\makebox(0,0){{\color[named]{Black}{$\vdots$}}}}%
\put(25.0000,-39.9000){\makebox(0,0)[rb]{{\color[named]{Black}{$\cfrac{p_{l_{2k+1} - 1}}{q_{l_{2k+1} - 1}}$}}}}%
\put(11.3000,-31.6500){\makebox(0,0){{\color[named]{Black}{$\tri_i$}}}}%
\put(17.1000,-44.3500){\makebox(0,0){{\color[named]{Black}{$\tri_{l_{2k+1} - 1}$}}}}%
%
{\color[named]{Black}{%
\special{pn 8}%
\special{pa 980 2795}%
\special{pa 2020 3995}%
\special{fp}%
\special{pa 2020 3995}%
\special{pa 900 3195}%
\special{fp}%
\special{pa 2020 3995}%
\special{pa 820 3595}%
\special{fp}%
\special{pa 2020 3995}%
\special{pa 620 4595}%
\special{fp}%
\special{pa 620 4595}%
\special{pa 2180 4795}%
\special{fp}%
}}%
\put(11.9000,-40.6500){\makebox(0,0){{\color[named]{Black}{$\tri_{l_{2k+1}}$}}}}%
%
{\color[named]{Black}{%
\special{pn 8}%
\special{pa 4780 2795}%
\special{pa 5180 4795}%
\special{fp}%
\special{pa 3980 2795}%
\special{pa 3580 4795}%
\special{fp}%
}}%
%
{\color[named]{Black}{%
\special{pn 8}%
\special{pa 5100 4395}%
\special{pa 3580 4795}%
\special{fp}%
\special{pa 5100 4395}%
\special{pa 3740 3995}%
\special{fp}%
\special{pa 3820 3595}%
\special{pa 5100 4395}%
\special{fp}%
\special{pa 5100 4395}%
\special{pa 3900 3195}%
\special{fp}%
\special{pa 3900 3195}%
\special{pa 4780 2795}%
\special{fp}%
}}%
\put(49.1000,-27.2500){\makebox(0,0){{\color[named]{Black}{$\cfrac{p_i}{q_i}$}}}}%
\put(42.9000,-28.4500){\makebox(0,0){{\color[named]{Black}{$\vdots$}}}}%
\put(41.1000,-39.4500){\makebox(0,0){{\color[named]{Black}{$\vdots$}}}}%
\put(45.9000,-33.6500){\makebox(0,0){{\color[named]{Black}{$\tri_i$}}}}%
\put(56.0000,-43.9000){\makebox(0,0)[rb]{{\color[named]{Black}{$\cfrac{p_{l_{2k-1} - 1}}{q_{l_{2k-1} - 1}}$}}}}%
\put(37.4000,-31.6500){\makebox(0,0){{\color[named]{Black}{$\cfrac{p_{i-1}}{q_{i-1}}$}}}}%
\put(48.5000,-46.6500){\makebox(0,0){{\color[named]{Black}{$\tri_{l_{2k-1} - 1}$}}}}%
\put(40.7000,-35.9500){\makebox(0,0){{\color[named]{Black}{$\tri_{i-1}$}}}}%
\put(39.8000,-43.9500){\makebox(0,0){{\color[named]{Black}{$\tri_{l_{2k-1}}$}}}}%
\end{picture}}%

%% file: two-bridge_knot.tex
{\unitlength 0.1in%
\begin{picture}(46.9000,30.0000)(5.1000,-34.0000)%
%
{\color[named]{Black}{%
\special{pn 4}%
\special{ar 1400 1465 200 200 1.5707963 4.7123890}%
}}%
%
{\color[named]{Black}{%
\special{pn 4}%
\special{ar 1400 665 200 200 1.5707963 4.7123890}%
}}%
%
{\color[named]{Black}{%
\special{pn 20}%
\special{pa 1400 865}%
\special{pa 1800 865}%
\special{pa 1800 1265}%
\special{pa 1400 1265}%
\special{pa 1400 865}%
\special{pa 1800 865}%
\special{fp}%
}}%
%
{\color[named]{Black}{%
\special{pn 20}%
\special{pa 2200 465}%
\special{pa 2600 465}%
\special{pa 2600 865}%
\special{pa 2200 865}%
\special{pa 2200 465}%
\special{pa 2600 465}%
\special{fp}%
}}%
%
{\color[named]{Black}{%
\special{pn 20}%
\special{pa 3000 865}%
\special{pa 3400 865}%
\special{pa 3400 1265}%
\special{pa 3000 1265}%
\special{pa 3000 865}%
\special{pa 3400 865}%
\special{fp}%
}}%
\put(16.0000,-10.6500){\makebox(0,0){{\color[named]{Black}{$-a_1$}}}}%
\put(24.0000,-6.6500){\makebox(0,0){{\color[named]{Black}{$a_2$}}}}%
\put(32.0000,-10.6500){\makebox(0,0){{\color[named]{Black}{$-a_3$}}}}%
\put(38.0000,-4.6500){\makebox(0,0){{\color[named]{Black}{$\cdots$}}}}%
\put(38.0000,-8.6500){\makebox(0,0){{\color[named]{Black}{$\cdots$}}}}%
\put(38.0000,-12.6500){\makebox(0,0){{\color[named]{Black}{$\cdots$}}}}%
\put(38.0000,-16.6500){\makebox(0,0){{\color[named]{Black}{$\cdots$}}}}%
%
{\color[named]{Black}{%
\special{pn 4}%
\special{pa 1400 465}%
\special{pa 2200 465}%
\special{fp}%
\special{pa 2600 465}%
\special{pa 3400 465}%
\special{fp}%
\special{pa 3000 865}%
\special{pa 2600 865}%
\special{fp}%
\special{pa 2200 865}%
\special{pa 1800 865}%
\special{fp}%
\special{pa 3400 865}%
\special{pa 3600 865}%
\special{fp}%
\special{pa 3400 465}%
\special{pa 3600 465}%
\special{fp}%
\special{pa 3600 1265}%
\special{pa 3400 1265}%
\special{fp}%
\special{pa 3000 1265}%
\special{pa 1800 1265}%
\special{fp}%
\special{pa 1400 1665}%
\special{pa 3600 1665}%
\special{fp}%
\special{pa 4000 1665}%
\special{pa 4200 1665}%
\special{fp}%
\special{pa 4000 1265}%
\special{pa 4200 1265}%
\special{fp}%
\special{pa 4000 865}%
\special{pa 4200 865}%
\special{fp}%
\special{pa 4000 465}%
\special{pa 4200 465}%
\special{fp}%
}}%
%
{\color[named]{Black}{%
\special{pn 20}%
\special{pa 4200 865}%
\special{pa 4600 865}%
\special{pa 4600 1265}%
\special{pa 4200 1265}%
\special{pa 4200 865}%
\special{pa 4600 865}%
\special{fp}%
}}%
\put(44.0000,-10.6500){\makebox(0,0){{\color[named]{Black}{$-a_n$}}}}%
%
{\color[named]{Black}{%
\special{pn 4}%
\special{ar 4600 665 200 200 4.7123890 1.5707963}%
}}%
%
{\color[named]{Black}{%
\special{pn 4}%
\special{ar 4600 1465 200 200 4.7123890 1.5707963}%
}}%
%
{\color[named]{Black}{%
\special{pn 4}%
\special{pa 4200 465}%
\special{pa 4600 465}%
\special{fp}%
\special{pa 4600 1665}%
\special{pa 4200 1665}%
\special{fp}%
}}%
\put(8.0000,-10.0000){\makebox(0,0){{\color[named]{Black}{$n$ is odd.}}}}%
%
{\color[named]{Black}{%
\special{pn 20}%
\special{pa 2200 2200}%
\special{pa 2600 2200}%
\special{pa 2600 2600}%
\special{pa 2200 2600}%
\special{pa 2200 2200}%
\special{pa 2600 2200}%
\special{fp}%
}}%
%
{\color[named]{Black}{%
\special{pn 20}%
\special{pa 1400 2600}%
\special{pa 1800 2600}%
\special{pa 1800 3000}%
\special{pa 1400 3000}%
\special{pa 1400 2600}%
\special{pa 1800 2600}%
\special{fp}%
}}%
%
{\color[named]{Black}{%
\special{pn 20}%
\special{pa 3000 3000}%
\special{pa 3400 3000}%
\special{pa 3400 2600}%
\special{pa 3000 2600}%
\special{pa 3000 3000}%
\special{pa 3400 3000}%
\special{fp}%
}}%
%
{\color[named]{Black}{%
\special{pn 20}%
\special{pa 4200 2190}%
\special{pa 4600 2190}%
\special{pa 4600 2590}%
\special{pa 4200 2590}%
\special{pa 4200 2190}%
\special{pa 4600 2190}%
\special{fp}%
}}%
%
{\color[named]{Black}{%
\special{pn 4}%
\special{pa 1800 2600}%
\special{pa 2200 2600}%
\special{fp}%
\special{pa 2600 2600}%
\special{pa 3000 2600}%
\special{fp}%
\special{pa 3400 2600}%
\special{pa 3600 2600}%
\special{fp}%
\special{pa 4000 2600}%
\special{pa 4200 2600}%
\special{fp}%
\special{pa 4200 2200}%
\special{pa 4000 2200}%
\special{fp}%
\special{pa 3600 2200}%
\special{pa 2600 2200}%
\special{fp}%
\special{pa 2200 2200}%
\special{pa 1400 2200}%
\special{fp}%
\special{pa 1800 3000}%
\special{pa 3000 3000}%
\special{fp}%
\special{pa 3400 3000}%
\special{pa 3600 3000}%
\special{fp}%
\special{pa 1400 3400}%
\special{pa 3600 3400}%
\special{fp}%
\special{pa 4000 3400}%
\special{pa 4200 3400}%
\special{fp}%
\special{pa 4000 3000}%
\special{pa 4200 3000}%
\special{fp}%
\special{pa 4200 3000}%
\special{pa 4600 3000}%
\special{fp}%
\special{pa 4200 3400}%
\special{pa 4600 3400}%
\special{fp}%
}}%
%
{\color[named]{Black}{%
\special{pn 4}%
\special{ar 1400 3200 200 200 1.5707963 4.7123890}%
}}%
%
{\color[named]{Black}{%
\special{pn 4}%
\special{ar 1400 2400 200 200 1.5707963 4.7123890}%
}}%
%
{\color[named]{Black}{%
\special{pn 4}%
\special{ar 4600 2800 200 200 4.7123890 1.5707963}%
}}%
%
{\color[named]{Black}{%
\special{pn 4}%
\special{ar 4600 2800 600 600 4.7123890 1.5707963}%
}}%
\put(38.0000,-22.0000){\makebox(0,0){{\color[named]{Black}{$\cdots$}}}}%
\put(38.0000,-26.0000){\makebox(0,0){{\color[named]{Black}{$\cdots$}}}}%
\put(38.0000,-30.0000){\makebox(0,0){{\color[named]{Black}{$\cdots$}}}}%
\put(38.0000,-34.0000){\makebox(0,0){{\color[named]{Black}{$\cdots$}}}}%
\put(16.0000,-28.1000){\makebox(0,0){{\color[named]{Black}{$-a_1$}}}}%
\put(24.0000,-24.0000){\makebox(0,0){{\color[named]{Black}{$a_2$}}}}%
\put(32.0000,-28.1000){\makebox(0,0){{\color[named]{Black}{$-a_3$}}}}%
\put(44.0000,-24.0000){\makebox(0,0){{\color[named]{Black}{$a_n$}}}}%
\put(8.0000,-28.0000){\makebox(0,0){{\color[named]{Black}{$n$ is even.}}}}%
\end{picture}}%

%% file: example_of_two-bridge_knot.tex
{\unitlength 0.1in%
\begin{picture}(42.0000,12.0000)(12.0000,-26.0000)%
%
{\color[named]{Black}{%
\special{pn 8}%
\special{pa 2400 1400}%
\special{pa 2800 1800}%
\special{fp}%
}}%
%
{\color[named]{Black}{%
\special{pn 8}%
\special{pa 2800 1800}%
\special{pa 3000 1800}%
\special{fp}%
}}%
%
{\color[named]{Black}{%
\special{pn 8}%
\special{pa 1400 1400}%
\special{pa 2400 1400}%
\special{fp}%
}}%
%
{\color[named]{Black}{%
\special{pn 8}%
\special{pa 2200 1800}%
\special{pa 2400 1800}%
\special{fp}%
}}%
%
{\color[named]{Black}{%
\special{pn 8}%
\special{pa 2650 1550}%
\special{pa 2800 1400}%
\special{fp}%
}}%
%
{\color[named]{Black}{%
\special{pn 8}%
\special{pa 2400 1800}%
\special{pa 2550 1650}%
\special{fp}%
}}%
%
{\color[named]{Black}{%
\special{pn 8}%
\special{pa 3800 1800}%
\special{pa 4000 1800}%
\special{fp}%
}}%
%
{\color[named]{Black}{%
\special{pn 8}%
\special{pa 2200 2200}%
\special{pa 3000 2200}%
\special{fp}%
}}%
%
{\color[named]{Black}{%
\special{pn 8}%
\special{pa 1400 2600}%
\special{pa 4800 2600}%
\special{fp}%
}}%
%
{\color[named]{Black}{%
\special{pn 8}%
\special{pa 4000 1400}%
\special{pa 2800 1400}%
\special{fp}%
}}%
%
{\color[named]{Black}{%
\special{pn 8}%
\special{pa 3800 2200}%
\special{pa 4800 2200}%
\special{fp}%
}}%
%
{\color[named]{Black}{%
\special{pn 8}%
\special{ar 1400 2400 200 200 1.5707963 4.7123890}%
}}%
%
{\color[named]{Black}{%
\special{pn 8}%
\special{ar 1400 1600 200 200 1.5707963 4.7123890}%
}}%
%
{\color[named]{Black}{%
\special{pn 8}%
\special{ar 4800 2000 200 200 4.7123890 1.5707963}%
}}%
%
{\color[named]{Black}{%
\special{pn 8}%
\special{ar 4800 2000 600 600 4.7123890 1.5707963}%
}}%
%
{\color[named]{Black}{%
\special{pn 8}%
\special{pa 1800 1800}%
\special{pa 1950 1950}%
\special{fp}%
\special{pa 2200 2200}%
\special{pa 2050 2050}%
\special{fp}%
\special{pa 1800 2200}%
\special{pa 2200 1800}%
\special{fp}%
}}%
%
{\color[named]{Black}{%
\special{pn 8}%
\special{pa 1400 1800}%
\special{pa 1550 1950}%
\special{fp}%
\special{pa 1800 2200}%
\special{pa 1650 2050}%
\special{fp}%
\special{pa 1400 2200}%
\special{pa 1800 1800}%
\special{fp}%
}}%
%
{\color[named]{Black}{%
\special{pn 8}%
\special{pa 3400 1800}%
\special{pa 3550 1950}%
\special{fp}%
\special{pa 3800 2200}%
\special{pa 3650 2050}%
\special{fp}%
\special{pa 3400 2200}%
\special{pa 3800 1800}%
\special{fp}%
}}%
%
{\color[named]{Black}{%
\special{pn 8}%
\special{pa 3000 1800}%
\special{pa 3150 1950}%
\special{fp}%
\special{pa 3400 2200}%
\special{pa 3250 2050}%
\special{fp}%
\special{pa 3000 2200}%
\special{pa 3400 1800}%
\special{fp}%
}}%
%
{\color[named]{Black}{%
\special{pn 8}%
\special{pa 4000 1400}%
\special{pa 4400 1800}%
\special{fp}%
}}%
%
{\color[named]{Black}{%
\special{pn 8}%
\special{pa 4250 1550}%
\special{pa 4400 1400}%
\special{fp}%
}}%
%
{\color[named]{Black}{%
\special{pn 8}%
\special{pa 4000 1800}%
\special{pa 4150 1650}%
\special{fp}%
}}%
%
{\color[named]{Black}{%
\special{pn 8}%
\special{pa 4400 1400}%
\special{pa 4800 1800}%
\special{fp}%
}}%
%
{\color[named]{Black}{%
\special{pn 8}%
\special{pa 4650 1550}%
\special{pa 4800 1400}%
\special{fp}%
}}%
%
{\color[named]{Black}{%
\special{pn 8}%
\special{pa 4400 1800}%
\special{pa 4550 1650}%
\special{fp}%
}}%
\end{picture}}%

%% file: example_of_oriented_two-bridge_knot.tex
{\unitlength 0.1in%
\begin{picture}(42.0000,12.9300)(12.0000,-26.2700)%
%
{\color[named]{Black}{%
\special{pn 8}%
\special{pa 2400 1800}%
\special{pa 2550 1650}%
\special{fp}%
}}%
%
{\color[named]{Black}{%
\special{pn 8}%
\special{ar 1400 2400 200 200 1.5707963 4.7123890}%
}}%
%
{\color[named]{Black}{%
\special{pn 8}%
\special{pa 1385 2201}%
\special{pa 1400 2200}%
\special{fp}%
\special{sh 1}%
\special{pa 1400 2200}%
\special{pa 1332 2184}%
\special{pa 1347 2204}%
\special{pa 1335 2224}%
\special{pa 1400 2200}%
\special{fp}%
}}%
%
{\color[named]{Black}{%
\special{pn 8}%
\special{ar 1400 1600 200 200 1.5707963 4.7123890}%
}}%
%
{\color[named]{Black}{%
\special{pn 8}%
\special{pa 1385 1799}%
\special{pa 1400 1800}%
\special{fp}%
\special{sh 1}%
\special{pa 1400 1800}%
\special{pa 1335 1776}%
\special{pa 1347 1796}%
\special{pa 1332 1816}%
\special{pa 1400 1800}%
\special{fp}%
}}%
%
{\color[named]{Black}{%
\special{pn 8}%
\special{ar 4800 2000 200 200 4.7123890 1.5707963}%
}}%
%
{\color[named]{Black}{%
\special{pn 8}%
\special{pa 4815 1801}%
\special{pa 4800 1800}%
\special{fp}%
\special{sh 1}%
\special{pa 4800 1800}%
\special{pa 4865 1824}%
\special{pa 4853 1804}%
\special{pa 4868 1784}%
\special{pa 4800 1800}%
\special{fp}%
}}%
%
{\color[named]{Black}{%
\special{pn 8}%
\special{ar 4800 2000 600 600 4.7123890 1.5707963}%
}}%
%
{\color[named]{Black}{%
\special{pn 8}%
\special{pa 4815 2600}%
\special{pa 4800 2600}%
\special{fp}%
\special{sh 1}%
\special{pa 4800 2600}%
\special{pa 4867 2620}%
\special{pa 4853 2600}%
\special{pa 4867 2580}%
\special{pa 4800 2600}%
\special{fp}%
}}%
%
{\color[named]{Black}{%
\special{pn 8}%
\special{pa 3400 1800}%
\special{pa 3550 1950}%
\special{fp}%
}}%
%
{\color[named]{Black}{%
\special{pn 8}%
\special{pa 3400 2200}%
\special{pa 3250 2050}%
\special{fp}%
}}%
%
{\color[named]{Black}{%
\special{pn 8}%
\special{pa 4250 1550}%
\special{pa 4400 1400}%
\special{fp}%
}}%
%
{\color[named]{Black}{%
\special{pn 8}%
\special{pa 4400 1800}%
\special{pa 4550 1650}%
\special{fp}%
}}%
%
{\color[named]{Black}{%
\special{pn 8}%
\special{pa 1400 1800}%
\special{pa 1550 1950}%
\special{fp}%
}}%
%
{\color[named]{Black}{%
\special{pn 8}%
\special{pa 1800 1800}%
\special{pa 1950 1950}%
\special{fp}%
}}%
%
{\color[named]{Black}{%
\special{pn 8}%
\special{pa 1400 2200}%
\special{pa 1800 1800}%
\special{fp}%
\special{sh 1}%
\special{pa 1800 1800}%
\special{pa 1739 1833}%
\special{pa 1762 1838}%
\special{pa 1767 1861}%
\special{pa 1800 1800}%
\special{fp}%
}}%
%
{\color[named]{Black}{%
\special{pn 8}%
\special{pa 2050 2050}%
\special{pa 2200 2200}%
\special{fp}%
\special{sh 1}%
\special{pa 2200 2200}%
\special{pa 2167 2139}%
\special{pa 2162 2162}%
\special{pa 2139 2167}%
\special{pa 2200 2200}%
\special{fp}%
}}%
%
{\color[named]{Black}{%
\special{pn 8}%
\special{pa 2200 2200}%
\special{pa 3000 2200}%
\special{fp}%
\special{sh 1}%
\special{pa 3000 2200}%
\special{pa 2933 2180}%
\special{pa 2947 2200}%
\special{pa 2933 2220}%
\special{pa 3000 2200}%
\special{fp}%
}}%
%
{\color[named]{Black}{%
\special{pn 8}%
\special{pa 3000 2200}%
\special{pa 3400 1800}%
\special{fp}%
\special{sh 1}%
\special{pa 3400 1800}%
\special{pa 3339 1833}%
\special{pa 3362 1838}%
\special{pa 3367 1861}%
\special{pa 3400 1800}%
\special{fp}%
}}%
%
{\color[named]{Black}{%
\special{pn 8}%
\special{pa 3650 2050}%
\special{pa 3800 2200}%
\special{fp}%
\special{sh 1}%
\special{pa 3800 2200}%
\special{pa 3767 2139}%
\special{pa 3762 2162}%
\special{pa 3739 2167}%
\special{pa 3800 2200}%
\special{fp}%
}}%
%
{\color[named]{Black}{%
\special{pn 8}%
\special{pa 3800 2200}%
\special{pa 4800 2200}%
\special{fp}%
\special{sh 1}%
\special{pa 4800 2200}%
\special{pa 4733 2180}%
\special{pa 4747 2200}%
\special{pa 4733 2220}%
\special{pa 4800 2200}%
\special{fp}%
}}%
%
{\color[named]{Black}{%
\special{pn 8}%
\special{pa 4800 1800}%
\special{pa 4400 1400}%
\special{fp}%
\special{sh 1}%
\special{pa 4400 1400}%
\special{pa 4433 1461}%
\special{pa 4438 1438}%
\special{pa 4461 1433}%
\special{pa 4400 1400}%
\special{fp}%
}}%
%
{\color[named]{Black}{%
\special{pn 8}%
\special{pa 4150 1650}%
\special{pa 4000 1800}%
\special{fp}%
\special{sh 1}%
\special{pa 4000 1800}%
\special{pa 4061 1767}%
\special{pa 4038 1762}%
\special{pa 4033 1739}%
\special{pa 4000 1800}%
\special{fp}%
}}%
%
{\color[named]{Black}{%
\special{pn 8}%
\special{pa 4000 1800}%
\special{pa 3800 1800}%
\special{fp}%
\special{sh 1}%
\special{pa 3800 1800}%
\special{pa 3867 1820}%
\special{pa 3853 1800}%
\special{pa 3867 1780}%
\special{pa 3800 1800}%
\special{fp}%
}}%
%
{\color[named]{Black}{%
\special{pn 8}%
\special{pa 3800 1800}%
\special{pa 3400 2200}%
\special{fp}%
\special{sh 1}%
\special{pa 3400 2200}%
\special{pa 3461 2167}%
\special{pa 3438 2162}%
\special{pa 3433 2139}%
\special{pa 3400 2200}%
\special{fp}%
}}%
%
{\color[named]{Black}{%
\special{pn 8}%
\special{pa 3150 1950}%
\special{pa 3000 1800}%
\special{fp}%
\special{sh 1}%
\special{pa 3000 1800}%
\special{pa 3033 1861}%
\special{pa 3038 1838}%
\special{pa 3061 1833}%
\special{pa 3000 1800}%
\special{fp}%
}}%
%
{\color[named]{Black}{%
\special{pn 8}%
\special{pa 3000 1800}%
\special{pa 2800 1800}%
\special{fp}%
\special{sh 1}%
\special{pa 2800 1800}%
\special{pa 2867 1820}%
\special{pa 2853 1800}%
\special{pa 2867 1780}%
\special{pa 2800 1800}%
\special{fp}%
}}%
%
{\color[named]{Black}{%
\special{pn 8}%
\special{pa 2800 1800}%
\special{pa 2400 1400}%
\special{fp}%
\special{sh 1}%
\special{pa 2400 1400}%
\special{pa 2433 1461}%
\special{pa 2438 1438}%
\special{pa 2461 1433}%
\special{pa 2400 1400}%
\special{fp}%
\special{pa 2400 1400}%
\special{pa 1400 1400}%
\special{fp}%
\special{sh 1}%
\special{pa 1400 1400}%
\special{pa 1467 1420}%
\special{pa 1453 1400}%
\special{pa 1467 1380}%
\special{pa 1400 1400}%
\special{fp}%
\special{pa 1400 1400}%
\special{pa 1400 1400}%
\special{fp}%
}}%
%
{\color[named]{Black}{%
\special{pn 8}%
\special{pa 1650 2050}%
\special{pa 1800 2200}%
\special{fp}%
\special{sh 1}%
\special{pa 1800 2200}%
\special{pa 1767 2139}%
\special{pa 1762 2162}%
\special{pa 1739 2167}%
\special{pa 1800 2200}%
\special{fp}%
\special{pa 1800 2200}%
\special{pa 2200 1800}%
\special{fp}%
\special{sh 1}%
\special{pa 2200 1800}%
\special{pa 2139 1833}%
\special{pa 2162 1838}%
\special{pa 2167 1861}%
\special{pa 2200 1800}%
\special{fp}%
\special{pa 2200 1800}%
\special{pa 2400 1800}%
\special{fp}%
\special{sh 1}%
\special{pa 2400 1800}%
\special{pa 2333 1780}%
\special{pa 2347 1800}%
\special{pa 2333 1820}%
\special{pa 2400 1800}%
\special{fp}%
}}%
%
{\color[named]{Black}{%
\special{pn 8}%
\special{pa 2650 1550}%
\special{pa 2800 1400}%
\special{fp}%
\special{sh 1}%
\special{pa 2800 1400}%
\special{pa 2739 1433}%
\special{pa 2762 1438}%
\special{pa 2767 1461}%
\special{pa 2800 1400}%
\special{fp}%
}}%
%
{\color[named]{Black}{%
\special{pn 8}%
\special{pa 2800 1400}%
\special{pa 4000 1400}%
\special{fp}%
\special{sh 1}%
\special{pa 4000 1400}%
\special{pa 3933 1380}%
\special{pa 3947 1400}%
\special{pa 3933 1420}%
\special{pa 4000 1400}%
\special{fp}%
\special{pa 4000 1400}%
\special{pa 4400 1800}%
\special{fp}%
\special{sh 1}%
\special{pa 4400 1800}%
\special{pa 4367 1739}%
\special{pa 4362 1762}%
\special{pa 4339 1767}%
\special{pa 4400 1800}%
\special{fp}%
}}%
%
{\color[named]{Black}{%
\special{pn 8}%
\special{pa 4650 1550}%
\special{pa 4800 1400}%
\special{fp}%
\special{sh 1}%
\special{pa 4800 1400}%
\special{pa 4739 1433}%
\special{pa 4762 1438}%
\special{pa 4767 1461}%
\special{pa 4800 1400}%
\special{fp}%
}}%
%
{\color[named]{Black}{%
\special{pn 8}%
\special{pa 4800 2600}%
\special{pa 1400 2600}%
\special{fp}%
\special{sh 1}%
\special{pa 1400 2600}%
\special{pa 1467 2620}%
\special{pa 1453 2600}%
\special{pa 1467 2580}%
\special{pa 1400 2600}%
\special{fp}%
}}%
\end{picture}}%

%% file: proof_of_sign_l2.tex
{\unitlength 0.1in%
\begin{picture}(47.4000,34.4500)(4.6000,-38.0000)%
%
{\color[named]{Black}{%
\special{pn 8}%
\special{pa 2800 1400}%
\special{pa 3200 1000}%
\special{fp}%
\special{sh 1}%
\special{pa 3200 1000}%
\special{pa 3139 1033}%
\special{pa 3162 1038}%
\special{pa 3167 1061}%
\special{pa 3200 1000}%
\special{fp}%
\special{pa 3450 1250}%
\special{pa 3600 1400}%
\special{fp}%
\special{sh 1}%
\special{pa 3600 1400}%
\special{pa 3567 1339}%
\special{pa 3562 1362}%
\special{pa 3539 1367}%
\special{pa 3600 1400}%
\special{fp}%
\special{pa 3600 1400}%
\special{pa 4600 1400}%
\special{fp}%
\special{sh 1}%
\special{pa 4600 1400}%
\special{pa 4533 1380}%
\special{pa 4547 1400}%
\special{pa 4533 1420}%
\special{pa 4600 1400}%
\special{fp}%
}}%
%
{\color[named]{Black}{%
\special{pn 8}%
\special{ar 4600 1200 200 200 4.7123890 1.5707963}%
}}%
%
{\color[named]{Black}{%
\special{pn 8}%
\special{pa 4615 1001}%
\special{pa 4600 1000}%
\special{fp}%
\special{sh 1}%
\special{pa 4600 1000}%
\special{pa 4665 1024}%
\special{pa 4653 1004}%
\special{pa 4668 984}%
\special{pa 4600 1000}%
\special{fp}%
}}%
%
{\color[named]{Black}{%
\special{pn 8}%
\special{pa 4600 1000}%
\special{pa 4200 600}%
\special{fp}%
\special{sh 1}%
\special{pa 4200 600}%
\special{pa 4233 661}%
\special{pa 4238 638}%
\special{pa 4261 633}%
\special{pa 4200 600}%
\special{fp}%
\special{pa 3950 850}%
\special{pa 3800 1000}%
\special{fp}%
\special{sh 1}%
\special{pa 3800 1000}%
\special{pa 3861 967}%
\special{pa 3838 962}%
\special{pa 3833 939}%
\special{pa 3800 1000}%
\special{fp}%
\special{pa 3800 1000}%
\special{pa 3600 1000}%
\special{fp}%
\special{sh 1}%
\special{pa 3600 1000}%
\special{pa 3667 1020}%
\special{pa 3653 1000}%
\special{pa 3667 980}%
\special{pa 3600 1000}%
\special{fp}%
\special{pa 3600 1000}%
\special{pa 3200 1400}%
\special{fp}%
\special{sh 1}%
\special{pa 3200 1400}%
\special{pa 3261 1367}%
\special{pa 3238 1362}%
\special{pa 3233 1339}%
\special{pa 3200 1400}%
\special{fp}%
\special{pa 2950 1150}%
\special{pa 2800 1000}%
\special{fp}%
\special{sh 1}%
\special{pa 2800 1000}%
\special{pa 2833 1061}%
\special{pa 2838 1038}%
\special{pa 2861 1033}%
\special{pa 2800 1000}%
\special{fp}%
}}%
%
{\color[named]{Black}{%
\special{pn 8}%
\special{ar 2800 800 200 200 1.5707963 4.7123890}%
}}%
%
{\color[named]{Black}{%
\special{pn 8}%
\special{pa 2785 601}%
\special{pa 2800 600}%
\special{fp}%
\special{sh 1}%
\special{pa 2800 600}%
\special{pa 2732 584}%
\special{pa 2747 604}%
\special{pa 2735 624}%
\special{pa 2800 600}%
\special{fp}%
}}%
%
{\color[named]{Black}{%
\special{pn 8}%
\special{pa 2800 600}%
\special{pa 3800 600}%
\special{fp}%
\special{sh 1}%
\special{pa 3800 600}%
\special{pa 3733 580}%
\special{pa 3747 600}%
\special{pa 3733 620}%
\special{pa 3800 600}%
\special{fp}%
\special{pa 3800 600}%
\special{pa 4200 1000}%
\special{fp}%
\special{sh 1}%
\special{pa 4200 1000}%
\special{pa 4167 939}%
\special{pa 4162 962}%
\special{pa 4139 967}%
\special{pa 4200 1000}%
\special{fp}%
\special{pa 4450 750}%
\special{pa 4600 600}%
\special{fp}%
\special{sh 1}%
\special{pa 4600 600}%
\special{pa 4539 633}%
\special{pa 4562 638}%
\special{pa 4567 661}%
\special{pa 4600 600}%
\special{fp}%
}}%
%
{\color[named]{Black}{%
\special{pn 8}%
\special{ar 4600 1200 600 600 4.7123890 1.5707963}%
}}%
%
{\color[named]{Black}{%
\special{pn 8}%
\special{pa 4615 1800}%
\special{pa 4600 1800}%
\special{fp}%
\special{sh 1}%
\special{pa 4600 1800}%
\special{pa 4667 1820}%
\special{pa 4653 1800}%
\special{pa 4667 1780}%
\special{pa 4600 1800}%
\special{fp}%
}}%
%
{\color[named]{Black}{%
\special{pn 8}%
\special{pa 4600 1800}%
\special{pa 2800 1800}%
\special{fp}%
\special{sh 1}%
\special{pa 2800 1800}%
\special{pa 2867 1820}%
\special{pa 2853 1800}%
\special{pa 2867 1780}%
\special{pa 2800 1800}%
\special{fp}%
}}%
%
{\color[named]{Black}{%
\special{pn 8}%
\special{ar 2800 1600 200 200 1.5707963 4.7123890}%
}}%
%
{\color[named]{Black}{%
\special{pn 8}%
\special{pa 2785 1401}%
\special{pa 2800 1400}%
\special{fp}%
\special{sh 1}%
\special{pa 2800 1400}%
\special{pa 2732 1384}%
\special{pa 2747 1404}%
\special{pa 2735 1424}%
\special{pa 2800 1400}%
\special{fp}%
}}%
%
{\color[named]{Black}{%
\special{pn 8}%
\special{pa 3200 1400}%
\special{pa 3050 1250}%
\special{fp}%
\special{pa 3200 1000}%
\special{pa 3350 1150}%
\special{fp}%
}}%
%
{\color[named]{Black}{%
\special{pn 8}%
\special{pa 4050 750}%
\special{pa 4200 600}%
\special{fp}%
\special{pa 4200 1000}%
\special{pa 4350 850}%
\special{fp}%
}}%
%
{\color[named]{Black}{%
\special{pn 8}%
\special{pa 1400 600}%
\special{pa 1800 1800}%
\special{fp}%
\special{pa 1800 1800}%
\special{pa 1000 1800}%
\special{fp}%
\special{pa 1000 1800}%
\special{pa 1400 600}%
\special{fp}%
\special{pa 1000 1800}%
\special{pa 1600 1200}%
\special{fp}%
\special{pa 1600 1200}%
\special{pa 1260 1020}%
\special{fp}%
}}%
\put(14.0000,-4.2000){\makebox(0,0){{\color[named]{Black}{$\cfrac{0}{1}$}}}}%
\put(11.1000,-10.1000){\makebox(0,0){{\color[named]{Black}{$\cfrac{1}{1}$}}}}%
\put(17.8000,-11.8000){\makebox(0,0){{\color[named]{Black}{$\cfrac{1}{0}$}}}}%
\put(8.2000,-17.8000){\makebox(0,0){{\color[named]{Black}{$\cfrac{0}{1}$}}}}%
\put(19.8000,-17.8000){\makebox(0,0){{\color[named]{Black}{$\cfrac{1}{1}$}}}}%
%
{\color[named]{Black}{%
\special{pn 8}%
\special{pa 1400 2600}%
\special{pa 1000 3800}%
\special{fp}%
\special{pa 1000 3800}%
\special{pa 1800 3800}%
\special{fp}%
\special{pa 1800 3800}%
\special{pa 1400 2600}%
\special{fp}%
\special{pa 1200 3200}%
\special{pa 1800 3800}%
\special{fp}%
}}%
%
{\color[named]{Black}{%
\special{pn 8}%
\special{pa 2800 3400}%
\special{pa 3200 3000}%
\special{fp}%
\special{sh 1}%
\special{pa 3200 3000}%
\special{pa 3139 3033}%
\special{pa 3162 3038}%
\special{pa 3167 3061}%
\special{pa 3200 3000}%
\special{fp}%
}}%
%
{\color[named]{Black}{%
\special{pn 8}%
\special{pa 3200 3000}%
\special{pa 3400 3000}%
\special{fp}%
\special{sh 1}%
\special{pa 3400 3000}%
\special{pa 3333 2980}%
\special{pa 3347 3000}%
\special{pa 3333 3020}%
\special{pa 3400 3000}%
\special{fp}%
}}%
%
{\color[named]{Black}{%
\special{pn 8}%
\special{pa 3650 2750}%
\special{pa 3800 2600}%
\special{fp}%
\special{sh 1}%
\special{pa 3800 2600}%
\special{pa 3739 2633}%
\special{pa 3762 2638}%
\special{pa 3767 2661}%
\special{pa 3800 2600}%
\special{fp}%
\special{pa 3800 2600}%
\special{pa 4200 3000}%
\special{fp}%
\special{sh 1}%
\special{pa 4200 3000}%
\special{pa 4167 2939}%
\special{pa 4162 2962}%
\special{pa 4139 2967}%
\special{pa 4200 3000}%
\special{fp}%
\special{pa 4200 3400}%
\special{pa 3200 3400}%
\special{fp}%
\special{sh 1}%
\special{pa 3200 3400}%
\special{pa 3267 3420}%
\special{pa 3253 3400}%
\special{pa 3267 3380}%
\special{pa 3200 3400}%
\special{fp}%
\special{pa 2950 3150}%
\special{pa 2800 3000}%
\special{fp}%
\special{sh 1}%
\special{pa 2800 3000}%
\special{pa 2833 3061}%
\special{pa 2838 3038}%
\special{pa 2861 3033}%
\special{pa 2800 3000}%
\special{fp}%
\special{pa 2800 2600}%
\special{pa 3400 2600}%
\special{fp}%
\special{sh 1}%
\special{pa 3400 2600}%
\special{pa 3333 2580}%
\special{pa 3347 2600}%
\special{pa 3333 2620}%
\special{pa 3400 2600}%
\special{fp}%
\special{pa 3400 2600}%
\special{pa 3800 3000}%
\special{fp}%
\special{sh 1}%
\special{pa 3800 3000}%
\special{pa 3767 2939}%
\special{pa 3762 2962}%
\special{pa 3739 2967}%
\special{pa 3800 3000}%
\special{fp}%
\special{pa 4050 2750}%
\special{pa 4200 2600}%
\special{fp}%
\special{sh 1}%
\special{pa 4200 2600}%
\special{pa 4139 2633}%
\special{pa 4162 2638}%
\special{pa 4167 2661}%
\special{pa 4200 2600}%
\special{fp}%
\special{pa 4200 3800}%
\special{pa 2800 3800}%
\special{fp}%
\special{sh 1}%
\special{pa 2800 3800}%
\special{pa 2867 3820}%
\special{pa 2853 3800}%
\special{pa 2867 3780}%
\special{pa 2800 3800}%
\special{fp}%
}}%
%
{\color[named]{Black}{%
\special{pn 8}%
\special{ar 4200 3200 200 200 4.7123890 1.5707963}%
}}%
%
{\color[named]{Black}{%
\special{pn 8}%
\special{pa 4215 3399}%
\special{pa 4200 3400}%
\special{fp}%
\special{sh 1}%
\special{pa 4200 3400}%
\special{pa 4268 3416}%
\special{pa 4253 3396}%
\special{pa 4265 3376}%
\special{pa 4200 3400}%
\special{fp}%
}}%
%
{\color[named]{Black}{%
\special{pn 8}%
\special{ar 4200 3200 600 600 4.7123890 1.5707963}%
}}%
%
{\color[named]{Black}{%
\special{pn 8}%
\special{pa 4215 3800}%
\special{pa 4200 3800}%
\special{fp}%
\special{sh 1}%
\special{pa 4200 3800}%
\special{pa 4267 3820}%
\special{pa 4253 3800}%
\special{pa 4267 3780}%
\special{pa 4200 3800}%
\special{fp}%
}}%
%
{\color[named]{Black}{%
\special{pn 8}%
\special{ar 2800 2800 200 200 1.5707963 4.7123890}%
}}%
%
{\color[named]{Black}{%
\special{pn 8}%
\special{pa 2785 2601}%
\special{pa 2800 2600}%
\special{fp}%
\special{sh 1}%
\special{pa 2800 2600}%
\special{pa 2732 2584}%
\special{pa 2747 2604}%
\special{pa 2735 2624}%
\special{pa 2800 2600}%
\special{fp}%
}}%
%
{\color[named]{Black}{%
\special{pn 8}%
\special{ar 2800 3600 200 200 1.5707963 4.7123890}%
}}%
%
{\color[named]{Black}{%
\special{pn 8}%
\special{pa 2785 3401}%
\special{pa 2800 3400}%
\special{fp}%
\special{sh 1}%
\special{pa 2800 3400}%
\special{pa 2732 3384}%
\special{pa 2747 3404}%
\special{pa 2735 3424}%
\special{pa 2800 3400}%
\special{fp}%
}}%
%
{\color[named]{Black}{%
\special{pn 8}%
\special{pa 3200 3400}%
\special{pa 3050 3250}%
\special{fp}%
\special{pa 3400 3000}%
\special{pa 3550 2850}%
\special{fp}%
\special{pa 3950 2850}%
\special{pa 3800 3000}%
\special{fp}%
}}%
\put(19.8000,-37.8000){\makebox(0,0){{\color[named]{Black}{$\cfrac{1}{1}$}}}}%
\put(8.2000,-37.8000){\makebox(0,0){{\color[named]{Black}{$\cfrac{0}{1}$}}}}%
\put(10.2000,-31.8000){\makebox(0,0){{\color[named]{Black}{$\cfrac{1}{0}$}}}}%
\put(14.0000,-24.2000){\makebox(0,0){{\color[named]{Black}{$\cfrac{0}{1}$}}}}%
%
{\color[named]{Black}{%
\special{pn 20}%
\special{pa 1400 600}%
\special{pa 1600 1200}%
\special{fp}%
\special{pa 1600 1200}%
\special{pa 1000 1800}%
\special{fp}%
}}%
%
{\color[named]{Black}{%
\special{pn 20}%
\special{pa 1400 2600}%
\special{pa 1000 3800}%
\special{fp}%
}}%
\end{picture}}%

%% file: proof_of_sign_odd_10.tex
{\unitlength 0.1in%
\begin{picture}(65.2000,30.2200)(-1.2000,-32.2700)%
%
{\color[named]{Black}{%
\special{pn 8}%
\special{pa 1000 400}%
\special{pa 600 1600}%
\special{fp}%
\special{pa 1000 400}%
\special{pa 1400 1600}%
\special{fp}%
}}%
%
{\color[named]{Black}{%
\special{pn 8}%
\special{pa 650 1450}%
\special{pa 1250 1150}%
\special{fp}%
\special{pa 1250 1150}%
\special{pa 800 1000}%
\special{fp}%
\special{pa 800 1000}%
\special{pa 1120 760}%
\special{fp}%
}}%
\put(10.0000,-2.7000){\makebox(0,0){{\color[named]{Black}{$\alpha$}}}}%
\put(14.3000,-11.5000){\makebox(0,0){{\color[named]{Black}{$\gamma$}}}}%
\put(10.0000,-14.9000){\makebox(0,0){{\color[named]{Black}{$\vdots$}}}}%
\put(5.5000,-9.8000){\makebox(0,0){{\color[named]{Black}{$\beta \equiv \cfrac{1}{0}$}}}}%
%
{\color[named]{Black}{%
\special{pn 20}%
\special{pa 1000 400}%
\special{pa 800 1000}%
\special{fp}%
\special{pa 800 1000}%
\special{pa 1250 1150}%
\special{fp}%
}}%
%
{\color[named]{Black}{%
\special{pn 20}%
\special{pa 2400 800}%
\special{pa 2800 800}%
\special{pa 2800 1200}%
\special{pa 2400 1200}%
\special{pa 2400 800}%
\special{pa 2800 800}%
\special{fp}%
}}%
%
{\color[named]{Black}{%
\special{pn 8}%
\special{ar 2800 1400 200 200 4.7123890 1.5707963}%
}}%
%
{\color[named]{Black}{%
\special{pn 8}%
\special{pa 2815 1599}%
\special{pa 2800 1600}%
\special{fp}%
\special{sh 1}%
\special{pa 2800 1600}%
\special{pa 2868 1616}%
\special{pa 2853 1596}%
\special{pa 2865 1576}%
\special{pa 2800 1600}%
\special{fp}%
}}%
%
{\color[named]{Black}{%
\special{pn 8}%
\special{pa 2800 400}%
\special{pa 2400 400}%
\special{fp}%
\special{sh 1}%
\special{pa 2400 400}%
\special{pa 2467 420}%
\special{pa 2453 400}%
\special{pa 2467 380}%
\special{pa 2400 400}%
\special{fp}%
}}%
%
{\color[named]{Black}{%
\special{pn 8}%
\special{ar 2800 600 200 200 4.7123890 1.5707963}%
}}%
%
{\color[named]{Black}{%
\special{pn 8}%
\special{pa 2815 401}%
\special{pa 2800 400}%
\special{fp}%
\special{sh 1}%
\special{pa 2800 400}%
\special{pa 2865 424}%
\special{pa 2853 404}%
\special{pa 2868 384}%
\special{pa 2800 400}%
\special{fp}%
}}%
\put(26.0000,-10.0000){\makebox(0,0){{\color[named]{Black}{$a_{n-2}$}}}}%
%
{\color[named]{Black}{%
\special{pn 20}%
\special{pa 3800 796}%
\special{pa 4200 796}%
\special{pa 4200 1196}%
\special{pa 3800 1196}%
\special{pa 3800 796}%
\special{pa 4200 796}%
\special{fp}%
}}%
%
{\color[named]{Black}{%
\special{pn 8}%
\special{pa 4200 796}%
\special{pa 4400 796}%
\special{fp}%
\special{sh 1}%
\special{pa 4400 796}%
\special{pa 4333 776}%
\special{pa 4347 796}%
\special{pa 4333 816}%
\special{pa 4400 796}%
\special{fp}%
}}%
%
{\color[named]{Black}{%
\special{pn 8}%
\special{pa 4650 546}%
\special{pa 4800 396}%
\special{fp}%
\special{sh 1}%
\special{pa 4800 396}%
\special{pa 4739 429}%
\special{pa 4762 434}%
\special{pa 4767 457}%
\special{pa 4800 396}%
\special{fp}%
}}%
%
{\color[named]{Black}{%
\special{pn 8}%
\special{pa 5800 1596}%
\special{pa 3800 1596}%
\special{fp}%
\special{sh 1}%
\special{pa 3800 1596}%
\special{pa 3867 1616}%
\special{pa 3853 1596}%
\special{pa 3867 1576}%
\special{pa 3800 1596}%
\special{fp}%
}}%
%
{\color[named]{Black}{%
\special{pn 8}%
\special{pa 5650 1046}%
\special{pa 5800 1196}%
\special{fp}%
\special{sh 1}%
\special{pa 5800 1196}%
\special{pa 5767 1135}%
\special{pa 5762 1158}%
\special{pa 5739 1163}%
\special{pa 5800 1196}%
\special{fp}%
}}%
%
{\color[named]{Black}{%
\special{pn 8}%
\special{pa 5000 1196}%
\special{pa 5400 796}%
\special{fp}%
\special{sh 1}%
\special{pa 5400 796}%
\special{pa 5339 829}%
\special{pa 5362 834}%
\special{pa 5367 857}%
\special{pa 5400 796}%
\special{fp}%
}}%
%
{\color[named]{Black}{%
\special{pn 8}%
\special{pa 4200 1196}%
\special{pa 5000 1196}%
\special{fp}%
\special{sh 1}%
\special{pa 5000 1196}%
\special{pa 4933 1176}%
\special{pa 4947 1196}%
\special{pa 4933 1216}%
\special{pa 5000 1196}%
\special{fp}%
}}%
%
{\color[named]{Black}{%
\special{pn 8}%
\special{pa 4800 396}%
\special{pa 5800 396}%
\special{fp}%
\special{sh 1}%
\special{pa 5800 396}%
\special{pa 5733 376}%
\special{pa 5747 396}%
\special{pa 5733 416}%
\special{pa 5800 396}%
\special{fp}%
}}%
%
{\color[named]{Black}{%
\special{pn 8}%
\special{ar 5800 1396 200 200 4.7123890 1.5707963}%
}}%
%
{\color[named]{Black}{%
\special{pn 8}%
\special{pa 5815 1595}%
\special{pa 5800 1596}%
\special{fp}%
\special{sh 1}%
\special{pa 5800 1596}%
\special{pa 5868 1612}%
\special{pa 5853 1592}%
\special{pa 5865 1572}%
\special{pa 5800 1596}%
\special{fp}%
}}%
%
{\color[named]{Black}{%
\special{pn 8}%
\special{pa 5800 796}%
\special{pa 5400 1196}%
\special{fp}%
\special{sh 1}%
\special{pa 5400 1196}%
\special{pa 5461 1163}%
\special{pa 5438 1158}%
\special{pa 5433 1135}%
\special{pa 5400 1196}%
\special{fp}%
}}%
%
{\color[named]{Black}{%
\special{pn 8}%
\special{pa 4400 396}%
\special{pa 3800 396}%
\special{fp}%
\special{sh 1}%
\special{pa 3800 396}%
\special{pa 3867 416}%
\special{pa 3853 396}%
\special{pa 3867 376}%
\special{pa 3800 396}%
\special{fp}%
}}%
%
{\color[named]{Black}{%
\special{pn 8}%
\special{pa 4800 796}%
\special{pa 4400 396}%
\special{fp}%
\special{sh 1}%
\special{pa 4400 396}%
\special{pa 4433 457}%
\special{pa 4438 434}%
\special{pa 4461 429}%
\special{pa 4400 396}%
\special{fp}%
}}%
%
{\color[named]{Black}{%
\special{pn 8}%
\special{pa 5000 796}%
\special{pa 4800 796}%
\special{fp}%
\special{sh 1}%
\special{pa 4800 796}%
\special{pa 4867 816}%
\special{pa 4853 796}%
\special{pa 4867 776}%
\special{pa 4800 796}%
\special{fp}%
}}%
%
{\color[named]{Black}{%
\special{pn 8}%
\special{pa 5150 946}%
\special{pa 5000 796}%
\special{fp}%
\special{sh 1}%
\special{pa 5000 796}%
\special{pa 5033 857}%
\special{pa 5038 834}%
\special{pa 5061 829}%
\special{pa 5000 796}%
\special{fp}%
}}%
%
{\color[named]{Black}{%
\special{pn 8}%
\special{pa 5400 1196}%
\special{pa 5250 1046}%
\special{fp}%
\special{pa 5400 796}%
\special{pa 5550 946}%
\special{fp}%
\special{pa 4400 796}%
\special{pa 4550 646}%
\special{fp}%
}}%
%
{\color[named]{Black}{%
\special{pn 8}%
\special{ar 5800 596 200 200 4.7123890 1.5707963}%
}}%
%
{\color[named]{Black}{%
\special{pn 8}%
\special{pa 5815 795}%
\special{pa 5800 796}%
\special{fp}%
\special{sh 1}%
\special{pa 5800 796}%
\special{pa 5868 812}%
\special{pa 5853 792}%
\special{pa 5865 772}%
\special{pa 5800 796}%
\special{fp}%
}}%
\put(40.0000,-9.9600){\makebox(0,0){{\color[named]{Black}{$a_{n-2}$}}}}%
%
{\color[named]{Black}{%
\special{pn 8}%
\special{pa 2800 1600}%
\special{pa 2400 1600}%
\special{fp}%
\special{sh 1}%
\special{pa 2400 1600}%
\special{pa 2467 1620}%
\special{pa 2453 1600}%
\special{pa 2467 1580}%
\special{pa 2400 1600}%
\special{fp}%
}}%
%
{\color[named]{Black}{%
\special{pn 8}%
\special{pa 1000 2000}%
\special{pa 600 3200}%
\special{fp}%
\special{pa 1000 2000}%
\special{pa 1400 3200}%
\special{fp}%
}}%
%
{\color[named]{Black}{%
\special{pn 8}%
\special{pa 650 3050}%
\special{pa 1280 2840}%
\special{fp}%
\special{pa 1280 2840}%
\special{pa 790 2630}%
\special{fp}%
\special{pa 1280 2840}%
\special{pa 870 2400}%
\special{fp}%
\special{pa 870 2400}%
\special{pa 1080 2240}%
\special{fp}%
}}%
\put(10.0000,-18.9000){\makebox(0,0){{\color[named]{Black}{$\alpha$}}}}%
\put(6.0000,-24.0000){\makebox(0,0){{\color[named]{Black}{$\beta \equiv \cfrac{1}{0}$}}}}%
\put(14.0000,-28.0000){\makebox(0,0){{\color[named]{Black}{$\gamma$}}}}%
%
{\color[named]{Black}{%
\special{pn 20}%
\special{pa 1000 2000}%
\special{pa 870 2390}%
\special{fp}%
\special{pa 870 2390}%
\special{pa 1280 2840}%
\special{fp}%
}}%
\put(10.0000,-30.9000){\makebox(0,0){{\color[named]{Black}{$\vdots$}}}}%
%
{\color[named]{Black}{%
\special{pn 20}%
\special{pa 2400 2400}%
\special{pa 2800 2400}%
\special{pa 2800 2800}%
\special{pa 2400 2800}%
\special{pa 2400 2400}%
\special{pa 2800 2400}%
\special{fp}%
}}%
%
{\color[named]{Black}{%
\special{pn 8}%
\special{ar 2800 3000 200 200 4.7123890 1.5707963}%
}}%
%
{\color[named]{Black}{%
\special{pn 8}%
\special{pa 2815 3199}%
\special{pa 2800 3200}%
\special{fp}%
\special{sh 1}%
\special{pa 2800 3200}%
\special{pa 2868 3216}%
\special{pa 2853 3196}%
\special{pa 2865 3176}%
\special{pa 2800 3200}%
\special{fp}%
}}%
%
{\color[named]{Black}{%
\special{pn 8}%
\special{pa 2400 2000}%
\special{pa 2800 2000}%
\special{fp}%
\special{sh 1}%
\special{pa 2800 2000}%
\special{pa 2733 1980}%
\special{pa 2747 2000}%
\special{pa 2733 2020}%
\special{pa 2800 2000}%
\special{fp}%
}}%
%
{\color[named]{Black}{%
\special{pn 8}%
\special{ar 2800 2200 200 200 4.7123890 1.5707963}%
}}%
%
{\color[named]{Black}{%
\special{pn 8}%
\special{pa 2815 2399}%
\special{pa 2800 2400}%
\special{fp}%
\special{sh 1}%
\special{pa 2800 2400}%
\special{pa 2868 2416}%
\special{pa 2853 2396}%
\special{pa 2865 2376}%
\special{pa 2800 2400}%
\special{fp}%
}}%
\put(26.0000,-26.0000){\makebox(0,0){{\color[named]{Black}{$a_{n-2}$}}}}%
%
{\color[named]{Black}{%
\special{pn 20}%
\special{pa 3800 2396}%
\special{pa 4200 2396}%
\special{pa 4200 2796}%
\special{pa 3800 2796}%
\special{pa 3800 2396}%
\special{pa 4200 2396}%
\special{fp}%
}}%
%
{\color[named]{Black}{%
\special{pn 8}%
\special{pa 4400 2396}%
\special{pa 4200 2396}%
\special{fp}%
\special{sh 1}%
\special{pa 4200 2396}%
\special{pa 4267 2416}%
\special{pa 4253 2396}%
\special{pa 4267 2376}%
\special{pa 4200 2396}%
\special{fp}%
}}%
%
{\color[named]{Black}{%
\special{pn 8}%
\special{pa 4550 2250}%
\special{pa 4400 2400}%
\special{fp}%
\special{sh 1}%
\special{pa 4400 2400}%
\special{pa 4461 2367}%
\special{pa 4438 2362}%
\special{pa 4433 2339}%
\special{pa 4400 2400}%
\special{fp}%
}}%
%
{\color[named]{Black}{%
\special{pn 8}%
\special{pa 6200 3190}%
\special{pa 3800 3190}%
\special{fp}%
\special{sh 1}%
\special{pa 3800 3190}%
\special{pa 3867 3210}%
\special{pa 3853 3190}%
\special{pa 3867 3170}%
\special{pa 3800 3190}%
\special{fp}%
}}%
%
{\color[named]{Black}{%
\special{pn 8}%
\special{pa 6050 2640}%
\special{pa 6200 2790}%
\special{fp}%
\special{sh 1}%
\special{pa 6200 2790}%
\special{pa 6167 2729}%
\special{pa 6162 2752}%
\special{pa 6139 2757}%
\special{pa 6200 2790}%
\special{fp}%
}}%
%
{\color[named]{Black}{%
\special{pn 8}%
\special{pa 5400 2790}%
\special{pa 5800 2390}%
\special{fp}%
\special{sh 1}%
\special{pa 5800 2390}%
\special{pa 5739 2423}%
\special{pa 5762 2428}%
\special{pa 5767 2451}%
\special{pa 5800 2390}%
\special{fp}%
}}%
%
{\color[named]{Black}{%
\special{pn 8}%
\special{pa 4200 2790}%
\special{pa 5400 2790}%
\special{fp}%
\special{sh 1}%
\special{pa 5400 2790}%
\special{pa 5333 2770}%
\special{pa 5347 2790}%
\special{pa 5333 2810}%
\special{pa 5400 2790}%
\special{fp}%
}}%
%
{\color[named]{Black}{%
\special{pn 8}%
\special{ar 6200 2990 200 200 4.7123890 1.5707963}%
}}%
%
{\color[named]{Black}{%
\special{pn 8}%
\special{pa 6215 3189}%
\special{pa 6200 3190}%
\special{fp}%
\special{sh 1}%
\special{pa 6200 3190}%
\special{pa 6268 3206}%
\special{pa 6253 3186}%
\special{pa 6265 3166}%
\special{pa 6200 3190}%
\special{fp}%
}}%
%
{\color[named]{Black}{%
\special{pn 8}%
\special{pa 5800 2790}%
\special{pa 5650 2640}%
\special{fp}%
}}%
%
{\color[named]{Black}{%
\special{pn 8}%
\special{pa 4650 2140}%
\special{pa 4800 1990}%
\special{fp}%
}}%
%
{\color[named]{Black}{%
\special{pn 8}%
\special{pa 5800 2390}%
\special{pa 5950 2540}%
\special{fp}%
}}%
%
{\color[named]{Black}{%
\special{pn 8}%
\special{ar 6200 2190 200 200 4.7123890 1.5707963}%
}}%
%
{\color[named]{Black}{%
\special{pn 8}%
\special{pa 6215 2389}%
\special{pa 6200 2390}%
\special{fp}%
\special{sh 1}%
\special{pa 6200 2390}%
\special{pa 6268 2406}%
\special{pa 6253 2386}%
\special{pa 6265 2366}%
\special{pa 6200 2390}%
\special{fp}%
}}%
\put(40.0000,-25.9600){\makebox(0,0){{\color[named]{Black}{$a_{n-2}$}}}}%
%
{\color[named]{Black}{%
\special{pn 8}%
\special{pa 2800 3200}%
\special{pa 2400 3200}%
\special{fp}%
\special{sh 1}%
\special{pa 2400 3200}%
\special{pa 2467 3220}%
\special{pa 2453 3200}%
\special{pa 2467 3180}%
\special{pa 2400 3200}%
\special{fp}%
}}%
%
{\color[named]{Black}{%
\special{pn 8}%
\special{pa 3800 1995}%
\special{pa 4400 1995}%
\special{fp}%
\special{sh 1}%
\special{pa 4400 1995}%
\special{pa 4333 1975}%
\special{pa 4347 1995}%
\special{pa 4333 2015}%
\special{pa 4400 1995}%
\special{fp}%
}}%
%
{\color[named]{Black}{%
\special{pn 8}%
\special{pa 5400 2400}%
\special{pa 5200 2400}%
\special{fp}%
\special{sh 1}%
\special{pa 5200 2400}%
\special{pa 5267 2420}%
\special{pa 5253 2400}%
\special{pa 5267 2380}%
\special{pa 5200 2400}%
\special{fp}%
}}%
%
{\color[named]{Black}{%
\special{pn 8}%
\special{pa 5550 2550}%
\special{pa 5400 2400}%
\special{fp}%
\special{sh 1}%
\special{pa 5400 2400}%
\special{pa 5433 2461}%
\special{pa 5438 2438}%
\special{pa 5461 2433}%
\special{pa 5400 2400}%
\special{fp}%
}}%
%
{\color[named]{Black}{%
\special{pn 8}%
\special{pa 6200 2400}%
\special{pa 5800 2800}%
\special{fp}%
\special{sh 1}%
\special{pa 5800 2800}%
\special{pa 5861 2767}%
\special{pa 5838 2762}%
\special{pa 5833 2739}%
\special{pa 5800 2800}%
\special{fp}%
}}%
%
{\color[named]{Black}{%
\special{pn 8}%
\special{pa 5200 2000}%
\special{pa 6200 2000}%
\special{fp}%
\special{sh 1}%
\special{pa 6200 2000}%
\special{pa 6133 1980}%
\special{pa 6147 2000}%
\special{pa 6133 2020}%
\special{pa 6200 2000}%
\special{fp}%
}}%
%
{\color[named]{Black}{%
\special{pn 8}%
\special{pa 5050 2150}%
\special{pa 5200 2000}%
\special{fp}%
\special{sh 1}%
\special{pa 5200 2000}%
\special{pa 5139 2033}%
\special{pa 5162 2038}%
\special{pa 5167 2061}%
\special{pa 5200 2000}%
\special{fp}%
}}%
%
{\color[named]{Black}{%
\special{pn 8}%
\special{pa 4400 2000}%
\special{pa 4800 2400}%
\special{fp}%
\special{sh 1}%
\special{pa 4800 2400}%
\special{pa 4767 2339}%
\special{pa 4762 2362}%
\special{pa 4739 2367}%
\special{pa 4800 2400}%
\special{fp}%
}}%
%
{\color[named]{Black}{%
\special{pn 8}%
\special{pa 5200 2400}%
\special{pa 4800 2000}%
\special{fp}%
\special{sh 1}%
\special{pa 4800 2000}%
\special{pa 4833 2061}%
\special{pa 4838 2038}%
\special{pa 4861 2033}%
\special{pa 4800 2000}%
\special{fp}%
}}%
%
{\color[named]{Black}{%
\special{pn 8}%
\special{pa 4800 2400}%
\special{pa 4950 2250}%
\special{fp}%
}}%
\put(22.0000,-4.0000){\makebox(0,0){{\color[named]{Black}{$\cdots$}}}}%
\put(22.0000,-8.0000){\makebox(0,0){{\color[named]{Black}{$\cdots$}}}}%
\put(22.0000,-12.0000){\makebox(0,0){{\color[named]{Black}{$\cdots$}}}}%
\put(22.0000,-16.0000){\makebox(0,0){{\color[named]{Black}{$\cdots$}}}}%
\put(22.0000,-20.0000){\makebox(0,0){{\color[named]{Black}{$\cdots$}}}}%
\put(22.0000,-24.0000){\makebox(0,0){{\color[named]{Black}{$\cdots$}}}}%
\put(22.0000,-28.0000){\makebox(0,0){{\color[named]{Black}{$\cdots$}}}}%
\put(22.0000,-32.0000){\makebox(0,0){{\color[named]{Black}{$\cdots$}}}}%
\put(36.0000,-3.9500){\makebox(0,0){{\color[named]{Black}{$\cdots$}}}}%
\put(36.0000,-7.9500){\makebox(0,0){{\color[named]{Black}{$\cdots$}}}}%
\put(36.0000,-11.9500){\makebox(0,0){{\color[named]{Black}{$\cdots$}}}}%
\put(36.0000,-15.9500){\makebox(0,0){{\color[named]{Black}{$\cdots$}}}}%
\put(36.0000,-19.9500){\makebox(0,0){{\color[named]{Black}{$\cdots$}}}}%
\put(36.0000,-23.9500){\makebox(0,0){{\color[named]{Black}{$\cdots$}}}}%
\put(36.0000,-27.9500){\makebox(0,0){{\color[named]{Black}{$\cdots$}}}}%
\put(36.0000,-31.9500){\makebox(0,0){{\color[named]{Black}{$\cdots$}}}}%
%
{\color[named]{Black}{%
\special{pn 20}%
\special{pa 1280 2840}%
\special{pa 650 3050}%
\special{fp}%
}}%
%
{\color[named]{Black}{%
\special{pn 20}%
\special{pa 1250 1150}%
\special{pa 1300 1300}%
\special{fp}%
}}%
\end{picture}}%

%% file: proof_of_sign_odd_not10.tex
{\unitlength 0.1in%
\begin{picture}(59.8000,30.2200)(4.3000,-32.2700)%
%
{\color[named]{Black}{%
\special{pn 8}%
\special{pa 1000 400}%
\special{pa 600 1600}%
\special{fp}%
\special{pa 1000 400}%
\special{pa 1400 1600}%
\special{fp}%
}}%
%
{\color[named]{Black}{%
\special{pn 8}%
\special{pa 650 1450}%
\special{pa 1250 1150}%
\special{fp}%
\special{pa 1250 1150}%
\special{pa 800 1000}%
\special{fp}%
\special{pa 800 1000}%
\special{pa 1120 760}%
\special{fp}%
}}%
\put(10.0000,-2.7000){\makebox(0,0){{\color[named]{Black}{$\alpha$}}}}%
\put(14.3000,-11.5000){\makebox(0,0){{\color[named]{Black}{$\gamma$}}}}%
\put(10.0000,-14.9000){\makebox(0,0){{\color[named]{Black}{$\vdots$}}}}%
\put(6.2000,-9.8000){\makebox(0,0){{\color[named]{Black}{$\beta$}}}}%
%
{\color[named]{Black}{%
\special{pn 20}%
\special{pa 2400 2400}%
\special{pa 2800 2400}%
\special{pa 2800 2800}%
\special{pa 2400 2800}%
\special{pa 2400 2400}%
\special{pa 2800 2400}%
\special{fp}%
}}%
%
{\color[named]{Black}{%
\special{pn 8}%
\special{ar 2800 3000 200 200 4.7123890 1.5707963}%
}}%
%
{\color[named]{Black}{%
\special{pn 8}%
\special{pa 2815 3199}%
\special{pa 2800 3200}%
\special{fp}%
\special{sh 1}%
\special{pa 2800 3200}%
\special{pa 2868 3216}%
\special{pa 2853 3196}%
\special{pa 2865 3176}%
\special{pa 2800 3200}%
\special{fp}%
}}%
%
{\color[named]{Black}{%
\special{pn 8}%
\special{pa 2800 2000}%
\special{pa 2400 2000}%
\special{fp}%
\special{sh 1}%
\special{pa 2400 2000}%
\special{pa 2467 2020}%
\special{pa 2453 2000}%
\special{pa 2467 1980}%
\special{pa 2400 2000}%
\special{fp}%
}}%
%
{\color[named]{Black}{%
\special{pn 8}%
\special{ar 2800 2200 200 200 4.7123890 1.5707963}%
}}%
%
{\color[named]{Black}{%
\special{pn 8}%
\special{pa 2815 2001}%
\special{pa 2800 2000}%
\special{fp}%
\special{sh 1}%
\special{pa 2800 2000}%
\special{pa 2865 2024}%
\special{pa 2853 2004}%
\special{pa 2868 1984}%
\special{pa 2800 2000}%
\special{fp}%
}}%
\put(26.0000,-26.0000){\makebox(0,0){{\color[named]{Black}{$a_{n-2}$}}}}%
%
{\color[named]{Black}{%
\special{pn 20}%
\special{pa 3800 806}%
\special{pa 4200 806}%
\special{pa 4200 1206}%
\special{pa 3800 1206}%
\special{pa 3800 806}%
\special{pa 4200 806}%
\special{fp}%
}}%
%
{\color[named]{Black}{%
\special{pn 8}%
\special{pa 4400 806}%
\special{pa 4200 806}%
\special{fp}%
\special{sh 1}%
\special{pa 4200 806}%
\special{pa 4267 826}%
\special{pa 4253 806}%
\special{pa 4267 786}%
\special{pa 4200 806}%
\special{fp}%
}}%
%
{\color[named]{Black}{%
\special{pn 8}%
\special{pa 5800 1606}%
\special{pa 3800 1606}%
\special{fp}%
\special{sh 1}%
\special{pa 3800 1606}%
\special{pa 3867 1626}%
\special{pa 3853 1606}%
\special{pa 3867 1586}%
\special{pa 3800 1606}%
\special{fp}%
}}%
%
{\color[named]{Black}{%
\special{pn 8}%
\special{pa 5650 1056}%
\special{pa 5800 1206}%
\special{fp}%
\special{sh 1}%
\special{pa 5800 1206}%
\special{pa 5767 1145}%
\special{pa 5762 1168}%
\special{pa 5739 1173}%
\special{pa 5800 1206}%
\special{fp}%
}}%
%
{\color[named]{Black}{%
\special{pn 8}%
\special{pa 5000 1206}%
\special{pa 5400 806}%
\special{fp}%
\special{sh 1}%
\special{pa 5400 806}%
\special{pa 5339 839}%
\special{pa 5362 844}%
\special{pa 5367 867}%
\special{pa 5400 806}%
\special{fp}%
}}%
%
{\color[named]{Black}{%
\special{pn 8}%
\special{pa 4200 1206}%
\special{pa 5000 1206}%
\special{fp}%
\special{sh 1}%
\special{pa 5000 1206}%
\special{pa 4933 1186}%
\special{pa 4947 1206}%
\special{pa 4933 1226}%
\special{pa 5000 1206}%
\special{fp}%
}}%
%
{\color[named]{Black}{%
\special{pn 8}%
\special{pa 5800 400}%
\special{pa 4800 400}%
\special{fp}%
\special{sh 1}%
\special{pa 4800 400}%
\special{pa 4867 420}%
\special{pa 4853 400}%
\special{pa 4867 380}%
\special{pa 4800 400}%
\special{fp}%
}}%
%
{\color[named]{Black}{%
\special{pn 8}%
\special{ar 5800 1406 200 200 4.7123890 1.5707963}%
}}%
%
{\color[named]{Black}{%
\special{pn 8}%
\special{pa 5815 1605}%
\special{pa 5800 1606}%
\special{fp}%
\special{sh 1}%
\special{pa 5800 1606}%
\special{pa 5868 1622}%
\special{pa 5853 1602}%
\special{pa 5865 1582}%
\special{pa 5800 1606}%
\special{fp}%
}}%
%
{\color[named]{Black}{%
\special{pn 8}%
\special{pa 5400 1200}%
\special{pa 5800 800}%
\special{fp}%
\special{sh 1}%
\special{pa 5800 800}%
\special{pa 5739 833}%
\special{pa 5762 838}%
\special{pa 5767 861}%
\special{pa 5800 800}%
\special{fp}%
}}%
%
{\color[named]{Black}{%
\special{pn 8}%
\special{pa 4800 800}%
\special{pa 5000 800}%
\special{fp}%
\special{sh 1}%
\special{pa 5000 800}%
\special{pa 4933 780}%
\special{pa 4947 800}%
\special{pa 4933 820}%
\special{pa 5000 800}%
\special{fp}%
}}%
%
{\color[named]{Black}{%
\special{pn 8}%
\special{pa 5150 956}%
\special{pa 5000 806}%
\special{fp}%
}}%
%
{\color[named]{Black}{%
\special{pn 8}%
\special{pa 4650 556}%
\special{pa 4800 406}%
\special{fp}%
}}%
%
{\color[named]{Black}{%
\special{pn 8}%
\special{pa 5400 806}%
\special{pa 5550 956}%
\special{fp}%
}}%
%
{\color[named]{Black}{%
\special{pn 8}%
\special{ar 5800 606 200 200 4.7123890 1.5707963}%
}}%
%
{\color[named]{Black}{%
\special{pn 8}%
\special{pa 5815 407}%
\special{pa 5800 406}%
\special{fp}%
\special{sh 1}%
\special{pa 5800 406}%
\special{pa 5865 430}%
\special{pa 5853 410}%
\special{pa 5868 390}%
\special{pa 5800 406}%
\special{fp}%
}}%
\put(40.0000,-10.0600){\makebox(0,0){{\color[named]{Black}{$a_{n-2}$}}}}%
%
{\color[named]{Black}{%
\special{pn 8}%
\special{pa 2800 3200}%
\special{pa 2400 3200}%
\special{fp}%
\special{sh 1}%
\special{pa 2400 3200}%
\special{pa 2467 3220}%
\special{pa 2453 3200}%
\special{pa 2467 3180}%
\special{pa 2400 3200}%
\special{fp}%
}}%
%
{\color[named]{Black}{%
\special{pn 8}%
\special{pa 1000 2000}%
\special{pa 600 3200}%
\special{fp}%
\special{pa 1000 2000}%
\special{pa 1400 3200}%
\special{fp}%
}}%
%
{\color[named]{Black}{%
\special{pn 8}%
\special{pa 650 3050}%
\special{pa 1280 2840}%
\special{fp}%
\special{pa 1280 2840}%
\special{pa 790 2630}%
\special{fp}%
\special{pa 1280 2840}%
\special{pa 870 2400}%
\special{fp}%
\special{pa 870 2400}%
\special{pa 1080 2240}%
\special{fp}%
}}%
\put(10.0000,-18.9000){\makebox(0,0){{\color[named]{Black}{$\alpha$}}}}%
\put(6.7000,-24.0000){\makebox(0,0){{\color[named]{Black}{$\beta$}}}}%
\put(14.0000,-28.0000){\makebox(0,0){{\color[named]{Black}{$\gamma$}}}}%
\put(10.0000,-30.9000){\makebox(0,0){{\color[named]{Black}{$\vdots$}}}}%
%
{\color[named]{Black}{%
\special{pn 20}%
\special{pa 2400 805}%
\special{pa 2800 805}%
\special{pa 2800 1205}%
\special{pa 2400 1205}%
\special{pa 2400 805}%
\special{pa 2800 805}%
\special{fp}%
}}%
%
{\color[named]{Black}{%
\special{pn 8}%
\special{ar 2800 1405 200 200 4.7123890 1.5707963}%
}}%
%
{\color[named]{Black}{%
\special{pn 8}%
\special{pa 2815 1604}%
\special{pa 2800 1605}%
\special{fp}%
\special{sh 1}%
\special{pa 2800 1605}%
\special{pa 2868 1621}%
\special{pa 2853 1601}%
\special{pa 2865 1581}%
\special{pa 2800 1605}%
\special{fp}%
}}%
%
{\color[named]{Black}{%
\special{pn 8}%
\special{pa 2400 405}%
\special{pa 2800 405}%
\special{fp}%
\special{sh 1}%
\special{pa 2800 405}%
\special{pa 2733 385}%
\special{pa 2747 405}%
\special{pa 2733 425}%
\special{pa 2800 405}%
\special{fp}%
}}%
%
{\color[named]{Black}{%
\special{pn 8}%
\special{ar 2800 605 200 200 4.7123890 1.5707963}%
}}%
%
{\color[named]{Black}{%
\special{pn 8}%
\special{pa 2815 804}%
\special{pa 2800 805}%
\special{fp}%
\special{sh 1}%
\special{pa 2800 805}%
\special{pa 2868 821}%
\special{pa 2853 801}%
\special{pa 2865 781}%
\special{pa 2800 805}%
\special{fp}%
}}%
\put(26.0000,-10.0500){\makebox(0,0){{\color[named]{Black}{$a_{n-2}$}}}}%
%
{\color[named]{Black}{%
\special{pn 20}%
\special{pa 3810 2406}%
\special{pa 4210 2406}%
\special{pa 4210 2806}%
\special{pa 3810 2806}%
\special{pa 3810 2406}%
\special{pa 4210 2406}%
\special{fp}%
}}%
%
{\color[named]{Black}{%
\special{pn 8}%
\special{pa 4210 2400}%
\special{pa 4410 2400}%
\special{fp}%
\special{sh 1}%
\special{pa 4410 2400}%
\special{pa 4343 2380}%
\special{pa 4357 2400}%
\special{pa 4343 2420}%
\special{pa 4410 2400}%
\special{fp}%
}}%
%
{\color[named]{Black}{%
\special{pn 8}%
\special{pa 6210 3200}%
\special{pa 3810 3200}%
\special{fp}%
\special{sh 1}%
\special{pa 3810 3200}%
\special{pa 3877 3220}%
\special{pa 3863 3200}%
\special{pa 3877 3180}%
\special{pa 3810 3200}%
\special{fp}%
}}%
%
{\color[named]{Black}{%
\special{pn 8}%
\special{pa 6060 2650}%
\special{pa 6210 2800}%
\special{fp}%
\special{sh 1}%
\special{pa 6210 2800}%
\special{pa 6177 2739}%
\special{pa 6172 2762}%
\special{pa 6149 2767}%
\special{pa 6210 2800}%
\special{fp}%
}}%
%
{\color[named]{Black}{%
\special{pn 8}%
\special{pa 5410 2800}%
\special{pa 5810 2400}%
\special{fp}%
\special{sh 1}%
\special{pa 5810 2400}%
\special{pa 5749 2433}%
\special{pa 5772 2438}%
\special{pa 5777 2461}%
\special{pa 5810 2400}%
\special{fp}%
}}%
%
{\color[named]{Black}{%
\special{pn 8}%
\special{pa 4210 2800}%
\special{pa 5410 2800}%
\special{fp}%
\special{sh 1}%
\special{pa 5410 2800}%
\special{pa 5343 2780}%
\special{pa 5357 2800}%
\special{pa 5343 2820}%
\special{pa 5410 2800}%
\special{fp}%
}}%
%
{\color[named]{Black}{%
\special{pn 8}%
\special{ar 6210 3000 200 200 4.7123890 1.5707963}%
}}%
%
{\color[named]{Black}{%
\special{pn 8}%
\special{pa 6225 3199}%
\special{pa 6210 3200}%
\special{fp}%
\special{sh 1}%
\special{pa 6210 3200}%
\special{pa 6278 3216}%
\special{pa 6263 3196}%
\special{pa 6275 3176}%
\special{pa 6210 3200}%
\special{fp}%
}}%
%
{\color[named]{Black}{%
\special{pn 8}%
\special{pa 4410 2400}%
\special{pa 4560 2250}%
\special{fp}%
}}%
%
{\color[named]{Black}{%
\special{pn 8}%
\special{pa 5810 2400}%
\special{pa 5960 2550}%
\special{fp}%
}}%
%
{\color[named]{Black}{%
\special{pn 8}%
\special{ar 6210 2200 200 200 4.7123890 1.5707963}%
}}%
%
{\color[named]{Black}{%
\special{pn 8}%
\special{pa 6225 2001}%
\special{pa 6210 2000}%
\special{fp}%
\special{sh 1}%
\special{pa 6210 2000}%
\special{pa 6275 2024}%
\special{pa 6263 2004}%
\special{pa 6278 1984}%
\special{pa 6210 2000}%
\special{fp}%
}}%
\put(40.1000,-26.0600){\makebox(0,0){{\color[named]{Black}{$a_{n-2}$}}}}%
%
{\color[named]{Black}{%
\special{pn 8}%
\special{pa 2800 1605}%
\special{pa 2400 1605}%
\special{fp}%
\special{sh 1}%
\special{pa 2400 1605}%
\special{pa 2467 1625}%
\special{pa 2453 1605}%
\special{pa 2467 1585}%
\special{pa 2400 1605}%
\special{fp}%
}}%
\put(22.0000,-20.0000){\makebox(0,0){{\color[named]{Black}{$\cdots$}}}}%
\put(22.0000,-24.0000){\makebox(0,0){{\color[named]{Black}{$\cdots$}}}}%
\put(22.0000,-28.0000){\makebox(0,0){{\color[named]{Black}{$\cdots$}}}}%
\put(22.0000,-32.0000){\makebox(0,0){{\color[named]{Black}{$\cdots$}}}}%
\put(22.0000,-4.0500){\makebox(0,0){{\color[named]{Black}{$\cdots$}}}}%
\put(22.0000,-8.0500){\makebox(0,0){{\color[named]{Black}{$\cdots$}}}}%
\put(22.0000,-12.0500){\makebox(0,0){{\color[named]{Black}{$\cdots$}}}}%
\put(22.0000,-16.0500){\makebox(0,0){{\color[named]{Black}{$\cdots$}}}}%
\put(36.0000,-4.0500){\makebox(0,0){{\color[named]{Black}{$\cdots$}}}}%
\put(36.0000,-8.0500){\makebox(0,0){{\color[named]{Black}{$\cdots$}}}}%
\put(36.0000,-12.0500){\makebox(0,0){{\color[named]{Black}{$\cdots$}}}}%
\put(36.0000,-16.0500){\makebox(0,0){{\color[named]{Black}{$\cdots$}}}}%
\put(36.1000,-20.0500){\makebox(0,0){{\color[named]{Black}{$\cdots$}}}}%
\put(36.1000,-24.0500){\makebox(0,0){{\color[named]{Black}{$\cdots$}}}}%
\put(36.1000,-28.0500){\makebox(0,0){{\color[named]{Black}{$\cdots$}}}}%
\put(36.1000,-32.0500){\makebox(0,0){{\color[named]{Black}{$\cdots$}}}}%
\put(12.3000,-7.4000){\makebox(0,0){{\color[named]{Black}{$\cfrac{1}{0}$}}}}%
\put(12.0000,-22.2000){\makebox(0,0){{\color[named]{Black}{$\cfrac{1}{0}$}}}}%
%
{\color[named]{Black}{%
\special{pn 20}%
\special{pa 1000 400}%
\special{pa 1250 1150}%
\special{fp}%
\special{pa 1250 1150}%
\special{pa 650 1450}%
\special{fp}%
}}%
%
{\color[named]{Black}{%
\special{pn 20}%
\special{pa 1000 2000}%
\special{pa 1320 2960}%
\special{fp}%
}}%
%
{\color[named]{Black}{%
\special{pn 8}%
\special{pa 4550 650}%
\special{pa 4400 800}%
\special{fp}%
\special{sh 1}%
\special{pa 4400 800}%
\special{pa 4461 767}%
\special{pa 4438 762}%
\special{pa 4433 739}%
\special{pa 4400 800}%
\special{fp}%
}}%
%
{\color[named]{Black}{%
\special{pn 8}%
\special{pa 5250 1050}%
\special{pa 5400 1200}%
\special{fp}%
\special{sh 1}%
\special{pa 5400 1200}%
\special{pa 5367 1139}%
\special{pa 5362 1162}%
\special{pa 5339 1167}%
\special{pa 5400 1200}%
\special{fp}%
}}%
%
{\color[named]{Black}{%
\special{pn 8}%
\special{pa 3800 400}%
\special{pa 4400 400}%
\special{fp}%
\special{sh 1}%
\special{pa 4400 400}%
\special{pa 4333 380}%
\special{pa 4347 400}%
\special{pa 4333 420}%
\special{pa 4400 400}%
\special{fp}%
\special{pa 4400 400}%
\special{pa 4800 800}%
\special{fp}%
\special{sh 1}%
\special{pa 4800 800}%
\special{pa 4767 739}%
\special{pa 4762 762}%
\special{pa 4739 767}%
\special{pa 4800 800}%
\special{fp}%
}}%
%
{\color[named]{Black}{%
\special{pn 8}%
\special{pa 4650 2150}%
\special{pa 4800 2000}%
\special{fp}%
\special{sh 1}%
\special{pa 4800 2000}%
\special{pa 4739 2033}%
\special{pa 4762 2038}%
\special{pa 4767 2061}%
\special{pa 4800 2000}%
\special{fp}%
}}%
%
{\color[named]{Black}{%
\special{pn 8}%
\special{pa 3800 2000}%
\special{pa 3800 2000}%
\special{fp}%
}}%
%
{\color[named]{Black}{%
\special{pn 8}%
\special{pa 4400 2000}%
\special{pa 3800 2000}%
\special{fp}%
\special{sh 1}%
\special{pa 3800 2000}%
\special{pa 3867 2020}%
\special{pa 3853 2000}%
\special{pa 3867 1980}%
\special{pa 3800 2000}%
\special{fp}%
}}%
%
{\color[named]{Black}{%
\special{pn 8}%
\special{pa 4800 2400}%
\special{pa 4400 2000}%
\special{fp}%
\special{sh 1}%
\special{pa 4400 2000}%
\special{pa 4433 2061}%
\special{pa 4438 2038}%
\special{pa 4461 2033}%
\special{pa 4400 2000}%
\special{fp}%
}}%
%
{\color[named]{Black}{%
\special{pn 8}%
\special{pa 4950 2250}%
\special{pa 4800 2400}%
\special{fp}%
\special{sh 1}%
\special{pa 4800 2400}%
\special{pa 4861 2367}%
\special{pa 4838 2362}%
\special{pa 4833 2339}%
\special{pa 4800 2400}%
\special{fp}%
}}%
%
{\color[named]{Black}{%
\special{pn 8}%
\special{pa 6200 2000}%
\special{pa 5200 2000}%
\special{fp}%
\special{sh 1}%
\special{pa 5200 2000}%
\special{pa 5267 2020}%
\special{pa 5253 2000}%
\special{pa 5267 1980}%
\special{pa 5200 2000}%
\special{fp}%
}}%
%
{\color[named]{Black}{%
\special{pn 8}%
\special{pa 5800 2800}%
\special{pa 6200 2400}%
\special{fp}%
\special{sh 1}%
\special{pa 6200 2400}%
\special{pa 6139 2433}%
\special{pa 6162 2438}%
\special{pa 6167 2461}%
\special{pa 6200 2400}%
\special{fp}%
}}%
%
{\color[named]{Black}{%
\special{pn 8}%
\special{pa 5200 2400}%
\special{pa 5400 2400}%
\special{fp}%
\special{sh 1}%
\special{pa 5400 2400}%
\special{pa 5333 2380}%
\special{pa 5347 2400}%
\special{pa 5333 2420}%
\special{pa 5400 2400}%
\special{fp}%
}}%
%
{\color[named]{Black}{%
\special{pn 8}%
\special{pa 4800 2000}%
\special{pa 5200 2400}%
\special{fp}%
\special{sh 1}%
\special{pa 5200 2400}%
\special{pa 5167 2339}%
\special{pa 5162 2362}%
\special{pa 5139 2367}%
\special{pa 5200 2400}%
\special{fp}%
}}%
%
{\color[named]{Black}{%
\special{pn 8}%
\special{pa 5650 2650}%
\special{pa 5800 2800}%
\special{fp}%
\special{sh 1}%
\special{pa 5800 2800}%
\special{pa 5767 2739}%
\special{pa 5762 2762}%
\special{pa 5739 2767}%
\special{pa 5800 2800}%
\special{fp}%
}}%
%
{\color[named]{Black}{%
\special{pn 8}%
\special{pa 5400 2400}%
\special{pa 5550 2550}%
\special{fp}%
\special{pa 5200 2000}%
\special{pa 5050 2150}%
\special{fp}%
}}%
\end{picture}}%

%% file: proof_of_sign_even_10.tex
{\unitlength 0.1in%
\begin{picture}(69.0000,30.1700)(3.0000,-32.2700)%
%
{\color[named]{Black}{%
\special{pn 8}%
\special{pa 990 405}%
\special{pa 1390 1605}%
\special{fp}%
\special{pa 990 405}%
\special{pa 590 1605}%
\special{fp}%
}}%
%
{\color[named]{Black}{%
\special{pn 8}%
\special{pa 1340 1455}%
\special{pa 740 1155}%
\special{fp}%
\special{pa 740 1155}%
\special{pa 1190 1005}%
\special{fp}%
\special{pa 1190 1005}%
\special{pa 870 765}%
\special{fp}%
}}%
\put(9.9000,-2.7500){\makebox(0,0){{\color[named]{Black}{$\alpha$}}}}%
\put(5.6000,-11.5500){\makebox(0,0){{\color[named]{Black}{$\gamma$}}}}%
\put(9.9000,-14.9500){\makebox(0,0){{\color[named]{Black}{$\vdots$}}}}%
\put(14.4000,-9.8500){\makebox(0,0){{\color[named]{Black}{$\beta \equiv \cfrac{1}{0}$}}}}%
%
{\color[named]{Black}{%
\special{pn 20}%
\special{pa 990 405}%
\special{pa 1190 1005}%
\special{fp}%
\special{pa 1190 1005}%
\special{pa 740 1155}%
\special{fp}%
}}%
%
{\color[named]{Black}{%
\special{pn 20}%
\special{pa 4200 390}%
\special{pa 4600 390}%
\special{pa 4600 790}%
\special{pa 4200 790}%
\special{pa 4200 390}%
\special{pa 4600 390}%
\special{fp}%
}}%
\put(44.0000,-5.9000){\makebox(0,0){{\color[named]{Black}{$a_{n-2}$}}}}%
%
{\color[named]{Black}{%
\special{pn 8}%
\special{pa 1000 2000}%
\special{pa 1400 3200}%
\special{fp}%
\special{pa 1000 2000}%
\special{pa 600 3200}%
\special{fp}%
}}%
%
{\color[named]{Black}{%
\special{pn 8}%
\special{pa 1350 3050}%
\special{pa 720 2840}%
\special{fp}%
\special{pa 720 2840}%
\special{pa 1210 2630}%
\special{fp}%
\special{pa 720 2840}%
\special{pa 1130 2400}%
\special{fp}%
\special{pa 1130 2400}%
\special{pa 920 2240}%
\special{fp}%
}}%
\put(10.0000,-18.9000){\makebox(0,0){{\color[named]{Black}{$\alpha$}}}}%
\put(14.0000,-24.0000){\makebox(0,0){{\color[named]{Black}{$\beta \equiv \cfrac{1}{0}$}}}}%
\put(6.0000,-28.0000){\makebox(0,0){{\color[named]{Black}{$\gamma$}}}}%
%
{\color[named]{Black}{%
\special{pn 20}%
\special{pa 1000 2000}%
\special{pa 1130 2390}%
\special{fp}%
\special{pa 1130 2390}%
\special{pa 720 2840}%
\special{fp}%
}}%
\put(10.0000,-30.9000){\makebox(0,0){{\color[named]{Black}{$\vdots$}}}}%
%
{\color[named]{Black}{%
\special{pn 20}%
\special{pa 2400 2000}%
\special{pa 2800 2000}%
\special{pa 2800 2400}%
\special{pa 2400 2400}%
\special{pa 2400 2000}%
\special{pa 2800 2000}%
\special{fp}%
}}%
%
{\color[named]{Black}{%
\special{pn 8}%
\special{ar 2800 2600 200 200 4.7123890 1.5707963}%
}}%
%
{\color[named]{Black}{%
\special{pn 8}%
\special{pa 2815 2799}%
\special{pa 2800 2800}%
\special{fp}%
\special{sh 1}%
\special{pa 2800 2800}%
\special{pa 2868 2816}%
\special{pa 2853 2796}%
\special{pa 2865 2776}%
\special{pa 2800 2800}%
\special{fp}%
}}%
\put(26.0000,-22.0000){\makebox(0,0){{\color[named]{Black}{$a_{n-2}$}}}}%
%
{\color[named]{Black}{%
\special{pn 20}%
\special{pa 4200 1996}%
\special{pa 4600 1996}%
\special{pa 4600 2396}%
\special{pa 4200 2396}%
\special{pa 4200 1996}%
\special{pa 4600 1996}%
\special{fp}%
}}%
%
{\color[named]{Black}{%
\special{pn 8}%
\special{pa 4600 2396}%
\special{pa 4800 2396}%
\special{fp}%
\special{sh 1}%
\special{pa 4800 2396}%
\special{pa 4733 2376}%
\special{pa 4747 2396}%
\special{pa 4733 2416}%
\special{pa 4800 2396}%
\special{fp}%
}}%
\put(44.0000,-21.9600){\makebox(0,0){{\color[named]{Black}{$a_{n-2}$}}}}%
%
{\color[named]{Black}{%
\special{pn 8}%
\special{pa 2800 2800}%
\special{pa 2400 2800}%
\special{fp}%
\special{sh 1}%
\special{pa 2400 2800}%
\special{pa 2467 2820}%
\special{pa 2453 2800}%
\special{pa 2467 2780}%
\special{pa 2400 2800}%
\special{fp}%
}}%
\put(22.0000,-4.0000){\makebox(0,0){{\color[named]{Black}{$\cdots$}}}}%
\put(22.0000,-8.0000){\makebox(0,0){{\color[named]{Black}{$\cdots$}}}}%
\put(22.0000,-12.0000){\makebox(0,0){{\color[named]{Black}{$\cdots$}}}}%
\put(22.0000,-16.0000){\makebox(0,0){{\color[named]{Black}{$\cdots$}}}}%
\put(22.0000,-20.0000){\makebox(0,0){{\color[named]{Black}{$\cdots$}}}}%
\put(22.0000,-24.0000){\makebox(0,0){{\color[named]{Black}{$\cdots$}}}}%
\put(22.0000,-28.0000){\makebox(0,0){{\color[named]{Black}{$\cdots$}}}}%
\put(22.0000,-32.0000){\makebox(0,0){{\color[named]{Black}{$\cdots$}}}}%
\put(40.0000,-3.9500){\makebox(0,0){{\color[named]{Black}{$\cdots$}}}}%
\put(40.0000,-7.9500){\makebox(0,0){{\color[named]{Black}{$\cdots$}}}}%
\put(40.0000,-11.9500){\makebox(0,0){{\color[named]{Black}{$\cdots$}}}}%
\put(40.0000,-15.9500){\makebox(0,0){{\color[named]{Black}{$\cdots$}}}}%
\put(40.0000,-19.9500){\makebox(0,0){{\color[named]{Black}{$\cdots$}}}}%
\put(40.0000,-23.9500){\makebox(0,0){{\color[named]{Black}{$\cdots$}}}}%
\put(40.0000,-27.9500){\makebox(0,0){{\color[named]{Black}{$\cdots$}}}}%
\put(40.0000,-31.9500){\makebox(0,0){{\color[named]{Black}{$\cdots$}}}}%
%
{\color[named]{Black}{%
\special{pn 20}%
\special{pa 720 2840}%
\special{pa 1350 3050}%
\special{fp}%
}}%
%
{\color[named]{Black}{%
\special{pn 20}%
\special{pa 740 1155}%
\special{pa 690 1305}%
\special{fp}%
}}%
%
{\color[named]{Black}{%
\special{pn 8}%
\special{ar 2800 2600 600 600 4.7123890 1.5707963}%
}}%
%
{\color[named]{Black}{%
\special{pn 8}%
\special{pa 2815 3200}%
\special{pa 2800 3200}%
\special{fp}%
\special{sh 1}%
\special{pa 2800 3200}%
\special{pa 2867 3220}%
\special{pa 2853 3200}%
\special{pa 2867 3180}%
\special{pa 2800 3200}%
\special{fp}%
}}%
%
{\color[named]{Black}{%
\special{pn 8}%
\special{pa 2800 3200}%
\special{pa 2400 3200}%
\special{fp}%
\special{sh 1}%
\special{pa 2400 3200}%
\special{pa 2467 3220}%
\special{pa 2453 3200}%
\special{pa 2467 3180}%
\special{pa 2400 3200}%
\special{fp}%
}}%
%
{\color[named]{Black}{%
\special{pn 8}%
\special{pa 5050 2650}%
\special{pa 5200 2800}%
\special{fp}%
\special{sh 1}%
\special{pa 5200 2800}%
\special{pa 5167 2739}%
\special{pa 5162 2762}%
\special{pa 5139 2767}%
\special{pa 5200 2800}%
\special{fp}%
\special{pa 5200 2800}%
\special{pa 5600 2400}%
\special{fp}%
\special{sh 1}%
\special{pa 5600 2400}%
\special{pa 5539 2433}%
\special{pa 5562 2438}%
\special{pa 5567 2461}%
\special{pa 5600 2400}%
\special{fp}%
\special{pa 5600 2400}%
\special{pa 5800 2400}%
\special{fp}%
\special{sh 1}%
\special{pa 5800 2400}%
\special{pa 5733 2380}%
\special{pa 5747 2400}%
\special{pa 5733 2420}%
\special{pa 5800 2400}%
\special{fp}%
\special{pa 6050 2150}%
\special{pa 6200 2000}%
\special{fp}%
\special{sh 1}%
\special{pa 6200 2000}%
\special{pa 6139 2033}%
\special{pa 6162 2038}%
\special{pa 6167 2061}%
\special{pa 6200 2000}%
\special{fp}%
\special{pa 6200 2000}%
\special{pa 6600 2400}%
\special{fp}%
\special{sh 1}%
\special{pa 6600 2400}%
\special{pa 6567 2339}%
\special{pa 6562 2362}%
\special{pa 6539 2367}%
\special{pa 6600 2400}%
\special{fp}%
\special{pa 6600 2800}%
\special{pa 5600 2800}%
\special{fp}%
\special{sh 1}%
\special{pa 5600 2800}%
\special{pa 5667 2820}%
\special{pa 5653 2800}%
\special{pa 5667 2780}%
\special{pa 5600 2800}%
\special{fp}%
\special{pa 5350 2550}%
\special{pa 5200 2400}%
\special{fp}%
\special{sh 1}%
\special{pa 5200 2400}%
\special{pa 5233 2461}%
\special{pa 5238 2438}%
\special{pa 5261 2433}%
\special{pa 5200 2400}%
\special{fp}%
\special{pa 5200 2400}%
\special{pa 4800 2800}%
\special{fp}%
\special{sh 1}%
\special{pa 4800 2800}%
\special{pa 4861 2767}%
\special{pa 4838 2762}%
\special{pa 4833 2739}%
\special{pa 4800 2800}%
\special{fp}%
\special{pa 4800 2800}%
\special{pa 4200 2800}%
\special{fp}%
\special{sh 1}%
\special{pa 4200 2800}%
\special{pa 4267 2820}%
\special{pa 4253 2800}%
\special{pa 4267 2780}%
\special{pa 4200 2800}%
\special{fp}%
}}%
%
{\color[named]{Black}{%
\special{pn 8}%
\special{pa 4600 2000}%
\special{pa 5800 2000}%
\special{fp}%
\special{sh 1}%
\special{pa 5800 2000}%
\special{pa 5733 1980}%
\special{pa 5747 2000}%
\special{pa 5733 2020}%
\special{pa 5800 2000}%
\special{fp}%
\special{pa 5800 2000}%
\special{pa 6200 2400}%
\special{fp}%
\special{sh 1}%
\special{pa 6200 2400}%
\special{pa 6167 2339}%
\special{pa 6162 2362}%
\special{pa 6139 2367}%
\special{pa 6200 2400}%
\special{fp}%
\special{pa 6450 2150}%
\special{pa 6600 2000}%
\special{fp}%
\special{sh 1}%
\special{pa 6600 2000}%
\special{pa 6539 2033}%
\special{pa 6562 2038}%
\special{pa 6567 2061}%
\special{pa 6600 2000}%
\special{fp}%
\special{pa 6600 3200}%
\special{pa 4200 3200}%
\special{fp}%
\special{sh 1}%
\special{pa 4200 3200}%
\special{pa 4267 3220}%
\special{pa 4253 3200}%
\special{pa 4267 3180}%
\special{pa 4200 3200}%
\special{fp}%
}}%
%
{\color[named]{Black}{%
\special{pn 8}%
\special{pa 5600 2800}%
\special{pa 5450 2650}%
\special{fp}%
\special{pa 4800 2400}%
\special{pa 4950 2550}%
\special{fp}%
\special{pa 5950 2250}%
\special{pa 5800 2400}%
\special{fp}%
\special{pa 6200 2400}%
\special{pa 6350 2250}%
\special{fp}%
}}%
%
{\color[named]{Black}{%
\special{pn 8}%
\special{ar 6600 2600 200 200 4.7123890 1.5707963}%
}}%
%
{\color[named]{Black}{%
\special{pn 8}%
\special{pa 6615 2799}%
\special{pa 6600 2800}%
\special{fp}%
\special{sh 1}%
\special{pa 6600 2800}%
\special{pa 6668 2816}%
\special{pa 6653 2796}%
\special{pa 6665 2776}%
\special{pa 6600 2800}%
\special{fp}%
}}%
%
{\color[named]{Black}{%
\special{pn 8}%
\special{ar 6600 2600 600 600 4.7123890 1.5707963}%
}}%
%
{\color[named]{Black}{%
\special{pn 8}%
\special{pa 6615 3200}%
\special{pa 6600 3200}%
\special{fp}%
\special{sh 1}%
\special{pa 6600 3200}%
\special{pa 6667 3220}%
\special{pa 6653 3200}%
\special{pa 6667 3180}%
\special{pa 6600 3200}%
\special{fp}%
}}%
%
{\color[named]{Black}{%
\special{pn 20}%
\special{pa 2400 400}%
\special{pa 2800 400}%
\special{pa 2800 800}%
\special{pa 2400 800}%
\special{pa 2400 400}%
\special{pa 2800 400}%
\special{fp}%
}}%
%
{\color[named]{Black}{%
\special{pn 8}%
\special{ar 2800 1000 200 200 4.7123890 1.5707963}%
}}%
%
{\color[named]{Black}{%
\special{pn 8}%
\special{pa 2815 1199}%
\special{pa 2800 1200}%
\special{fp}%
\special{sh 1}%
\special{pa 2800 1200}%
\special{pa 2868 1216}%
\special{pa 2853 1196}%
\special{pa 2865 1176}%
\special{pa 2800 1200}%
\special{fp}%
}}%
\put(26.0000,-6.0000){\makebox(0,0){{\color[named]{Black}{$a_{n-2}$}}}}%
%
{\color[named]{Black}{%
\special{pn 8}%
\special{pa 2800 1200}%
\special{pa 2400 1200}%
\special{fp}%
\special{sh 1}%
\special{pa 2400 1200}%
\special{pa 2467 1220}%
\special{pa 2453 1200}%
\special{pa 2467 1180}%
\special{pa 2400 1200}%
\special{fp}%
}}%
%
{\color[named]{Black}{%
\special{pn 8}%
\special{ar 2800 1000 600 600 4.7123890 1.5707963}%
}}%
%
{\color[named]{Black}{%
\special{pn 8}%
\special{pa 2815 400}%
\special{pa 2800 400}%
\special{fp}%
\special{sh 1}%
\special{pa 2800 400}%
\special{pa 2867 420}%
\special{pa 2853 400}%
\special{pa 2867 380}%
\special{pa 2800 400}%
\special{fp}%
}}%
%
{\color[named]{Black}{%
\special{pn 8}%
\special{pa 2400 1600}%
\special{pa 2800 1600}%
\special{fp}%
\special{sh 1}%
\special{pa 2800 1600}%
\special{pa 2733 1580}%
\special{pa 2747 1600}%
\special{pa 2733 1620}%
\special{pa 2800 1600}%
\special{fp}%
}}%
%
{\color[named]{Black}{%
\special{pn 8}%
\special{pa 4600 800}%
\special{pa 4800 800}%
\special{fp}%
\special{sh 1}%
\special{pa 4800 800}%
\special{pa 4733 780}%
\special{pa 4747 800}%
\special{pa 4733 820}%
\special{pa 4800 800}%
\special{fp}%
}}%
%
{\color[named]{Black}{%
\special{pn 8}%
\special{pa 4800 1200}%
\special{pa 4200 1200}%
\special{fp}%
\special{sh 1}%
\special{pa 4200 1200}%
\special{pa 4267 1220}%
\special{pa 4253 1200}%
\special{pa 4267 1180}%
\special{pa 4200 1200}%
\special{fp}%
}}%
%
{\color[named]{Black}{%
\special{pn 8}%
\special{pa 5200 800}%
\special{pa 4800 1200}%
\special{fp}%
\special{sh 1}%
\special{pa 4800 1200}%
\special{pa 4861 1167}%
\special{pa 4838 1162}%
\special{pa 4833 1139}%
\special{pa 4800 1200}%
\special{fp}%
}}%
%
{\color[named]{Black}{%
\special{pn 8}%
\special{pa 5400 800}%
\special{pa 5200 800}%
\special{fp}%
\special{sh 1}%
\special{pa 5200 800}%
\special{pa 5267 820}%
\special{pa 5253 800}%
\special{pa 5267 780}%
\special{pa 5200 800}%
\special{fp}%
}}%
%
{\color[named]{Black}{%
\special{pn 8}%
\special{pa 5550 650}%
\special{pa 5400 800}%
\special{fp}%
\special{sh 1}%
\special{pa 5400 800}%
\special{pa 5461 767}%
\special{pa 5438 762}%
\special{pa 5433 739}%
\special{pa 5400 800}%
\special{fp}%
}}%
%
{\color[named]{Black}{%
\special{pn 8}%
\special{pa 6200 800}%
\special{pa 5800 400}%
\special{fp}%
\special{sh 1}%
\special{pa 5800 400}%
\special{pa 5833 461}%
\special{pa 5838 438}%
\special{pa 5861 433}%
\special{pa 5800 400}%
\special{fp}%
}}%
%
{\color[named]{Black}{%
\special{pn 8}%
\special{pa 5200 1200}%
\special{pa 6200 1200}%
\special{fp}%
\special{sh 1}%
\special{pa 6200 1200}%
\special{pa 6133 1180}%
\special{pa 6147 1200}%
\special{pa 6133 1220}%
\special{pa 6200 1200}%
\special{fp}%
}}%
%
{\color[named]{Black}{%
\special{pn 8}%
\special{pa 5050 1050}%
\special{pa 5200 1200}%
\special{fp}%
\special{sh 1}%
\special{pa 5200 1200}%
\special{pa 5167 1139}%
\special{pa 5162 1162}%
\special{pa 5139 1167}%
\special{pa 5200 1200}%
\special{fp}%
}}%
%
{\color[named]{Black}{%
\special{pn 8}%
\special{ar 6200 1000 200 200 4.7123890 1.5707963}%
}}%
%
{\color[named]{Black}{%
\special{pn 8}%
\special{pa 6215 801}%
\special{pa 6200 800}%
\special{fp}%
\special{sh 1}%
\special{pa 6200 800}%
\special{pa 6265 824}%
\special{pa 6253 804}%
\special{pa 6268 784}%
\special{pa 6200 800}%
\special{fp}%
}}%
%
{\color[named]{Black}{%
\special{pn 8}%
\special{ar 6200 1000 600 600 4.7123890 1.5707963}%
}}%
%
{\color[named]{Black}{%
\special{pn 8}%
\special{pa 6215 400}%
\special{pa 6200 400}%
\special{fp}%
\special{sh 1}%
\special{pa 6200 400}%
\special{pa 6267 420}%
\special{pa 6253 400}%
\special{pa 6267 380}%
\special{pa 6200 400}%
\special{fp}%
}}%
%
{\color[named]{Black}{%
\special{pn 8}%
\special{pa 4800 800}%
\special{pa 4950 950}%
\special{fp}%
}}%
%
{\color[named]{Black}{%
\special{pn 8}%
\special{pa 5650 550}%
\special{pa 5800 400}%
\special{fp}%
}}%
%
{\color[named]{Black}{%
\special{pn 8}%
\special{pa 4200 1600}%
\special{pa 6200 1600}%
\special{fp}%
\special{sh 1}%
\special{pa 6200 1600}%
\special{pa 6133 1580}%
\special{pa 6147 1600}%
\special{pa 6133 1620}%
\special{pa 6200 1600}%
\special{fp}%
}}%
%
{\color[named]{Black}{%
\special{pn 8}%
\special{pa 5400 400}%
\special{pa 4600 400}%
\special{fp}%
\special{sh 1}%
\special{pa 4600 400}%
\special{pa 4667 420}%
\special{pa 4653 400}%
\special{pa 4667 380}%
\special{pa 4600 400}%
\special{fp}%
}}%
%
{\color[named]{Black}{%
\special{pn 8}%
\special{pa 5800 800}%
\special{pa 5400 400}%
\special{fp}%
\special{sh 1}%
\special{pa 5400 400}%
\special{pa 5433 461}%
\special{pa 5438 438}%
\special{pa 5461 433}%
\special{pa 5400 400}%
\special{fp}%
}}%
%
{\color[named]{Black}{%
\special{pn 8}%
\special{pa 5950 650}%
\special{pa 5800 800}%
\special{fp}%
\special{sh 1}%
\special{pa 5800 800}%
\special{pa 5861 767}%
\special{pa 5838 762}%
\special{pa 5833 739}%
\special{pa 5800 800}%
\special{fp}%
}}%
%
{\color[named]{Black}{%
\special{pn 8}%
\special{pa 6200 400}%
\special{pa 6050 550}%
\special{fp}%
}}%
\end{picture}}%

%% file: proof_of_sign_even_not10.tex
{\unitlength 0.1in%
\begin{picture}(69.0000,30.1700)(3.0000,-32.2700)%
%
{\color[named]{Black}{%
\special{pn 8}%
\special{pa 990 405}%
\special{pa 1390 1605}%
\special{fp}%
\special{pa 990 405}%
\special{pa 590 1605}%
\special{fp}%
}}%
%
{\color[named]{Black}{%
\special{pn 8}%
\special{pa 1340 1455}%
\special{pa 740 1155}%
\special{fp}%
\special{pa 740 1155}%
\special{pa 1190 1005}%
\special{fp}%
\special{pa 1190 1005}%
\special{pa 870 765}%
\special{fp}%
}}%
\put(9.9000,-2.7500){\makebox(0,0){{\color[named]{Black}{$\alpha$}}}}%
\put(5.6000,-11.5500){\makebox(0,0){{\color[named]{Black}{$\gamma$}}}}%
\put(9.9000,-14.9500){\makebox(0,0){{\color[named]{Black}{$\vdots$}}}}%
\put(14.4000,-9.8500){\makebox(0,0){{\color[named]{Black}{$\beta$}}}}%
%
{\color[named]{Black}{%
\special{pn 20}%
\special{pa 4200 390}%
\special{pa 4600 390}%
\special{pa 4600 790}%
\special{pa 4200 790}%
\special{pa 4200 390}%
\special{pa 4600 390}%
\special{fp}%
}}%
\put(44.0000,-5.9000){\makebox(0,0){{\color[named]{Black}{$a_{n-2}$}}}}%
%
{\color[named]{Black}{%
\special{pn 8}%
\special{pa 1000 2000}%
\special{pa 1400 3200}%
\special{fp}%
\special{pa 1000 2000}%
\special{pa 600 3200}%
\special{fp}%
}}%
%
{\color[named]{Black}{%
\special{pn 8}%
\special{pa 1350 3050}%
\special{pa 720 2840}%
\special{fp}%
\special{pa 720 2840}%
\special{pa 1210 2630}%
\special{fp}%
\special{pa 720 2840}%
\special{pa 1130 2400}%
\special{fp}%
\special{pa 1130 2400}%
\special{pa 920 2240}%
\special{fp}%
}}%
\put(10.0000,-18.9000){\makebox(0,0){{\color[named]{Black}{$\alpha$}}}}%
\put(14.0000,-24.0000){\makebox(0,0){{\color[named]{Black}{$\beta$}}}}%
\put(6.0000,-28.0000){\makebox(0,0){{\color[named]{Black}{$\gamma$}}}}%
\put(10.0000,-30.9000){\makebox(0,0){{\color[named]{Black}{$\vdots$}}}}%
%
{\color[named]{Black}{%
\special{pn 20}%
\special{pa 2400 405}%
\special{pa 2800 405}%
\special{pa 2800 805}%
\special{pa 2400 805}%
\special{pa 2400 405}%
\special{pa 2800 405}%
\special{fp}%
}}%
%
{\color[named]{Black}{%
\special{pn 8}%
\special{ar 2800 1005 200 200 4.7123890 1.5707963}%
}}%
%
{\color[named]{Black}{%
\special{pn 8}%
\special{pa 2815 1204}%
\special{pa 2800 1205}%
\special{fp}%
\special{sh 1}%
\special{pa 2800 1205}%
\special{pa 2868 1221}%
\special{pa 2853 1201}%
\special{pa 2865 1181}%
\special{pa 2800 1205}%
\special{fp}%
}}%
\put(26.0000,-6.0500){\makebox(0,0){{\color[named]{Black}{$a_{n-2}$}}}}%
%
{\color[named]{Black}{%
\special{pn 20}%
\special{pa 4200 1996}%
\special{pa 4600 1996}%
\special{pa 4600 2396}%
\special{pa 4200 2396}%
\special{pa 4200 1996}%
\special{pa 4600 1996}%
\special{fp}%
}}%
\put(44.0000,-21.9600){\makebox(0,0){{\color[named]{Black}{$a_{n-2}$}}}}%
%
{\color[named]{Black}{%
\special{pn 8}%
\special{pa 2800 1205}%
\special{pa 2400 1205}%
\special{fp}%
\special{sh 1}%
\special{pa 2400 1205}%
\special{pa 2467 1225}%
\special{pa 2453 1205}%
\special{pa 2467 1185}%
\special{pa 2400 1205}%
\special{fp}%
}}%
\put(22.0000,-20.0000){\makebox(0,0){{\color[named]{Black}{$\cdots$}}}}%
\put(22.0000,-24.0000){\makebox(0,0){{\color[named]{Black}{$\cdots$}}}}%
\put(22.0000,-28.0000){\makebox(0,0){{\color[named]{Black}{$\cdots$}}}}%
\put(22.0000,-32.0000){\makebox(0,0){{\color[named]{Black}{$\cdots$}}}}%
\put(22.0000,-4.0500){\makebox(0,0){{\color[named]{Black}{$\cdots$}}}}%
\put(22.0000,-8.0500){\makebox(0,0){{\color[named]{Black}{$\cdots$}}}}%
\put(22.0000,-12.0500){\makebox(0,0){{\color[named]{Black}{$\cdots$}}}}%
\put(22.0000,-16.0500){\makebox(0,0){{\color[named]{Black}{$\cdots$}}}}%
\put(40.0000,-3.9500){\makebox(0,0){{\color[named]{Black}{$\cdots$}}}}%
\put(40.0000,-7.9500){\makebox(0,0){{\color[named]{Black}{$\cdots$}}}}%
\put(40.0000,-11.9500){\makebox(0,0){{\color[named]{Black}{$\cdots$}}}}%
\put(40.0000,-15.9500){\makebox(0,0){{\color[named]{Black}{$\cdots$}}}}%
\put(40.0000,-19.9500){\makebox(0,0){{\color[named]{Black}{$\cdots$}}}}%
\put(40.0000,-23.9500){\makebox(0,0){{\color[named]{Black}{$\cdots$}}}}%
\put(40.0000,-27.9500){\makebox(0,0){{\color[named]{Black}{$\cdots$}}}}%
\put(40.0000,-31.9500){\makebox(0,0){{\color[named]{Black}{$\cdots$}}}}%
%
{\color[named]{Black}{%
\special{pn 8}%
\special{ar 2800 1005 600 600 4.7123890 1.5707963}%
}}%
%
{\color[named]{Black}{%
\special{pn 8}%
\special{pa 2815 1605}%
\special{pa 2800 1605}%
\special{fp}%
\special{sh 1}%
\special{pa 2800 1605}%
\special{pa 2867 1625}%
\special{pa 2853 1605}%
\special{pa 2867 1585}%
\special{pa 2800 1605}%
\special{fp}%
}}%
%
{\color[named]{Black}{%
\special{pn 8}%
\special{pa 2800 1605}%
\special{pa 2400 1605}%
\special{fp}%
\special{sh 1}%
\special{pa 2400 1605}%
\special{pa 2467 1625}%
\special{pa 2453 1605}%
\special{pa 2467 1585}%
\special{pa 2400 1605}%
\special{fp}%
}}%
%
{\color[named]{Black}{%
\special{pn 20}%
\special{pa 2400 2000}%
\special{pa 2800 2000}%
\special{pa 2800 2400}%
\special{pa 2400 2400}%
\special{pa 2400 2000}%
\special{pa 2800 2000}%
\special{fp}%
}}%
%
{\color[named]{Black}{%
\special{pn 8}%
\special{ar 2800 2600 200 200 4.7123890 1.5707963}%
}}%
%
{\color[named]{Black}{%
\special{pn 8}%
\special{pa 2815 2799}%
\special{pa 2800 2800}%
\special{fp}%
\special{sh 1}%
\special{pa 2800 2800}%
\special{pa 2868 2816}%
\special{pa 2853 2796}%
\special{pa 2865 2776}%
\special{pa 2800 2800}%
\special{fp}%
}}%
\put(26.0000,-22.0000){\makebox(0,0){{\color[named]{Black}{$a_{n-2}$}}}}%
%
{\color[named]{Black}{%
\special{pn 8}%
\special{pa 2800 2800}%
\special{pa 2400 2800}%
\special{fp}%
\special{sh 1}%
\special{pa 2400 2800}%
\special{pa 2467 2820}%
\special{pa 2453 2800}%
\special{pa 2467 2780}%
\special{pa 2400 2800}%
\special{fp}%
}}%
%
{\color[named]{Black}{%
\special{pn 8}%
\special{ar 2800 2600 600 600 4.7123890 1.5707963}%
}}%
%
{\color[named]{Black}{%
\special{pn 8}%
\special{pa 2815 2000}%
\special{pa 2800 2000}%
\special{fp}%
\special{sh 1}%
\special{pa 2800 2000}%
\special{pa 2867 2020}%
\special{pa 2853 2000}%
\special{pa 2867 1980}%
\special{pa 2800 2000}%
\special{fp}%
}}%
%
{\color[named]{Black}{%
\special{pn 8}%
\special{pa 2400 3200}%
\special{pa 2800 3200}%
\special{fp}%
\special{sh 1}%
\special{pa 2800 3200}%
\special{pa 2733 3180}%
\special{pa 2747 3200}%
\special{pa 2733 3220}%
\special{pa 2800 3200}%
\special{fp}%
}}%
\put(7.4000,-7.5000){\makebox(0,0){{\color[named]{Black}{$\cfrac{1}{0}$}}}}%
%
{\color[named]{Black}{%
\special{pn 20}%
\special{pa 990 400}%
\special{pa 740 1150}%
\special{fp}%
\special{pa 740 1150}%
\special{pa 1340 1450}%
\special{fp}%
}}%
\put(8.0000,-22.0000){\makebox(0,0){{\color[named]{Black}{$\cfrac{1}{0}$}}}}%
%
{\color[named]{Black}{%
\special{pn 20}%
\special{pa 1000 2000}%
\special{pa 700 2900}%
\special{fp}%
}}%
%
{\color[named]{Black}{%
\special{pn 20}%
\special{pa 700 2900}%
\special{pa 680 2960}%
\special{fp}%
}}%
%
{\color[named]{Black}{%
\special{pn 8}%
\special{pa 4600 800}%
\special{pa 4800 800}%
\special{fp}%
\special{sh 1}%
\special{pa 4800 800}%
\special{pa 4733 780}%
\special{pa 4747 800}%
\special{pa 4733 820}%
\special{pa 4800 800}%
\special{fp}%
\special{pa 5050 1050}%
\special{pa 5200 1200}%
\special{fp}%
\special{sh 1}%
\special{pa 5200 1200}%
\special{pa 5167 1139}%
\special{pa 5162 1162}%
\special{pa 5139 1167}%
\special{pa 5200 1200}%
\special{fp}%
\special{pa 5200 1200}%
\special{pa 6200 1200}%
\special{fp}%
\special{sh 1}%
\special{pa 6200 1200}%
\special{pa 6133 1180}%
\special{pa 6147 1200}%
\special{pa 6133 1220}%
\special{pa 6200 1200}%
\special{fp}%
\special{pa 6200 800}%
\special{pa 5800 400}%
\special{fp}%
\special{sh 1}%
\special{pa 5800 400}%
\special{pa 5833 461}%
\special{pa 5838 438}%
\special{pa 5861 433}%
\special{pa 5800 400}%
\special{fp}%
\special{pa 5550 650}%
\special{pa 5400 800}%
\special{fp}%
\special{sh 1}%
\special{pa 5400 800}%
\special{pa 5461 767}%
\special{pa 5438 762}%
\special{pa 5433 739}%
\special{pa 5400 800}%
\special{fp}%
\special{pa 5400 800}%
\special{pa 5200 800}%
\special{fp}%
\special{sh 1}%
\special{pa 5200 800}%
\special{pa 5267 820}%
\special{pa 5253 800}%
\special{pa 5267 780}%
\special{pa 5200 800}%
\special{fp}%
\special{pa 5200 800}%
\special{pa 4800 1200}%
\special{fp}%
\special{sh 1}%
\special{pa 4800 1200}%
\special{pa 4861 1167}%
\special{pa 4838 1162}%
\special{pa 4833 1139}%
\special{pa 4800 1200}%
\special{fp}%
\special{pa 4800 1200}%
\special{pa 4200 1200}%
\special{fp}%
\special{sh 1}%
\special{pa 4200 1200}%
\special{pa 4267 1220}%
\special{pa 4253 1200}%
\special{pa 4267 1180}%
\special{pa 4200 1200}%
\special{fp}%
\special{pa 4600 400}%
\special{pa 5400 400}%
\special{fp}%
\special{sh 1}%
\special{pa 5400 400}%
\special{pa 5333 380}%
\special{pa 5347 400}%
\special{pa 5333 420}%
\special{pa 5400 400}%
\special{fp}%
\special{pa 5400 400}%
\special{pa 5800 800}%
\special{fp}%
\special{sh 1}%
\special{pa 5800 800}%
\special{pa 5767 739}%
\special{pa 5762 762}%
\special{pa 5739 767}%
\special{pa 5800 800}%
\special{fp}%
\special{pa 6050 550}%
\special{pa 6200 400}%
\special{fp}%
\special{sh 1}%
\special{pa 6200 400}%
\special{pa 6139 433}%
\special{pa 6162 438}%
\special{pa 6167 461}%
\special{pa 6200 400}%
\special{fp}%
\special{pa 6200 1600}%
\special{pa 4200 1600}%
\special{fp}%
\special{sh 1}%
\special{pa 4200 1600}%
\special{pa 4267 1620}%
\special{pa 4253 1600}%
\special{pa 4267 1580}%
\special{pa 4200 1600}%
\special{fp}%
}}%
%
{\color[named]{Black}{%
\special{pn 8}%
\special{pa 4800 800}%
\special{pa 4950 950}%
\special{fp}%
\special{pa 5800 400}%
\special{pa 5650 550}%
\special{fp}%
\special{pa 5800 800}%
\special{pa 5950 650}%
\special{fp}%
}}%
%
{\color[named]{Black}{%
\special{pn 8}%
\special{ar 6200 1000 200 200 4.7123890 1.5707963}%
}}%
%
{\color[named]{Black}{%
\special{pn 8}%
\special{pa 6215 801}%
\special{pa 6200 800}%
\special{fp}%
\special{sh 1}%
\special{pa 6200 800}%
\special{pa 6265 824}%
\special{pa 6253 804}%
\special{pa 6268 784}%
\special{pa 6200 800}%
\special{fp}%
}}%
%
{\color[named]{Black}{%
\special{pn 8}%
\special{ar 6200 1000 600 600 4.7123890 1.5707963}%
}}%
%
{\color[named]{Black}{%
\special{pn 8}%
\special{pa 6215 1600}%
\special{pa 6200 1600}%
\special{fp}%
\special{sh 1}%
\special{pa 6200 1600}%
\special{pa 6267 1620}%
\special{pa 6253 1600}%
\special{pa 6267 1580}%
\special{pa 6200 1600}%
\special{fp}%
}}%
%
{\color[named]{Black}{%
\special{pn 8}%
\special{pa 4600 2400}%
\special{pa 4800 2400}%
\special{fp}%
\special{sh 1}%
\special{pa 4800 2400}%
\special{pa 4733 2380}%
\special{pa 4747 2400}%
\special{pa 4733 2420}%
\special{pa 4800 2400}%
\special{fp}%
\special{pa 5050 2650}%
\special{pa 5200 2800}%
\special{fp}%
\special{sh 1}%
\special{pa 5200 2800}%
\special{pa 5167 2739}%
\special{pa 5162 2762}%
\special{pa 5139 2767}%
\special{pa 5200 2800}%
\special{fp}%
\special{pa 5200 2800}%
\special{pa 5600 2400}%
\special{fp}%
\special{sh 1}%
\special{pa 5600 2400}%
\special{pa 5539 2433}%
\special{pa 5562 2438}%
\special{pa 5567 2461}%
\special{pa 5600 2400}%
\special{fp}%
\special{pa 5600 2400}%
\special{pa 5800 2400}%
\special{fp}%
\special{sh 1}%
\special{pa 5800 2400}%
\special{pa 5733 2380}%
\special{pa 5747 2400}%
\special{pa 5733 2420}%
\special{pa 5800 2400}%
\special{fp}%
\special{pa 6050 2150}%
\special{pa 6200 2000}%
\special{fp}%
\special{sh 1}%
\special{pa 6200 2000}%
\special{pa 6139 2033}%
\special{pa 6162 2038}%
\special{pa 6167 2061}%
\special{pa 6200 2000}%
\special{fp}%
\special{pa 6200 2000}%
\special{pa 6600 2400}%
\special{fp}%
\special{sh 1}%
\special{pa 6600 2400}%
\special{pa 6567 2339}%
\special{pa 6562 2362}%
\special{pa 6539 2367}%
\special{pa 6600 2400}%
\special{fp}%
\special{pa 6600 2800}%
\special{pa 5600 2800}%
\special{fp}%
\special{sh 1}%
\special{pa 5600 2800}%
\special{pa 5667 2820}%
\special{pa 5653 2800}%
\special{pa 5667 2780}%
\special{pa 5600 2800}%
\special{fp}%
\special{pa 5350 2550}%
\special{pa 5200 2400}%
\special{fp}%
\special{sh 1}%
\special{pa 5200 2400}%
\special{pa 5233 2461}%
\special{pa 5238 2438}%
\special{pa 5261 2433}%
\special{pa 5200 2400}%
\special{fp}%
\special{pa 5200 2400}%
\special{pa 4800 2800}%
\special{fp}%
\special{sh 1}%
\special{pa 4800 2800}%
\special{pa 4861 2767}%
\special{pa 4838 2762}%
\special{pa 4833 2739}%
\special{pa 4800 2800}%
\special{fp}%
\special{pa 4800 2800}%
\special{pa 4200 2800}%
\special{fp}%
\special{sh 1}%
\special{pa 4200 2800}%
\special{pa 4267 2820}%
\special{pa 4253 2800}%
\special{pa 4267 2780}%
\special{pa 4200 2800}%
\special{fp}%
\special{pa 4200 3200}%
\special{pa 6600 3200}%
\special{fp}%
\special{sh 1}%
\special{pa 6600 3200}%
\special{pa 6533 3180}%
\special{pa 6547 3200}%
\special{pa 6533 3220}%
\special{pa 6600 3200}%
\special{fp}%
\special{pa 6350 2250}%
\special{pa 6200 2400}%
\special{fp}%
\special{sh 1}%
\special{pa 6200 2400}%
\special{pa 6261 2367}%
\special{pa 6238 2362}%
\special{pa 6233 2339}%
\special{pa 6200 2400}%
\special{fp}%
\special{pa 6200 2400}%
\special{pa 5800 2000}%
\special{fp}%
\special{sh 1}%
\special{pa 5800 2000}%
\special{pa 5833 2061}%
\special{pa 5838 2038}%
\special{pa 5861 2033}%
\special{pa 5800 2000}%
\special{fp}%
\special{pa 5800 2000}%
\special{pa 4600 2000}%
\special{fp}%
\special{sh 1}%
\special{pa 4600 2000}%
\special{pa 4667 2020}%
\special{pa 4653 2000}%
\special{pa 4667 1980}%
\special{pa 4600 2000}%
\special{fp}%
}}%
%
{\color[named]{Black}{%
\special{pn 8}%
\special{pa 4800 2400}%
\special{pa 4950 2550}%
\special{fp}%
\special{pa 5600 2800}%
\special{pa 5600 2800}%
\special{fp}%
\special{pa 5450 2650}%
\special{pa 5600 2800}%
\special{fp}%
\special{pa 5800 2400}%
\special{pa 5950 2250}%
\special{fp}%
\special{pa 6450 2150}%
\special{pa 6600 2000}%
\special{fp}%
}}%
%
{\color[named]{Black}{%
\special{pn 8}%
\special{ar 6600 2600 200 200 4.7123890 1.5707963}%
}}%
%
{\color[named]{Black}{%
\special{pn 8}%
\special{pa 6615 2799}%
\special{pa 6600 2800}%
\special{fp}%
\special{sh 1}%
\special{pa 6600 2800}%
\special{pa 6668 2816}%
\special{pa 6653 2796}%
\special{pa 6665 2776}%
\special{pa 6600 2800}%
\special{fp}%
}}%
%
{\color[named]{Black}{%
\special{pn 8}%
\special{ar 6600 2600 600 600 4.7123890 1.5707963}%
}}%
%
{\color[named]{Black}{%
\special{pn 8}%
\special{pa 6615 2000}%
\special{pa 6600 2000}%
\special{fp}%
\special{sh 1}%
\special{pa 6600 2000}%
\special{pa 6667 2020}%
\special{pa 6653 2000}%
\special{pa 6667 1980}%
\special{pa 6600 2000}%
\special{fp}%
}}%
\end{picture}}%

%% file: proof_of_sign_10.tex
{\unitlength 0.1in%
\begin{picture}(63.4000,30.7200)(-1.4000,-32.2700)%
%
{\color[named]{Black}{%
\special{pn 8}%
\special{pa 1000 400}%
\special{pa 600 1600}%
\special{fp}%
\special{pa 1000 400}%
\special{pa 1400 1600}%
\special{fp}%
}}%
%
{\color[named]{Black}{%
\special{pn 8}%
\special{pa 870 800}%
\special{pa 1080 640}%
\special{fp}%
}}%
%
{\color[named]{Black}{%
\special{pn 8}%
\special{pa 1280 1240}%
\special{pa 870 800}%
\special{fp}%
}}%
%
{\color[named]{Black}{%
\special{pn 8}%
\special{pa 1280 1240}%
\special{pa 790 1030}%
\special{fp}%
}}%
\put(10.0000,-2.2000){\makebox(0,0){{\color[named]{Black}{$\cfrac{1}{0}$}}}}%
\put(6.0000,-8.0000){\makebox(0,0){{\color[named]{Black}{$\gamma \equiv \cfrac{0}{1}$}}}}%
%
{\color[named]{Black}{%
\special{pn 20}%
\special{pa 1000 400}%
\special{pa 870 790}%
\special{fp}%
\special{pa 870 790}%
\special{pa 1280 1240}%
\special{fp}%
}}%
\put(10.0000,-14.1000){\makebox(0,0){{\color[named]{Black}{$\vdots$}}}}%
\put(22.0000,-4.0000){\makebox(0,0){{\color[named]{Black}{$\cdots$}}}}%
\put(22.0000,-8.0000){\makebox(0,0){{\color[named]{Black}{$\cdots$}}}}%
\put(22.0000,-12.0000){\makebox(0,0){{\color[named]{Black}{$\cdots$}}}}%
\put(22.0000,-16.0000){\makebox(0,0){{\color[named]{Black}{$\cdots$}}}}%
\put(44.0000,-3.9500){\makebox(0,0){{\color[named]{Black}{$\cdots$}}}}%
\put(44.0000,-7.9500){\makebox(0,0){{\color[named]{Black}{$\cdots$}}}}%
\put(44.0000,-11.9500){\makebox(0,0){{\color[named]{Black}{$\cdots$}}}}%
\put(44.0000,-15.9500){\makebox(0,0){{\color[named]{Black}{$\cdots$}}}}%
\put(14.0000,-11.8000){\makebox(0,0){{\color[named]{Black}{$\cfrac{1}{0}$}}}}%
%
{\color[named]{Black}{%
\special{pn 20}%
\special{pa 2400 400}%
\special{pa 2800 400}%
\special{pa 2800 800}%
\special{pa 2400 800}%
\special{pa 2400 400}%
\special{pa 2800 400}%
\special{fp}%
}}%
\put(26.0000,-6.0000){\makebox(0,0){{\color[named]{Black}{$a_{n-1}$}}}}%
%
{\color[named]{Black}{%
\special{pn 8}%
\special{ar 2800 1000 200 200 4.7123890 1.5707963}%
}}%
%
{\color[named]{Black}{%
\special{pn 8}%
\special{pa 2815 1199}%
\special{pa 2800 1200}%
\special{fp}%
\special{sh 1}%
\special{pa 2800 1200}%
\special{pa 2868 1216}%
\special{pa 2853 1196}%
\special{pa 2865 1176}%
\special{pa 2800 1200}%
\special{fp}%
}}%
%
{\color[named]{Black}{%
\special{pn 8}%
\special{ar 2800 1000 600 600 4.7123890 1.5707963}%
}}%
%
{\color[named]{Black}{%
\special{pn 8}%
\special{pa 2815 400}%
\special{pa 2800 400}%
\special{fp}%
\special{sh 1}%
\special{pa 2800 400}%
\special{pa 2867 420}%
\special{pa 2853 400}%
\special{pa 2867 380}%
\special{pa 2800 400}%
\special{fp}%
}}%
%
{\color[named]{Black}{%
\special{pn 8}%
\special{pa 2800 1200}%
\special{pa 2400 1200}%
\special{fp}%
\special{sh 1}%
\special{pa 2400 1200}%
\special{pa 2467 1220}%
\special{pa 2453 1200}%
\special{pa 2467 1180}%
\special{pa 2400 1200}%
\special{fp}%
}}%
%
{\color[named]{Black}{%
\special{pn 8}%
\special{pa 2400 1600}%
\special{pa 2800 1600}%
\special{fp}%
\special{sh 1}%
\special{pa 2800 1600}%
\special{pa 2733 1580}%
\special{pa 2747 1600}%
\special{pa 2733 1620}%
\special{pa 2800 1600}%
\special{fp}%
}}%
%
{\color[named]{Black}{%
\special{pn 20}%
\special{pa 4600 400}%
\special{pa 5000 400}%
\special{pa 5000 800}%
\special{pa 4600 800}%
\special{pa 4600 400}%
\special{pa 5000 400}%
\special{fp}%
}}%
\put(48.0000,-6.0000){\makebox(0,0){{\color[named]{Black}{$a_{n-1}$}}}}%
%
{\color[named]{Black}{%
\special{pn 8}%
\special{pa 5000 800}%
\special{pa 5200 800}%
\special{fp}%
\special{sh 1}%
\special{pa 5200 800}%
\special{pa 5133 780}%
\special{pa 5147 800}%
\special{pa 5133 820}%
\special{pa 5200 800}%
\special{fp}%
\special{pa 5450 1050}%
\special{pa 5600 1200}%
\special{fp}%
\special{sh 1}%
\special{pa 5600 1200}%
\special{pa 5567 1139}%
\special{pa 5562 1162}%
\special{pa 5539 1167}%
\special{pa 5600 1200}%
\special{fp}%
\special{pa 5600 1200}%
\special{pa 6000 800}%
\special{fp}%
\special{sh 1}%
\special{pa 6000 800}%
\special{pa 5939 833}%
\special{pa 5962 838}%
\special{pa 5967 861}%
\special{pa 6000 800}%
\special{fp}%
\special{pa 6000 400}%
\special{pa 5000 400}%
\special{fp}%
\special{sh 1}%
\special{pa 5000 400}%
\special{pa 5067 420}%
\special{pa 5053 400}%
\special{pa 5067 380}%
\special{pa 5000 400}%
\special{fp}%
\special{pa 5600 800}%
\special{pa 5200 1200}%
\special{fp}%
\special{sh 1}%
\special{pa 5200 1200}%
\special{pa 5261 1167}%
\special{pa 5238 1162}%
\special{pa 5233 1139}%
\special{pa 5200 1200}%
\special{fp}%
\special{pa 5750 950}%
\special{pa 5600 800}%
\special{fp}%
\special{sh 1}%
\special{pa 5600 800}%
\special{pa 5633 861}%
\special{pa 5638 838}%
\special{pa 5661 833}%
\special{pa 5600 800}%
\special{fp}%
}}%
%
{\color[named]{Black}{%
\special{pn 8}%
\special{pa 4600 1600}%
\special{pa 6000 1600}%
\special{fp}%
\special{sh 1}%
\special{pa 6000 1600}%
\special{pa 5933 1580}%
\special{pa 5947 1600}%
\special{pa 5933 1620}%
\special{pa 6000 1600}%
\special{fp}%
\special{pa 5200 1200}%
\special{pa 4600 1200}%
\special{fp}%
\special{sh 1}%
\special{pa 4600 1200}%
\special{pa 4667 1220}%
\special{pa 4653 1200}%
\special{pa 4667 1180}%
\special{pa 4600 1200}%
\special{fp}%
}}%
%
{\color[named]{Black}{%
\special{pn 8}%
\special{pa 5200 800}%
\special{pa 5350 950}%
\special{fp}%
\special{pa 6000 1200}%
\special{pa 5850 1050}%
\special{fp}%
}}%
%
{\color[named]{Black}{%
\special{pn 8}%
\special{ar 6000 1400 200 200 4.7123890 1.5707963}%
}}%
%
{\color[named]{Black}{%
\special{pn 8}%
\special{pa 6015 1201}%
\special{pa 6000 1200}%
\special{fp}%
\special{sh 1}%
\special{pa 6000 1200}%
\special{pa 6065 1224}%
\special{pa 6053 1204}%
\special{pa 6068 1184}%
\special{pa 6000 1200}%
\special{fp}%
}}%
%
{\color[named]{Black}{%
\special{pn 8}%
\special{ar 6000 600 200 200 4.7123890 1.5707963}%
}}%
%
{\color[named]{Black}{%
\special{pn 8}%
\special{pa 6015 401}%
\special{pa 6000 400}%
\special{fp}%
\special{sh 1}%
\special{pa 6000 400}%
\special{pa 6065 424}%
\special{pa 6053 404}%
\special{pa 6068 384}%
\special{pa 6000 400}%
\special{fp}%
}}%
\put(22.0000,-20.0000){\makebox(0,0){{\color[named]{Black}{$\cdots$}}}}%
\put(22.0000,-24.0000){\makebox(0,0){{\color[named]{Black}{$\cdots$}}}}%
\put(22.0000,-28.0000){\makebox(0,0){{\color[named]{Black}{$\cdots$}}}}%
\put(22.0000,-32.0000){\makebox(0,0){{\color[named]{Black}{$\cdots$}}}}%
\put(44.0000,-19.9500){\makebox(0,0){{\color[named]{Black}{$\cdots$}}}}%
\put(44.0000,-23.9500){\makebox(0,0){{\color[named]{Black}{$\cdots$}}}}%
\put(44.0000,-27.9500){\makebox(0,0){{\color[named]{Black}{$\cdots$}}}}%
\put(44.0000,-31.9500){\makebox(0,0){{\color[named]{Black}{$\cdots$}}}}%
%
{\color[named]{Black}{%
\special{pn 20}%
\special{pa 2400 2000}%
\special{pa 2800 2000}%
\special{pa 2800 2400}%
\special{pa 2400 2400}%
\special{pa 2400 2000}%
\special{pa 2800 2000}%
\special{fp}%
}}%
\put(26.0000,-22.0000){\makebox(0,0){{\color[named]{Black}{$a_{n-1}$}}}}%
%
{\color[named]{Black}{%
\special{pn 8}%
\special{ar 2800 2600 200 200 4.7123890 1.5707963}%
}}%
%
{\color[named]{Black}{%
\special{pn 8}%
\special{pa 2815 2401}%
\special{pa 2800 2400}%
\special{fp}%
\special{sh 1}%
\special{pa 2800 2400}%
\special{pa 2865 2424}%
\special{pa 2853 2404}%
\special{pa 2868 2384}%
\special{pa 2800 2400}%
\special{fp}%
}}%
%
{\color[named]{Black}{%
\special{pn 8}%
\special{ar 2800 2600 600 600 4.7123890 1.5707963}%
}}%
%
{\color[named]{Black}{%
\special{pn 8}%
\special{pa 2815 3200}%
\special{pa 2800 3200}%
\special{fp}%
\special{sh 1}%
\special{pa 2800 3200}%
\special{pa 2867 3220}%
\special{pa 2853 3200}%
\special{pa 2867 3180}%
\special{pa 2800 3200}%
\special{fp}%
}}%
%
{\color[named]{Black}{%
\special{pn 8}%
\special{pa 2400 2800}%
\special{pa 2800 2800}%
\special{fp}%
\special{sh 1}%
\special{pa 2800 2800}%
\special{pa 2733 2780}%
\special{pa 2747 2800}%
\special{pa 2733 2820}%
\special{pa 2800 2800}%
\special{fp}%
}}%
%
{\color[named]{Black}{%
\special{pn 8}%
\special{pa 2800 3200}%
\special{pa 2400 3200}%
\special{fp}%
\special{sh 1}%
\special{pa 2400 3200}%
\special{pa 2467 3220}%
\special{pa 2453 3200}%
\special{pa 2467 3180}%
\special{pa 2400 3200}%
\special{fp}%
}}%
%
{\color[named]{Black}{%
\special{pn 20}%
\special{pa 4600 2000}%
\special{pa 5000 2000}%
\special{pa 5000 2400}%
\special{pa 4600 2400}%
\special{pa 4600 2000}%
\special{pa 5000 2000}%
\special{fp}%
}}%
\put(48.0000,-22.0000){\makebox(0,0){{\color[named]{Black}{$a_{n-1}$}}}}%
%
{\color[named]{Black}{%
\special{pn 8}%
\special{pa 5000 2000}%
\special{pa 6000 2000}%
\special{fp}%
\special{sh 1}%
\special{pa 6000 2000}%
\special{pa 5933 1980}%
\special{pa 5947 2000}%
\special{pa 5933 2020}%
\special{pa 6000 2000}%
\special{fp}%
\special{pa 6000 2400}%
\special{pa 5600 2800}%
\special{fp}%
\special{sh 1}%
\special{pa 5600 2800}%
\special{pa 5661 2767}%
\special{pa 5638 2762}%
\special{pa 5633 2739}%
\special{pa 5600 2800}%
\special{fp}%
\special{pa 5350 2550}%
\special{pa 5200 2400}%
\special{fp}%
\special{sh 1}%
\special{pa 5200 2400}%
\special{pa 5233 2461}%
\special{pa 5238 2438}%
\special{pa 5261 2433}%
\special{pa 5200 2400}%
\special{fp}%
\special{pa 5200 2400}%
\special{pa 5000 2400}%
\special{fp}%
\special{sh 1}%
\special{pa 5000 2400}%
\special{pa 5067 2420}%
\special{pa 5053 2400}%
\special{pa 5067 2380}%
\special{pa 5000 2400}%
\special{fp}%
\special{pa 4600 2800}%
\special{pa 5200 2800}%
\special{fp}%
\special{sh 1}%
\special{pa 5200 2800}%
\special{pa 5133 2780}%
\special{pa 5147 2800}%
\special{pa 5133 2820}%
\special{pa 5200 2800}%
\special{fp}%
\special{pa 5200 2800}%
\special{pa 5600 2400}%
\special{fp}%
\special{sh 1}%
\special{pa 5600 2400}%
\special{pa 5539 2433}%
\special{pa 5562 2438}%
\special{pa 5567 2461}%
\special{pa 5600 2400}%
\special{fp}%
\special{pa 5850 2650}%
\special{pa 6000 2800}%
\special{fp}%
\special{sh 1}%
\special{pa 6000 2800}%
\special{pa 5967 2739}%
\special{pa 5962 2762}%
\special{pa 5939 2767}%
\special{pa 6000 2800}%
\special{fp}%
\special{pa 6000 3200}%
\special{pa 4600 3200}%
\special{fp}%
\special{sh 1}%
\special{pa 4600 3200}%
\special{pa 4667 3220}%
\special{pa 4653 3200}%
\special{pa 4667 3180}%
\special{pa 4600 3200}%
\special{fp}%
}}%
%
{\color[named]{Black}{%
\special{pn 8}%
\special{pa 5600 2800}%
\special{pa 5450 2650}%
\special{fp}%
\special{pa 5600 2400}%
\special{pa 5750 2550}%
\special{fp}%
}}%
%
{\color[named]{Black}{%
\special{pn 8}%
\special{ar 6000 2200 200 200 4.7123890 1.5707963}%
}}%
%
{\color[named]{Black}{%
\special{pn 8}%
\special{pa 6015 2399}%
\special{pa 6000 2400}%
\special{fp}%
\special{sh 1}%
\special{pa 6000 2400}%
\special{pa 6068 2416}%
\special{pa 6053 2396}%
\special{pa 6065 2376}%
\special{pa 6000 2400}%
\special{fp}%
}}%
%
{\color[named]{Black}{%
\special{pn 8}%
\special{ar 6000 3000 200 200 4.7123890 1.5707963}%
}}%
%
{\color[named]{Black}{%
\special{pn 8}%
\special{pa 6015 3199}%
\special{pa 6000 3200}%
\special{fp}%
\special{sh 1}%
\special{pa 6000 3200}%
\special{pa 6068 3216}%
\special{pa 6053 3196}%
\special{pa 6065 3176}%
\special{pa 6000 3200}%
\special{fp}%
}}%
\end{picture}}%

%% file: example_of_Seifert_path.tex
{\unitlength 0.1in%
\begin{picture}(17.6000,22.6500)(8.4000,-26.0500)%
%
{\color[named]{Black}{%
\special{pn 8}%
\special{pa 1960 1015}%
\special{pa 1690 875}%
\special{fp}%
}}%
%
{\color[named]{Black}{%
\special{pn 8}%
\special{pa 1320 1805}%
\special{pa 2120 1405}%
\special{fp}%
}}%
%
{\color[named]{Black}{%
\special{pn 8}%
\special{pa 1800 605}%
\special{pa 1000 2605}%
\special{fp}%
}}%
%
{\color[named]{Black}{%
\special{pn 8}%
\special{pa 2600 2605}%
\special{pa 1800 605}%
\special{fp}%
}}%
\put(27.5000,-23.9500){\makebox(0,0)[rt]{{\color[named]{Black}{$\cfrac{1}{1}$}}}}%
\put(8.5000,-23.9500){\makebox(0,0)[lt]{{\color[named]{Black}{$\cfrac{0}{1}$}}}}%
\put(25.2000,-20.9500){\makebox(0,0){{\color[named]{Black}{$\cfrac{1}{2}$}}}}%
\put(12.0000,-18.0500){\makebox(0,0){{\color[named]{Black}{$\cfrac{1}{3}$}}}}%
\put(22.5000,-14.0500){\makebox(0,0){{\color[named]{Black}{$\cfrac{2}{5}$}}}}%
\put(20.7000,-9.7500){\makebox(0,0){{\color[named]{Black}{$\cfrac{3}{8}$}}}}%
\put(15.5000,-8.5500){\makebox(0,0){{\color[named]{Black}{$\cfrac{4}{11}$}}}}%
\put(18.0000,-4.0500){\makebox(0,0){{\color[named]{Black}{$\cfrac{7}{19}$}}}}%
%
{\color[named]{Black}{%
\special{pn 8}%
\special{pa 1000 2605}%
\special{pa 2600 2605}%
\special{fp}%
}}%
%
{\color[named]{Black}{%
\special{pn 8}%
\special{pa 1000 2605}%
\special{pa 2400 2105}%
\special{fp}%
}}%
%
{\color[named]{Black}{%
\special{pn 8}%
\special{pa 2400 2105}%
\special{pa 1330 1795}%
\special{fp}%
}}%
%
{\color[named]{Black}{%
\special{pn 8}%
\special{pa 1330 1795}%
\special{pa 1960 995}%
\special{fp}%
}}%
%
{\color[named]{Black}{%
\special{pn 20}%
\special{pa 1800 600}%
\special{pa 1958 1001}%
\special{fp}%
\special{pa 1958 1001}%
\special{pa 1330 1794}%
\special{fp}%
\special{pa 1330 1794}%
\special{pa 2398 2103}%
\special{fp}%
\special{pa 2398 2103}%
\special{pa 2600 2600}%
\special{fp}%
}}%
\end{picture}}%

%% file: skein_relation.tex
{\unitlength 0.1in%
\begin{picture}(28.0000,4.6000)(10.0000,-18.6000)%
%
{\color[named]{Black}{%
\special{pn 8}%
\special{pa 2600 1800}%
\special{pa 2200 1400}%
\special{fp}%
\special{sh 1}%
\special{pa 2200 1400}%
\special{pa 2233 1461}%
\special{pa 2238 1438}%
\special{pa 2261 1433}%
\special{pa 2200 1400}%
\special{fp}%
\special{pa 2450 1550}%
\special{pa 2600 1400}%
\special{fp}%
\special{sh 1}%
\special{pa 2600 1400}%
\special{pa 2539 1433}%
\special{pa 2562 1438}%
\special{pa 2567 1461}%
\special{pa 2600 1400}%
\special{fp}%
}}%
%
{\color[named]{Black}{%
\special{pn 8}%
\special{pa 1000 1800}%
\special{pa 1400 1400}%
\special{fp}%
\special{sh 1}%
\special{pa 1400 1400}%
\special{pa 1339 1433}%
\special{pa 1362 1438}%
\special{pa 1367 1461}%
\special{pa 1400 1400}%
\special{fp}%
\special{pa 1150 1550}%
\special{pa 1000 1400}%
\special{fp}%
\special{sh 1}%
\special{pa 1000 1400}%
\special{pa 1033 1461}%
\special{pa 1038 1438}%
\special{pa 1061 1433}%
\special{pa 1000 1400}%
\special{fp}%
}}%
%
{\color[named]{Black}{%
\special{pn 8}%
\special{pa 1400 1800}%
\special{pa 1250 1650}%
\special{fp}%
\special{pa 2200 1800}%
\special{pa 2350 1650}%
\special{fp}%
}}%
%
{\color[named]{Black}{%
\special{pn 8}%
\special{pa 3400 1800}%
\special{pa 3421 1771}%
\special{pa 3442 1743}%
\special{pa 3460 1714}%
\special{pa 3476 1686}%
\special{pa 3489 1657}%
\special{pa 3497 1628}%
\special{pa 3500 1600}%
\special{pa 3497 1571}%
\special{pa 3489 1542}%
\special{pa 3476 1514}%
\special{pa 3460 1485}%
\special{pa 3441 1457}%
\special{pa 3421 1428}%
\special{pa 3400 1400}%
\special{fp}%
}}%
%
{\color[named]{Black}{%
\special{pn 8}%
\special{pa 3421 1428}%
\special{pa 3400 1400}%
\special{fp}%
\special{sh 1}%
\special{pa 3400 1400}%
\special{pa 3424 1465}%
\special{pa 3432 1443}%
\special{pa 3456 1441}%
\special{pa 3400 1400}%
\special{fp}%
}}%
%
{\color[named]{Black}{%
\special{pn 8}%
\special{pa 3800 1800}%
\special{pa 3779 1771}%
\special{pa 3758 1743}%
\special{pa 3740 1714}%
\special{pa 3724 1686}%
\special{pa 3711 1657}%
\special{pa 3703 1628}%
\special{pa 3700 1600}%
\special{pa 3703 1571}%
\special{pa 3711 1542}%
\special{pa 3724 1514}%
\special{pa 3740 1485}%
\special{pa 3759 1457}%
\special{pa 3779 1428}%
\special{pa 3800 1400}%
\special{fp}%
}}%
%
{\color[named]{Black}{%
\special{pn 8}%
\special{pa 3779 1428}%
\special{pa 3800 1400}%
\special{fp}%
\special{sh 1}%
\special{pa 3800 1400}%
\special{pa 3744 1441}%
\special{pa 3768 1443}%
\special{pa 3776 1465}%
\special{pa 3800 1400}%
\special{fp}%
}}%
\put(12.0500,-19.2500){\makebox(0,0){{\color[named]{Black}{$L_+$}}}}%
\put(24.0500,-19.2500){\makebox(0,0){{\color[named]{Black}{$L_-$}}}}%
\put(36.0500,-19.2500){\makebox(0,0){{\color[named]{Black}{$L_0$}}}}%
\end{picture}}%

%% file: example_of_alexander_polynomial.tex
{\unitlength 0.1in%
\begin{picture}(57.0000,35.9900)(8.0000,-39.3400)%
\put(12.0000,-4.0000){\makebox(0,0){{\color[named]{Black}{$L_+$}}}}%
\put(28.0000,-4.0000){\makebox(0,0){{\color[named]{Black}{$L_-$}}}}%
\put(44.0000,-4.0000){\makebox(0,0){{\color[named]{Black}{$L_0$}}}}%
%
{\color[named]{Black}{%
\special{pn 8}%
\special{pa 950 3607}%
\special{pa 1250 3307}%
\special{fp}%
\special{sh 1}%
\special{pa 1250 3307}%
\special{pa 1189 3340}%
\special{pa 1212 3345}%
\special{pa 1217 3368}%
\special{pa 1250 3307}%
\special{fp}%
}}%
%
{\color[named]{Black}{%
\special{pn 8}%
\special{pa 1350 3307}%
\special{pa 1250 3307}%
\special{fp}%
}}%
%
{\color[named]{Black}{%
\special{pn 8}%
\special{pa 1150 3507}%
\special{pa 1250 3607}%
\special{fp}%
\special{sh 1}%
\special{pa 1250 3607}%
\special{pa 1217 3546}%
\special{pa 1212 3569}%
\special{pa 1189 3574}%
\special{pa 1250 3607}%
\special{fp}%
}}%
%
{\color[named]{Black}{%
\special{pn 8}%
\special{pa 950 3307}%
\special{pa 1050 3407}%
\special{fp}%
}}%
%
{\color[named]{Black}{%
\special{pn 8}%
\special{pa 1350 3307}%
\special{pa 1450 3207}%
\special{fp}%
}}%
%
{\color[named]{Black}{%
\special{pn 8}%
\special{pa 1650 3307}%
\special{pa 1350 3007}%
\special{fp}%
\special{sh 1}%
\special{pa 1350 3007}%
\special{pa 1383 3068}%
\special{pa 1388 3045}%
\special{pa 1411 3040}%
\special{pa 1350 3007}%
\special{fp}%
}}%
%
{\color[named]{Black}{%
\special{pn 8}%
\special{pa 1350 3007}%
\special{pa 950 3007}%
\special{fp}%
}}%
%
{\color[named]{Black}{%
\special{pn 8}%
\special{pa 1550 3107}%
\special{pa 1650 3007}%
\special{fp}%
\special{sh 1}%
\special{pa 1650 3007}%
\special{pa 1589 3040}%
\special{pa 1612 3045}%
\special{pa 1617 3068}%
\special{pa 1650 3007}%
\special{fp}%
}}%
%
{\color[named]{Black}{%
\special{pn 8}%
\special{pa 1750 3307}%
\special{pa 1650 3307}%
\special{fp}%
}}%
%
{\color[named]{Black}{%
\special{pn 8}%
\special{pa 1750 3607}%
\special{pa 2050 3307}%
\special{fp}%
\special{sh 1}%
\special{pa 2050 3307}%
\special{pa 1989 3340}%
\special{pa 2012 3345}%
\special{pa 2017 3368}%
\special{pa 2050 3307}%
\special{fp}%
\special{pa 2350 3307}%
\special{pa 2050 3607}%
\special{fp}%
\special{sh 1}%
\special{pa 2050 3607}%
\special{pa 2111 3574}%
\special{pa 2088 3569}%
\special{pa 2083 3546}%
\special{pa 2050 3607}%
\special{fp}%
}}%
%
{\color[named]{Black}{%
\special{pn 8}%
\special{pa 2050 3607}%
\special{pa 1950 3507}%
\special{fp}%
\special{pa 1750 3607}%
\special{pa 1250 3607}%
\special{fp}%
\special{pa 2050 3307}%
\special{pa 2150 3407}%
\special{fp}%
}}%
%
{\color[named]{Black}{%
\special{pn 8}%
\special{pa 1850 3407}%
\special{pa 1750 3307}%
\special{fp}%
\special{sh 1}%
\special{pa 1750 3307}%
\special{pa 1783 3368}%
\special{pa 1788 3345}%
\special{pa 1811 3340}%
\special{pa 1750 3307}%
\special{fp}%
\special{pa 2250 3507}%
\special{pa 2350 3607}%
\special{fp}%
\special{sh 1}%
\special{pa 2350 3607}%
\special{pa 2317 3546}%
\special{pa 2312 3569}%
\special{pa 2289 3574}%
\special{pa 2350 3607}%
\special{fp}%
}}%
%
{\color[named]{Black}{%
\special{pn 8}%
\special{pa 1650 3007}%
\special{pa 2350 3007}%
\special{fp}%
\special{sh 1}%
\special{pa 2350 3007}%
\special{pa 2283 2987}%
\special{pa 2297 3007}%
\special{pa 2283 3027}%
\special{pa 2350 3007}%
\special{fp}%
}}%
%
{\color[named]{Black}{%
\special{pn 8}%
\special{pa 2350 3907}%
\special{pa 950 3907}%
\special{fp}%
\special{sh 1}%
\special{pa 950 3907}%
\special{pa 1017 3927}%
\special{pa 1003 3907}%
\special{pa 1017 3887}%
\special{pa 950 3907}%
\special{fp}%
}}%
%
{\color[named]{Black}{%
\special{pn 8}%
\special{ar 950 3757 150 150 1.5707963 4.7123890}%
}}%
%
{\color[named]{Black}{%
\special{pn 8}%
\special{pa 935 3608}%
\special{pa 950 3607}%
\special{fp}%
\special{sh 1}%
\special{pa 950 3607}%
\special{pa 882 3591}%
\special{pa 897 3611}%
\special{pa 885 3631}%
\special{pa 950 3607}%
\special{fp}%
}}%
%
{\color[named]{Black}{%
\special{pn 8}%
\special{ar 950 3157 150 150 1.5707963 4.7123890}%
}}%
%
{\color[named]{Black}{%
\special{pn 8}%
\special{pa 935 3306}%
\special{pa 950 3307}%
\special{fp}%
\special{sh 1}%
\special{pa 950 3307}%
\special{pa 885 3283}%
\special{pa 897 3303}%
\special{pa 882 3323}%
\special{pa 950 3307}%
\special{fp}%
}}%
%
{\color[named]{Black}{%
\special{pn 8}%
\special{ar 2350 3157 150 150 4.7123890 1.5707963}%
}}%
%
{\color[named]{Black}{%
\special{pn 8}%
\special{pa 2365 3306}%
\special{pa 2350 3307}%
\special{fp}%
\special{sh 1}%
\special{pa 2350 3307}%
\special{pa 2418 3323}%
\special{pa 2403 3303}%
\special{pa 2415 3283}%
\special{pa 2350 3307}%
\special{fp}%
}}%
%
{\color[named]{Black}{%
\special{pn 8}%
\special{ar 2350 3757 150 150 4.7123890 1.5707963}%
}}%
%
{\color[named]{Black}{%
\special{pn 8}%
\special{pa 2365 3906}%
\special{pa 2350 3907}%
\special{fp}%
\special{sh 1}%
\special{pa 2350 3907}%
\special{pa 2418 3923}%
\special{pa 2403 3903}%
\special{pa 2415 3883}%
\special{pa 2350 3907}%
\special{fp}%
}}%
%
{\color[named]{Black}{%
\special{pn 8}%
\special{pa 2950 3607}%
\special{pa 3250 3307}%
\special{fp}%
\special{sh 1}%
\special{pa 3250 3307}%
\special{pa 3189 3340}%
\special{pa 3212 3345}%
\special{pa 3217 3368}%
\special{pa 3250 3307}%
\special{fp}%
}}%
%
{\color[named]{Black}{%
\special{pn 8}%
\special{pa 3350 3307}%
\special{pa 3250 3307}%
\special{fp}%
}}%
%
{\color[named]{Black}{%
\special{pn 8}%
\special{pa 3150 3507}%
\special{pa 3250 3607}%
\special{fp}%
\special{sh 1}%
\special{pa 3250 3607}%
\special{pa 3217 3546}%
\special{pa 3212 3569}%
\special{pa 3189 3574}%
\special{pa 3250 3607}%
\special{fp}%
}}%
%
{\color[named]{Black}{%
\special{pn 8}%
\special{pa 2950 3307}%
\special{pa 3050 3407}%
\special{fp}%
}}%
%
{\color[named]{Black}{%
\special{pn 8}%
\special{pa 3350 3307}%
\special{pa 3450 3207}%
\special{fp}%
}}%
%
{\color[named]{Black}{%
\special{pn 8}%
\special{pa 3650 3307}%
\special{pa 3350 3007}%
\special{fp}%
\special{sh 1}%
\special{pa 3350 3007}%
\special{pa 3383 3068}%
\special{pa 3388 3045}%
\special{pa 3411 3040}%
\special{pa 3350 3007}%
\special{fp}%
}}%
%
{\color[named]{Black}{%
\special{pn 8}%
\special{pa 3350 3007}%
\special{pa 2950 3007}%
\special{fp}%
}}%
%
{\color[named]{Black}{%
\special{pn 8}%
\special{pa 3550 3107}%
\special{pa 3650 3007}%
\special{fp}%
\special{sh 1}%
\special{pa 3650 3007}%
\special{pa 3589 3040}%
\special{pa 3612 3045}%
\special{pa 3617 3068}%
\special{pa 3650 3007}%
\special{fp}%
}}%
%
{\color[named]{Black}{%
\special{pn 8}%
\special{pa 3750 3307}%
\special{pa 3650 3307}%
\special{fp}%
}}%
%
{\color[named]{Black}{%
\special{pn 8}%
\special{pa 3750 3607}%
\special{pa 4050 3307}%
\special{fp}%
\special{sh 1}%
\special{pa 4050 3307}%
\special{pa 3989 3340}%
\special{pa 4012 3345}%
\special{pa 4017 3368}%
\special{pa 4050 3307}%
\special{fp}%
}}%
%
{\color[named]{Black}{%
\special{pn 8}%
\special{pa 4050 3607}%
\special{pa 3950 3507}%
\special{fp}%
}}%
%
{\color[named]{Black}{%
\special{pn 8}%
\special{pa 3750 3607}%
\special{pa 3250 3607}%
\special{fp}%
}}%
%
{\color[named]{Black}{%
\special{pn 8}%
\special{pa 3850 3407}%
\special{pa 3750 3307}%
\special{fp}%
\special{sh 1}%
\special{pa 3750 3307}%
\special{pa 3783 3368}%
\special{pa 3788 3345}%
\special{pa 3811 3340}%
\special{pa 3750 3307}%
\special{fp}%
}}%
%
{\color[named]{Black}{%
\special{pn 8}%
\special{pa 3650 3007}%
\special{pa 4350 3007}%
\special{fp}%
\special{sh 1}%
\special{pa 4350 3007}%
\special{pa 4283 2987}%
\special{pa 4297 3007}%
\special{pa 4283 3027}%
\special{pa 4350 3007}%
\special{fp}%
}}%
%
{\color[named]{Black}{%
\special{pn 8}%
\special{pa 4350 3907}%
\special{pa 2950 3907}%
\special{fp}%
\special{sh 1}%
\special{pa 2950 3907}%
\special{pa 3017 3927}%
\special{pa 3003 3907}%
\special{pa 3017 3887}%
\special{pa 2950 3907}%
\special{fp}%
}}%
%
{\color[named]{Black}{%
\special{pn 8}%
\special{ar 2950 3757 150 150 1.5707963 4.7123890}%
}}%
%
{\color[named]{Black}{%
\special{pn 8}%
\special{pa 2935 3608}%
\special{pa 2950 3607}%
\special{fp}%
\special{sh 1}%
\special{pa 2950 3607}%
\special{pa 2882 3591}%
\special{pa 2897 3611}%
\special{pa 2885 3631}%
\special{pa 2950 3607}%
\special{fp}%
}}%
%
{\color[named]{Black}{%
\special{pn 8}%
\special{ar 2950 3157 150 150 1.5707963 4.7123890}%
}}%
%
{\color[named]{Black}{%
\special{pn 8}%
\special{pa 2935 3306}%
\special{pa 2950 3307}%
\special{fp}%
\special{sh 1}%
\special{pa 2950 3307}%
\special{pa 2885 3283}%
\special{pa 2897 3303}%
\special{pa 2882 3323}%
\special{pa 2950 3307}%
\special{fp}%
}}%
%
{\color[named]{Black}{%
\special{pn 8}%
\special{ar 4350 3157 150 150 4.7123890 1.5707963}%
}}%
%
{\color[named]{Black}{%
\special{pn 8}%
\special{pa 4365 3306}%
\special{pa 4350 3307}%
\special{fp}%
\special{sh 1}%
\special{pa 4350 3307}%
\special{pa 4418 3323}%
\special{pa 4403 3303}%
\special{pa 4415 3283}%
\special{pa 4350 3307}%
\special{fp}%
}}%
%
{\color[named]{Black}{%
\special{pn 8}%
\special{ar 4350 3757 150 150 4.7123890 1.5707963}%
}}%
%
{\color[named]{Black}{%
\special{pn 8}%
\special{pa 4365 3906}%
\special{pa 4350 3907}%
\special{fp}%
\special{sh 1}%
\special{pa 4350 3907}%
\special{pa 4418 3923}%
\special{pa 4403 3903}%
\special{pa 4415 3883}%
\special{pa 4350 3907}%
\special{fp}%
}}%
%
{\color[named]{Black}{%
\special{pn 8}%
\special{pa 4950 3607}%
\special{pa 5250 3307}%
\special{fp}%
\special{sh 1}%
\special{pa 5250 3307}%
\special{pa 5189 3340}%
\special{pa 5212 3345}%
\special{pa 5217 3368}%
\special{pa 5250 3307}%
\special{fp}%
}}%
%
{\color[named]{Black}{%
\special{pn 8}%
\special{pa 5350 3307}%
\special{pa 5250 3307}%
\special{fp}%
}}%
%
{\color[named]{Black}{%
\special{pn 8}%
\special{pa 5150 3507}%
\special{pa 5250 3607}%
\special{fp}%
\special{sh 1}%
\special{pa 5250 3607}%
\special{pa 5217 3546}%
\special{pa 5212 3569}%
\special{pa 5189 3574}%
\special{pa 5250 3607}%
\special{fp}%
}}%
%
{\color[named]{Black}{%
\special{pn 8}%
\special{pa 4950 3307}%
\special{pa 5050 3407}%
\special{fp}%
}}%
%
{\color[named]{Black}{%
\special{pn 8}%
\special{pa 5350 3307}%
\special{pa 5450 3207}%
\special{fp}%
}}%
%
{\color[named]{Black}{%
\special{pn 8}%
\special{pa 5650 3307}%
\special{pa 5350 3007}%
\special{fp}%
\special{sh 1}%
\special{pa 5350 3007}%
\special{pa 5383 3068}%
\special{pa 5388 3045}%
\special{pa 5411 3040}%
\special{pa 5350 3007}%
\special{fp}%
}}%
%
{\color[named]{Black}{%
\special{pn 8}%
\special{pa 5350 3007}%
\special{pa 4950 3007}%
\special{fp}%
}}%
%
{\color[named]{Black}{%
\special{pn 8}%
\special{pa 5550 3107}%
\special{pa 5650 3007}%
\special{fp}%
\special{sh 1}%
\special{pa 5650 3007}%
\special{pa 5589 3040}%
\special{pa 5612 3045}%
\special{pa 5617 3068}%
\special{pa 5650 3007}%
\special{fp}%
}}%
%
{\color[named]{Black}{%
\special{pn 8}%
\special{pa 5750 3307}%
\special{pa 5650 3307}%
\special{fp}%
}}%
%
{\color[named]{Black}{%
\special{pn 8}%
\special{pa 5750 3607}%
\special{pa 6050 3307}%
\special{fp}%
\special{sh 1}%
\special{pa 6050 3307}%
\special{pa 5989 3340}%
\special{pa 6012 3345}%
\special{pa 6017 3368}%
\special{pa 6050 3307}%
\special{fp}%
}}%
%
{\color[named]{Black}{%
\special{pn 8}%
\special{pa 6050 3607}%
\special{pa 5950 3507}%
\special{fp}%
}}%
%
{\color[named]{Black}{%
\special{pn 8}%
\special{pa 5750 3607}%
\special{pa 5250 3607}%
\special{fp}%
}}%
%
{\color[named]{Black}{%
\special{pn 8}%
\special{pa 5850 3407}%
\special{pa 5750 3307}%
\special{fp}%
\special{sh 1}%
\special{pa 5750 3307}%
\special{pa 5783 3368}%
\special{pa 5788 3345}%
\special{pa 5811 3340}%
\special{pa 5750 3307}%
\special{fp}%
}}%
%
{\color[named]{Black}{%
\special{pn 8}%
\special{pa 5650 3007}%
\special{pa 6350 3007}%
\special{fp}%
\special{sh 1}%
\special{pa 6350 3007}%
\special{pa 6283 2987}%
\special{pa 6297 3007}%
\special{pa 6283 3027}%
\special{pa 6350 3007}%
\special{fp}%
}}%
%
{\color[named]{Black}{%
\special{pn 8}%
\special{pa 6350 3907}%
\special{pa 4950 3907}%
\special{fp}%
\special{sh 1}%
\special{pa 4950 3907}%
\special{pa 5017 3927}%
\special{pa 5003 3907}%
\special{pa 5017 3887}%
\special{pa 4950 3907}%
\special{fp}%
}}%
%
{\color[named]{Black}{%
\special{pn 8}%
\special{ar 4950 3757 150 150 1.5707963 4.7123890}%
}}%
%
{\color[named]{Black}{%
\special{pn 8}%
\special{pa 4935 3608}%
\special{pa 4950 3607}%
\special{fp}%
\special{sh 1}%
\special{pa 4950 3607}%
\special{pa 4882 3591}%
\special{pa 4897 3611}%
\special{pa 4885 3631}%
\special{pa 4950 3607}%
\special{fp}%
}}%
%
{\color[named]{Black}{%
\special{pn 8}%
\special{ar 4950 3157 150 150 1.5707963 4.7123890}%
}}%
%
{\color[named]{Black}{%
\special{pn 8}%
\special{pa 4935 3306}%
\special{pa 4950 3307}%
\special{fp}%
\special{sh 1}%
\special{pa 4950 3307}%
\special{pa 4885 3283}%
\special{pa 4897 3303}%
\special{pa 4882 3323}%
\special{pa 4950 3307}%
\special{fp}%
}}%
%
{\color[named]{Black}{%
\special{pn 8}%
\special{ar 6350 3157 150 150 4.7123890 1.5707963}%
}}%
%
{\color[named]{Black}{%
\special{pn 8}%
\special{pa 6365 3306}%
\special{pa 6350 3307}%
\special{fp}%
\special{sh 1}%
\special{pa 6350 3307}%
\special{pa 6418 3323}%
\special{pa 6403 3303}%
\special{pa 6415 3283}%
\special{pa 6350 3307}%
\special{fp}%
}}%
%
{\color[named]{Black}{%
\special{pn 8}%
\special{ar 6350 3757 150 150 4.7123890 1.5707963}%
}}%
%
{\color[named]{Black}{%
\special{pn 8}%
\special{pa 6365 3906}%
\special{pa 6350 3907}%
\special{fp}%
\special{sh 1}%
\special{pa 6350 3907}%
\special{pa 6418 3923}%
\special{pa 6403 3903}%
\special{pa 6415 3883}%
\special{pa 6350 3907}%
\special{fp}%
}}%
%
{\color[named]{Black}{%
\special{pn 8}%
\special{ar 1000 800 200 200 1.5707963 4.7123890}%
}}%
%
{\color[named]{Black}{%
\special{pn 8}%
\special{ar 1400 800 200 200 4.7123890 1.5707963}%
}}%
%
{\color[named]{Black}{%
\special{pn 8}%
\special{ar 2600 800 200 200 1.5707963 4.7123890}%
}}%
%
{\color[named]{Black}{%
\special{pn 8}%
\special{ar 3000 800 200 200 4.7123890 1.5707963}%
}}%
%
{\color[named]{Black}{%
\special{pn 8}%
\special{ar 4200 800 200 200 1.5707963 4.7123890}%
}}%
%
{\color[named]{Black}{%
\special{pn 8}%
\special{ar 4600 800 200 200 4.7123890 1.5707963}%
}}%
%
{\color[named]{Black}{%
\special{pn 8}%
\special{pa 1000 1000}%
\special{pa 1400 600}%
\special{fp}%
\special{sh 1}%
\special{pa 1400 600}%
\special{pa 1339 633}%
\special{pa 1362 638}%
\special{pa 1367 661}%
\special{pa 1400 600}%
\special{fp}%
\special{pa 1150 750}%
\special{pa 1000 600}%
\special{fp}%
\special{sh 1}%
\special{pa 1000 600}%
\special{pa 1033 661}%
\special{pa 1038 638}%
\special{pa 1061 633}%
\special{pa 1000 600}%
\special{fp}%
\special{pa 3000 1000}%
\special{pa 2600 600}%
\special{fp}%
\special{sh 1}%
\special{pa 2600 600}%
\special{pa 2633 661}%
\special{pa 2638 638}%
\special{pa 2661 633}%
\special{pa 2600 600}%
\special{fp}%
\special{pa 2850 750}%
\special{pa 3000 600}%
\special{fp}%
\special{sh 1}%
\special{pa 3000 600}%
\special{pa 2939 633}%
\special{pa 2962 638}%
\special{pa 2967 661}%
\special{pa 3000 600}%
\special{fp}%
}}%
%
{\color[named]{Black}{%
\special{pn 8}%
\special{pa 1400 1000}%
\special{pa 1250 850}%
\special{fp}%
\special{pa 2600 1000}%
\special{pa 2750 850}%
\special{fp}%
}}%
%
{\color[named]{Black}{%
\special{pn 8}%
\special{ar 4000 800 283 283 5.4977871 0.7853982}%
}}%
%
{\color[named]{Black}{%
\special{pn 8}%
\special{pa 4210 611}%
\special{pa 4200 600}%
\special{fp}%
\special{sh 1}%
\special{pa 4200 600}%
\special{pa 4230 663}%
\special{pa 4236 639}%
\special{pa 4260 636}%
\special{pa 4200 600}%
\special{fp}%
}}%
%
{\color[named]{Black}{%
\special{pn 8}%
\special{ar 4750 800 250 250 2.2142974 4.0688879}%
}}%
%
{\color[named]{Black}{%
\special{pn 8}%
\special{pa 4588 609}%
\special{pa 4600 600}%
\special{fp}%
\special{sh 1}%
\special{pa 4600 600}%
\special{pa 4535 624}%
\special{pa 4557 632}%
\special{pa 4559 656}%
\special{pa 4600 600}%
\special{fp}%
}}%
%
{\color[named]{Black}{%
\special{pn 8}%
\special{pa 1000 2200}%
\special{pa 1400 1800}%
\special{fp}%
\special{sh 1}%
\special{pa 1400 1800}%
\special{pa 1339 1833}%
\special{pa 1362 1838}%
\special{pa 1367 1861}%
\special{pa 1400 1800}%
\special{fp}%
\special{pa 1400 2200}%
\special{pa 1800 1800}%
\special{fp}%
\special{sh 1}%
\special{pa 1800 1800}%
\special{pa 1739 1833}%
\special{pa 1762 1838}%
\special{pa 1767 1861}%
\special{pa 1800 1800}%
\special{fp}%
\special{pa 1650 2050}%
\special{pa 1800 2200}%
\special{fp}%
\special{sh 1}%
\special{pa 1800 2200}%
\special{pa 1767 2139}%
\special{pa 1762 2162}%
\special{pa 1739 2167}%
\special{pa 1800 2200}%
\special{fp}%
\special{pa 1250 2050}%
\special{pa 1400 2200}%
\special{fp}%
\special{sh 1}%
\special{pa 1400 2200}%
\special{pa 1367 2139}%
\special{pa 1362 2162}%
\special{pa 1339 2167}%
\special{pa 1400 2200}%
\special{fp}%
\special{pa 1800 1600}%
\special{pa 1000 1600}%
\special{fp}%
\special{sh 1}%
\special{pa 1000 1600}%
\special{pa 1067 1620}%
\special{pa 1053 1600}%
\special{pa 1067 1580}%
\special{pa 1000 1600}%
\special{fp}%
\special{pa 1800 2400}%
\special{pa 1000 2400}%
\special{fp}%
\special{sh 1}%
\special{pa 1000 2400}%
\special{pa 1067 2420}%
\special{pa 1053 2400}%
\special{pa 1067 2380}%
\special{pa 1000 2400}%
\special{fp}%
}}%
%
{\color[named]{Black}{%
\special{pn 8}%
\special{pa 1000 1800}%
\special{pa 1150 1950}%
\special{fp}%
\special{pa 1550 1950}%
\special{pa 1400 1800}%
\special{fp}%
}}%
%
{\color[named]{Black}{%
\special{pn 8}%
\special{ar 1000 1700 100 100 1.5707963 4.7123890}%
}}%
%
{\color[named]{Black}{%
\special{pn 8}%
\special{ar 1000 2300 100 100 1.5707963 4.7123890}%
}}%
%
{\color[named]{Black}{%
\special{pn 8}%
\special{ar 1800 2300 100 100 4.7123890 1.5707963}%
}}%
%
{\color[named]{Black}{%
\special{pn 8}%
\special{ar 1800 1700 100 100 4.7123890 1.5707963}%
}}%
%
{\color[named]{Black}{%
\special{pn 8}%
\special{pa 2600 2200}%
\special{pa 3000 1800}%
\special{fp}%
\special{sh 1}%
\special{pa 3000 1800}%
\special{pa 2939 1833}%
\special{pa 2962 1838}%
\special{pa 2967 1861}%
\special{pa 3000 1800}%
\special{fp}%
\special{pa 3000 1800}%
\special{pa 3400 2200}%
\special{fp}%
\special{sh 1}%
\special{pa 3400 2200}%
\special{pa 3367 2139}%
\special{pa 3362 2162}%
\special{pa 3339 2167}%
\special{pa 3400 2200}%
\special{fp}%
\special{pa 3250 1950}%
\special{pa 3400 1800}%
\special{fp}%
\special{sh 1}%
\special{pa 3400 1800}%
\special{pa 3339 1833}%
\special{pa 3362 1838}%
\special{pa 3367 1861}%
\special{pa 3400 1800}%
\special{fp}%
\special{pa 2850 2050}%
\special{pa 3000 2200}%
\special{fp}%
\special{sh 1}%
\special{pa 3000 2200}%
\special{pa 2967 2139}%
\special{pa 2962 2162}%
\special{pa 2939 2167}%
\special{pa 3000 2200}%
\special{fp}%
\special{pa 3400 1600}%
\special{pa 2600 1600}%
\special{fp}%
\special{sh 1}%
\special{pa 2600 1600}%
\special{pa 2667 1620}%
\special{pa 2653 1600}%
\special{pa 2667 1580}%
\special{pa 2600 1600}%
\special{fp}%
\special{pa 3400 2400}%
\special{pa 2600 2400}%
\special{fp}%
\special{sh 1}%
\special{pa 2600 2400}%
\special{pa 2667 2420}%
\special{pa 2653 2400}%
\special{pa 2667 2380}%
\special{pa 2600 2400}%
\special{fp}%
}}%
%
{\color[named]{Black}{%
\special{pn 8}%
\special{ar 2600 2300 100 100 1.5707963 4.7123890}%
}}%
%
{\color[named]{Black}{%
\special{pn 8}%
\special{ar 2600 1700 100 100 1.5707963 4.7123890}%
}}%
%
{\color[named]{Black}{%
\special{pn 8}%
\special{ar 3400 2300 100 100 4.7123890 1.5707963}%
}}%
%
{\color[named]{Black}{%
\special{pn 8}%
\special{ar 3400 1700 100 100 4.7123890 1.5707963}%
}}%
%
{\color[named]{Black}{%
\special{pn 8}%
\special{pa 2600 1800}%
\special{pa 2750 1950}%
\special{fp}%
\special{pa 3000 2200}%
\special{pa 3150 2050}%
\special{fp}%
}}%
%
{\color[named]{Black}{%
\special{pn 8}%
\special{pa 4200 2200}%
\special{pa 4600 1800}%
\special{fp}%
\special{sh 1}%
\special{pa 4600 1800}%
\special{pa 4539 1833}%
\special{pa 4562 1838}%
\special{pa 4567 1861}%
\special{pa 4600 1800}%
\special{fp}%
}}%
%
{\color[named]{Black}{%
\special{pn 8}%
\special{pa 4450 2050}%
\special{pa 4600 2200}%
\special{fp}%
\special{sh 1}%
\special{pa 4600 2200}%
\special{pa 4567 2139}%
\special{pa 4562 2162}%
\special{pa 4539 2167}%
\special{pa 4600 2200}%
\special{fp}%
}}%
%
{\color[named]{Black}{%
\special{pn 8}%
\special{pa 5000 2400}%
\special{pa 4200 2400}%
\special{fp}%
\special{sh 1}%
\special{pa 4200 2400}%
\special{pa 4267 2420}%
\special{pa 4253 2400}%
\special{pa 4267 2380}%
\special{pa 4200 2400}%
\special{fp}%
}}%
%
{\color[named]{Black}{%
\special{pn 8}%
\special{pa 5000 1600}%
\special{pa 4200 1600}%
\special{fp}%
\special{sh 1}%
\special{pa 4200 1600}%
\special{pa 4267 1620}%
\special{pa 4253 1600}%
\special{pa 4267 1580}%
\special{pa 4200 1600}%
\special{fp}%
}}%
%
{\color[named]{Black}{%
\special{pn 8}%
\special{ar 4800 1650 250 250 0.6435011 2.4980915}%
}}%
%
{\color[named]{Black}{%
\special{pn 8}%
\special{pa 4991 1812}%
\special{pa 5000 1800}%
\special{fp}%
\special{sh 1}%
\special{pa 5000 1800}%
\special{pa 4944 1841}%
\special{pa 4968 1843}%
\special{pa 4976 1865}%
\special{pa 5000 1800}%
\special{fp}%
}}%
%
{\color[named]{Black}{%
\special{pn 8}%
\special{ar 4800 2350 250 250 3.7850938 5.6396842}%
}}%
%
{\color[named]{Black}{%
\special{pn 8}%
\special{pa 4991 2188}%
\special{pa 5000 2200}%
\special{fp}%
\special{sh 1}%
\special{pa 5000 2200}%
\special{pa 4976 2135}%
\special{pa 4968 2157}%
\special{pa 4944 2159}%
\special{pa 5000 2200}%
\special{fp}%
}}%
%
{\color[named]{Black}{%
\special{pn 8}%
\special{ar 5000 2300 100 100 4.7123890 1.5707963}%
}}%
%
{\color[named]{Black}{%
\special{pn 8}%
\special{ar 5000 1700 100 100 4.7123890 1.5707963}%
}}%
%
{\color[named]{Black}{%
\special{pn 8}%
\special{ar 4200 1700 100 100 1.5707963 4.7123890}%
}}%
%
{\color[named]{Black}{%
\special{pn 8}%
\special{ar 4200 2300 100 100 1.5707963 4.7123890}%
}}%
%
{\color[named]{Black}{%
\special{pn 8}%
\special{pa 4200 1800}%
\special{pa 4350 1950}%
\special{fp}%
}}%
\put(16.0000,-14.0000){\makebox(0,0){{\color[named]{Black}{$L_-$}}}}%
\put(32.0000,-14.0000){\makebox(0,0){{\color[named]{Black}{$L_+$}}}}%
\put(48.0000,-14.0000){\makebox(0,0){{\color[named]{Black}{$L_0$}}}}%
%
{\color[named]{Black}{%
\special{pn 8}%
\special{pa 4350 3300}%
\special{pa 4250 3400}%
\special{fp}%
}}%
%
{\color[named]{Black}{%
\special{pn 8}%
\special{pa 4050 3300}%
\special{pa 4350 3600}%
\special{fp}%
\special{sh 1}%
\special{pa 4350 3600}%
\special{pa 4317 3539}%
\special{pa 4312 3562}%
\special{pa 4289 3567}%
\special{pa 4350 3600}%
\special{fp}%
\special{pa 4150 3500}%
\special{pa 4050 3600}%
\special{fp}%
\special{sh 1}%
\special{pa 4050 3600}%
\special{pa 4111 3567}%
\special{pa 4088 3562}%
\special{pa 4083 3539}%
\special{pa 4050 3600}%
\special{fp}%
}}%
%
{\color[named]{Black}{%
\special{pn 8}%
\special{ar 5850 3450 250 250 5.6396842 0.6435011}%
}}%
%
{\color[named]{Black}{%
\special{pn 8}%
\special{pa 6059 3588}%
\special{pa 6050 3600}%
\special{fp}%
\special{sh 1}%
\special{pa 6050 3600}%
\special{pa 6106 3559}%
\special{pa 6082 3557}%
\special{pa 6074 3535}%
\special{pa 6050 3600}%
\special{fp}%
}}%
%
{\color[named]{Black}{%
\special{pn 8}%
\special{ar 6550 3450 250 250 2.4980915 3.7850938}%
}}%
%
{\color[named]{Black}{%
\special{pn 8}%
\special{pa 6341 3588}%
\special{pa 6350 3600}%
\special{fp}%
\special{sh 1}%
\special{pa 6350 3600}%
\special{pa 6326 3535}%
\special{pa 6318 3557}%
\special{pa 6294 3559}%
\special{pa 6350 3600}%
\special{fp}%
}}%
\put(22.0000,-28.0000){\makebox(0,0){{\color[named]{Black}{$L_+$}}}}%
\put(42.0000,-28.0000){\makebox(0,0){{\color[named]{Black}{$L_-$}}}}%
\put(62.0000,-28.0000){\makebox(0,0){{\color[named]{Black}{$L_0$}}}}%
\end{picture}}%

%% file: proof_of_alexander_recursion.tex
{\unitlength 0.1in%
\begin{picture}(59.5500,80.3200)(4.5000,-82.0500)%
\put(20.0000,-8.0000){\makebox(0,0){{\color[named]{Black}{$L_+$}}}}%
%
{\color[named]{Black}{%
\special{pn 8}%
\special{pa 2400 1000}%
\special{pa 2800 600}%
\special{fp}%
\special{sh 1}%
\special{pa 2800 600}%
\special{pa 2739 633}%
\special{pa 2762 638}%
\special{pa 2767 661}%
\special{pa 2800 600}%
\special{fp}%
}}%
%
{\color[named]{Black}{%
\special{pn 8}%
\special{pa 3350 750}%
\special{pa 3200 600}%
\special{fp}%
\special{sh 1}%
\special{pa 3200 600}%
\special{pa 3233 661}%
\special{pa 3238 638}%
\special{pa 3261 633}%
\special{pa 3200 600}%
\special{fp}%
}}%
%
{\color[named]{Black}{%
\special{pn 8}%
\special{pa 3850 850}%
\special{pa 4000 1000}%
\special{fp}%
\special{sh 1}%
\special{pa 4000 1000}%
\special{pa 3967 939}%
\special{pa 3962 962}%
\special{pa 3939 967}%
\special{pa 4000 1000}%
\special{fp}%
}}%
%
{\color[named]{Black}{%
\special{pn 8}%
\special{pa 4150 750}%
\special{pa 4000 600}%
\special{fp}%
\special{sh 1}%
\special{pa 4000 600}%
\special{pa 4033 661}%
\special{pa 4038 638}%
\special{pa 4061 633}%
\special{pa 4000 600}%
\special{fp}%
}}%
%
{\color[named]{Black}{%
\special{pn 8}%
\special{pa 4000 1000}%
\special{pa 4400 600}%
\special{fp}%
\special{sh 1}%
\special{pa 4400 600}%
\special{pa 4339 633}%
\special{pa 4362 638}%
\special{pa 4367 661}%
\special{pa 4400 600}%
\special{fp}%
}}%
%
{\color[named]{Black}{%
\special{pn 8}%
\special{pa 4000 600}%
\special{pa 3600 1000}%
\special{fp}%
\special{sh 1}%
\special{pa 3600 1000}%
\special{pa 3661 967}%
\special{pa 3638 962}%
\special{pa 3633 939}%
\special{pa 3600 1000}%
\special{fp}%
}}%
%
{\color[named]{Black}{%
\special{pn 8}%
\special{pa 3200 1000}%
\special{pa 3600 600}%
\special{fp}%
\special{sh 1}%
\special{pa 3600 600}%
\special{pa 3539 633}%
\special{pa 3562 638}%
\special{pa 3567 661}%
\special{pa 3600 600}%
\special{fp}%
}}%
%
{\color[named]{Black}{%
\special{pn 8}%
\special{pa 2550 750}%
\special{pa 2400 600}%
\special{fp}%
\special{sh 1}%
\special{pa 2400 600}%
\special{pa 2433 661}%
\special{pa 2438 638}%
\special{pa 2461 633}%
\special{pa 2400 600}%
\special{fp}%
}}%
%
{\color[named]{Black}{%
\special{pn 8}%
\special{pa 2800 1000}%
\special{pa 2650 850}%
\special{fp}%
}}%
%
{\color[named]{Black}{%
\special{pn 8}%
\special{pa 3600 1000}%
\special{pa 3450 850}%
\special{fp}%
}}%
%
{\color[named]{Black}{%
\special{pn 8}%
\special{pa 4400 1000}%
\special{pa 4250 850}%
\special{fp}%
}}%
%
{\color[named]{Black}{%
\special{pn 8}%
\special{pa 3600 600}%
\special{pa 3750 750}%
\special{fp}%
}}%
%
{\color[named]{Black}{%
\special{pn 8}%
\special{ar 4400 400 200 200 4.7123890 1.5707963}%
}}%
%
{\color[named]{Black}{%
\special{pn 8}%
\special{ar 4400 1200 200 200 4.7123890 1.5707963}%
}}%
%
{\color[named]{Black}{%
\special{pn 8}%
\special{pa 4400 200}%
\special{pa 2200 200}%
\special{fp}%
\special{sh 1}%
\special{pa 2200 200}%
\special{pa 2267 220}%
\special{pa 2253 200}%
\special{pa 2267 180}%
\special{pa 2200 200}%
\special{fp}%
}}%
%
{\color[named]{Black}{%
\special{pn 8}%
\special{pa 2200 1400}%
\special{pa 4400 1400}%
\special{fp}%
\special{sh 1}%
\special{pa 4400 1400}%
\special{pa 4333 1380}%
\special{pa 4347 1400}%
\special{pa 4333 1420}%
\special{pa 4400 1400}%
\special{fp}%
}}%
%
{\color[named]{Black}{%
\special{pn 8}%
\special{pa 3047 847}%
\special{pa 3197 997}%
\special{fp}%
\special{sh 1}%
\special{pa 3197 997}%
\special{pa 3164 936}%
\special{pa 3159 959}%
\special{pa 3136 964}%
\special{pa 3197 997}%
\special{fp}%
}}%
%
{\color[named]{Black}{%
\special{pn 8}%
\special{pa 3197 597}%
\special{pa 2797 997}%
\special{fp}%
\special{sh 1}%
\special{pa 2797 997}%
\special{pa 2858 964}%
\special{pa 2835 959}%
\special{pa 2830 936}%
\special{pa 2797 997}%
\special{fp}%
}}%
%
{\color[named]{Black}{%
\special{pn 8}%
\special{pa 2797 597}%
\special{pa 2947 747}%
\special{fp}%
}}%
\put(5.9500,-23.9700){\makebox(0,0){{\color[named]{Black}{$L_-$}}}}%
%
{\color[named]{Black}{%
\special{pn 8}%
\special{pa 1005 2607}%
\special{pa 1405 2207}%
\special{fp}%
\special{sh 1}%
\special{pa 1405 2207}%
\special{pa 1344 2240}%
\special{pa 1367 2245}%
\special{pa 1372 2268}%
\special{pa 1405 2207}%
\special{fp}%
}}%
%
{\color[named]{Black}{%
\special{pn 8}%
\special{pa 1955 2357}%
\special{pa 1805 2207}%
\special{fp}%
\special{sh 1}%
\special{pa 1805 2207}%
\special{pa 1838 2268}%
\special{pa 1843 2245}%
\special{pa 1866 2240}%
\special{pa 1805 2207}%
\special{fp}%
}}%
%
{\color[named]{Black}{%
\special{pn 8}%
\special{pa 1805 2607}%
\special{pa 2205 2207}%
\special{fp}%
\special{sh 1}%
\special{pa 2205 2207}%
\special{pa 2144 2240}%
\special{pa 2167 2245}%
\special{pa 2172 2268}%
\special{pa 2205 2207}%
\special{fp}%
}}%
%
{\color[named]{Black}{%
\special{pn 8}%
\special{pa 1155 2357}%
\special{pa 1005 2207}%
\special{fp}%
\special{sh 1}%
\special{pa 1005 2207}%
\special{pa 1038 2268}%
\special{pa 1043 2245}%
\special{pa 1066 2240}%
\special{pa 1005 2207}%
\special{fp}%
}}%
%
{\color[named]{Black}{%
\special{pn 8}%
\special{pa 1405 2607}%
\special{pa 1255 2457}%
\special{fp}%
}}%
%
{\color[named]{Black}{%
\special{pn 8}%
\special{pa 2205 2607}%
\special{pa 2055 2457}%
\special{fp}%
}}%
%
{\color[named]{Black}{%
\special{pn 8}%
\special{ar 2995 1997 200 200 4.7123890 1.5707963}%
}}%
%
{\color[named]{Black}{%
\special{pn 8}%
\special{ar 2995 2800 200 200 4.7123890 1.5707963}%
}}%
%
{\color[named]{Black}{%
\special{pn 8}%
\special{pa 2995 1797}%
\special{pa 795 1797}%
\special{fp}%
\special{sh 1}%
\special{pa 795 1797}%
\special{pa 862 1817}%
\special{pa 848 1797}%
\special{pa 862 1777}%
\special{pa 795 1797}%
\special{fp}%
}}%
%
{\color[named]{Black}{%
\special{pn 8}%
\special{pa 790 3000}%
\special{pa 2990 3000}%
\special{fp}%
\special{sh 1}%
\special{pa 2990 3000}%
\special{pa 2923 2980}%
\special{pa 2937 3000}%
\special{pa 2923 3020}%
\special{pa 2990 3000}%
\special{fp}%
}}%
%
{\color[named]{Black}{%
\special{pn 8}%
\special{pa 1652 2454}%
\special{pa 1802 2604}%
\special{fp}%
\special{sh 1}%
\special{pa 1802 2604}%
\special{pa 1769 2543}%
\special{pa 1764 2566}%
\special{pa 1741 2571}%
\special{pa 1802 2604}%
\special{fp}%
}}%
%
{\color[named]{Black}{%
\special{pn 8}%
\special{pa 1802 2204}%
\special{pa 1402 2604}%
\special{fp}%
\special{sh 1}%
\special{pa 1402 2604}%
\special{pa 1463 2571}%
\special{pa 1440 2566}%
\special{pa 1435 2543}%
\special{pa 1402 2604}%
\special{fp}%
}}%
%
{\color[named]{Black}{%
\special{pn 8}%
\special{pa 1402 2204}%
\special{pa 1552 2354}%
\special{fp}%
}}%
%
{\color[named]{Black}{%
\special{pn 8}%
\special{pa 2200 2200}%
\special{pa 3000 2200}%
\special{fp}%
\special{sh 1}%
\special{pa 3000 2200}%
\special{pa 2933 2180}%
\special{pa 2947 2200}%
\special{pa 2933 2220}%
\special{pa 3000 2200}%
\special{fp}%
\special{pa 3000 2600}%
\special{pa 2200 2600}%
\special{fp}%
\special{sh 1}%
\special{pa 2200 2600}%
\special{pa 2267 2620}%
\special{pa 2253 2600}%
\special{pa 2267 2580}%
\special{pa 2200 2600}%
\special{fp}%
}}%
\put(38.0000,-24.0000){\makebox(0,0){{\color[named]{Black}{$L_0$}}}}%
%
{\color[named]{Black}{%
\special{pn 8}%
\special{pa 4205 2607}%
\special{pa 4605 2207}%
\special{fp}%
\special{sh 1}%
\special{pa 4605 2207}%
\special{pa 4544 2240}%
\special{pa 4567 2245}%
\special{pa 4572 2268}%
\special{pa 4605 2207}%
\special{fp}%
}}%
%
{\color[named]{Black}{%
\special{pn 8}%
\special{pa 5155 2357}%
\special{pa 5005 2207}%
\special{fp}%
\special{sh 1}%
\special{pa 5005 2207}%
\special{pa 5038 2268}%
\special{pa 5043 2245}%
\special{pa 5066 2240}%
\special{pa 5005 2207}%
\special{fp}%
}}%
%
{\color[named]{Black}{%
\special{pn 8}%
\special{pa 5655 2457}%
\special{pa 5805 2607}%
\special{fp}%
\special{sh 1}%
\special{pa 5805 2607}%
\special{pa 5772 2546}%
\special{pa 5767 2569}%
\special{pa 5744 2574}%
\special{pa 5805 2607}%
\special{fp}%
}}%
%
{\color[named]{Black}{%
\special{pn 8}%
\special{pa 5805 2207}%
\special{pa 5405 2607}%
\special{fp}%
\special{sh 1}%
\special{pa 5405 2607}%
\special{pa 5466 2574}%
\special{pa 5443 2569}%
\special{pa 5438 2546}%
\special{pa 5405 2607}%
\special{fp}%
}}%
%
{\color[named]{Black}{%
\special{pn 8}%
\special{pa 5005 2607}%
\special{pa 5405 2207}%
\special{fp}%
\special{sh 1}%
\special{pa 5405 2207}%
\special{pa 5344 2240}%
\special{pa 5367 2245}%
\special{pa 5372 2268}%
\special{pa 5405 2207}%
\special{fp}%
}}%
%
{\color[named]{Black}{%
\special{pn 8}%
\special{pa 4355 2357}%
\special{pa 4205 2207}%
\special{fp}%
\special{sh 1}%
\special{pa 4205 2207}%
\special{pa 4238 2268}%
\special{pa 4243 2245}%
\special{pa 4266 2240}%
\special{pa 4205 2207}%
\special{fp}%
}}%
%
{\color[named]{Black}{%
\special{pn 8}%
\special{pa 4605 2607}%
\special{pa 4455 2457}%
\special{fp}%
}}%
%
{\color[named]{Black}{%
\special{pn 8}%
\special{pa 5405 2607}%
\special{pa 5255 2457}%
\special{fp}%
}}%
%
{\color[named]{Black}{%
\special{pn 8}%
\special{pa 5405 2207}%
\special{pa 5555 2357}%
\special{fp}%
}}%
%
{\color[named]{Black}{%
\special{pn 8}%
\special{ar 6205 2007 200 200 4.7123890 1.5707963}%
}}%
%
{\color[named]{Black}{%
\special{pn 8}%
\special{ar 6205 2807 200 200 4.7123890 1.5707963}%
}}%
%
{\color[named]{Black}{%
\special{pn 8}%
\special{pa 6205 1807}%
\special{pa 4005 1807}%
\special{fp}%
\special{sh 1}%
\special{pa 4005 1807}%
\special{pa 4072 1827}%
\special{pa 4058 1807}%
\special{pa 4072 1787}%
\special{pa 4005 1807}%
\special{fp}%
}}%
%
{\color[named]{Black}{%
\special{pn 8}%
\special{pa 4000 3000}%
\special{pa 6200 3000}%
\special{fp}%
\special{sh 1}%
\special{pa 6200 3000}%
\special{pa 6133 2980}%
\special{pa 6147 3000}%
\special{pa 6133 3020}%
\special{pa 6200 3000}%
\special{fp}%
}}%
%
{\color[named]{Black}{%
\special{pn 8}%
\special{pa 4852 2454}%
\special{pa 5002 2604}%
\special{fp}%
\special{sh 1}%
\special{pa 5002 2604}%
\special{pa 4969 2543}%
\special{pa 4964 2566}%
\special{pa 4941 2571}%
\special{pa 5002 2604}%
\special{fp}%
}}%
%
{\color[named]{Black}{%
\special{pn 8}%
\special{pa 5002 2204}%
\special{pa 4602 2604}%
\special{fp}%
\special{sh 1}%
\special{pa 4602 2604}%
\special{pa 4663 2571}%
\special{pa 4640 2566}%
\special{pa 4635 2543}%
\special{pa 4602 2604}%
\special{fp}%
}}%
%
{\color[named]{Black}{%
\special{pn 8}%
\special{pa 4602 2204}%
\special{pa 4752 2354}%
\special{fp}%
}}%
%
{\color[named]{Black}{%
\special{pn 8}%
\special{ar 5650 2400 250 250 5.3558901 0.9272952}%
}}%
%
{\color[named]{Black}{%
\special{pn 8}%
\special{pa 5812 2209}%
\special{pa 5800 2200}%
\special{fp}%
\special{sh 1}%
\special{pa 5800 2200}%
\special{pa 5841 2256}%
\special{pa 5843 2232}%
\special{pa 5865 2224}%
\special{pa 5800 2200}%
\special{fp}%
}}%
%
{\color[named]{Black}{%
\special{pn 8}%
\special{ar 6350 2400 250 250 2.2142974 4.0688879}%
}}%
%
{\color[named]{Black}{%
\special{pn 8}%
\special{pa 6188 2209}%
\special{pa 6200 2200}%
\special{fp}%
\special{sh 1}%
\special{pa 6200 2200}%
\special{pa 6135 2224}%
\special{pa 6157 2232}%
\special{pa 6159 2256}%
\special{pa 6200 2200}%
\special{fp}%
}}%
\put(50.0000,-21.0000){\makebox(0,0){{\color[named]{Black}{trivial}}}}%
%
{\color[named]{Black}{%
\special{pn 8}%
\special{pa 2400 600}%
\special{pa 2200 600}%
\special{fp}%
\special{sh 1}%
\special{pa 2200 600}%
\special{pa 2267 620}%
\special{pa 2253 600}%
\special{pa 2267 580}%
\special{pa 2200 600}%
\special{fp}%
\special{pa 2200 1000}%
\special{pa 2400 1000}%
\special{fp}%
\special{sh 1}%
\special{pa 2400 1000}%
\special{pa 2333 980}%
\special{pa 2347 1000}%
\special{pa 2333 1020}%
\special{pa 2400 1000}%
\special{fp}%
}}%
%
{\color[named]{Black}{%
\special{pn 8}%
\special{pa 1010 2210}%
\special{pa 810 2210}%
\special{fp}%
\special{sh 1}%
\special{pa 810 2210}%
\special{pa 877 2230}%
\special{pa 863 2210}%
\special{pa 877 2190}%
\special{pa 810 2210}%
\special{fp}%
\special{pa 810 2610}%
\special{pa 1010 2610}%
\special{fp}%
\special{sh 1}%
\special{pa 1010 2610}%
\special{pa 943 2590}%
\special{pa 957 2610}%
\special{pa 943 2630}%
\special{pa 1010 2610}%
\special{fp}%
\special{pa 4010 2610}%
\special{pa 4210 2610}%
\special{fp}%
\special{sh 1}%
\special{pa 4210 2610}%
\special{pa 4143 2590}%
\special{pa 4157 2610}%
\special{pa 4143 2630}%
\special{pa 4210 2610}%
\special{fp}%
\special{pa 4210 2210}%
\special{pa 4010 2210}%
\special{fp}%
\special{sh 1}%
\special{pa 4010 2210}%
\special{pa 4077 2230}%
\special{pa 4063 2210}%
\special{pa 4077 2190}%
\special{pa 4010 2210}%
\special{fp}%
}}%
%
{\color[named]{Black}{%
\special{pn 8}%
\special{pa 4395 3606}%
\special{pa 2195 3606}%
\special{fp}%
\special{sh 1}%
\special{pa 2195 3606}%
\special{pa 2262 3626}%
\special{pa 2248 3606}%
\special{pa 2262 3586}%
\special{pa 2195 3606}%
\special{fp}%
\special{pa 2195 4006}%
\special{pa 2395 4006}%
\special{fp}%
\special{sh 1}%
\special{pa 2395 4006}%
\special{pa 2328 3986}%
\special{pa 2342 4006}%
\special{pa 2328 4026}%
\special{pa 2395 4006}%
\special{fp}%
\special{pa 2195 4406}%
\special{pa 2395 4406}%
\special{fp}%
\special{sh 1}%
\special{pa 2395 4406}%
\special{pa 2328 4386}%
\special{pa 2342 4406}%
\special{pa 2328 4426}%
\special{pa 2395 4406}%
\special{fp}%
\special{pa 2395 4406}%
\special{pa 2795 4006}%
\special{fp}%
\special{sh 1}%
\special{pa 2795 4006}%
\special{pa 2734 4039}%
\special{pa 2757 4044}%
\special{pa 2762 4067}%
\special{pa 2795 4006}%
\special{fp}%
\special{pa 2795 4406}%
\special{pa 3195 4006}%
\special{fp}%
\special{sh 1}%
\special{pa 3195 4006}%
\special{pa 3134 4039}%
\special{pa 3157 4044}%
\special{pa 3162 4067}%
\special{pa 3195 4006}%
\special{fp}%
\special{pa 3195 4406}%
\special{pa 3595 4006}%
\special{fp}%
\special{sh 1}%
\special{pa 3595 4006}%
\special{pa 3534 4039}%
\special{pa 3557 4044}%
\special{pa 3562 4067}%
\special{pa 3595 4006}%
\special{fp}%
\special{pa 3595 4406}%
\special{pa 3995 4006}%
\special{fp}%
\special{sh 1}%
\special{pa 3995 4006}%
\special{pa 3934 4039}%
\special{pa 3957 4044}%
\special{pa 3962 4067}%
\special{pa 3995 4006}%
\special{fp}%
\special{pa 3995 4406}%
\special{pa 4395 4006}%
\special{fp}%
\special{sh 1}%
\special{pa 4395 4006}%
\special{pa 4334 4039}%
\special{pa 4357 4044}%
\special{pa 4362 4067}%
\special{pa 4395 4006}%
\special{fp}%
\special{pa 2645 4256}%
\special{pa 2795 4406}%
\special{fp}%
\special{sh 1}%
\special{pa 2795 4406}%
\special{pa 2762 4345}%
\special{pa 2757 4368}%
\special{pa 2734 4373}%
\special{pa 2795 4406}%
\special{fp}%
\special{pa 3045 4256}%
\special{pa 3195 4406}%
\special{fp}%
\special{sh 1}%
\special{pa 3195 4406}%
\special{pa 3162 4345}%
\special{pa 3157 4368}%
\special{pa 3134 4373}%
\special{pa 3195 4406}%
\special{fp}%
\special{pa 3445 4256}%
\special{pa 3595 4406}%
\special{fp}%
\special{sh 1}%
\special{pa 3595 4406}%
\special{pa 3562 4345}%
\special{pa 3557 4368}%
\special{pa 3534 4373}%
\special{pa 3595 4406}%
\special{fp}%
\special{pa 3845 4256}%
\special{pa 3995 4406}%
\special{fp}%
\special{sh 1}%
\special{pa 3995 4406}%
\special{pa 3962 4345}%
\special{pa 3957 4368}%
\special{pa 3934 4373}%
\special{pa 3995 4406}%
\special{fp}%
\special{pa 4245 4256}%
\special{pa 4395 4406}%
\special{fp}%
\special{sh 1}%
\special{pa 4395 4406}%
\special{pa 4362 4345}%
\special{pa 4357 4368}%
\special{pa 4334 4373}%
\special{pa 4395 4406}%
\special{fp}%
\special{pa 4395 4806}%
\special{pa 2195 4806}%
\special{fp}%
\special{sh 1}%
\special{pa 2195 4806}%
\special{pa 2262 4826}%
\special{pa 2248 4806}%
\special{pa 2262 4786}%
\special{pa 2195 4806}%
\special{fp}%
}}%
%
{\color[named]{Black}{%
\special{pn 8}%
\special{pa 2395 4006}%
\special{pa 2545 4156}%
\special{fp}%
\special{pa 2795 4006}%
\special{pa 2945 4156}%
\special{fp}%
\special{pa 3195 4006}%
\special{pa 3345 4156}%
\special{fp}%
\special{pa 3595 4006}%
\special{pa 3745 4156}%
\special{fp}%
\special{pa 3995 4006}%
\special{pa 4145 4156}%
\special{fp}%
}}%
%
{\color[named]{Black}{%
\special{pn 8}%
\special{ar 4395 3806 200 200 4.7123890 1.5707963}%
}}%
%
{\color[named]{Black}{%
\special{pn 8}%
\special{pa 4410 3607}%
\special{pa 4395 3606}%
\special{fp}%
\special{sh 1}%
\special{pa 4395 3606}%
\special{pa 4460 3630}%
\special{pa 4448 3610}%
\special{pa 4463 3590}%
\special{pa 4395 3606}%
\special{fp}%
}}%
%
{\color[named]{Black}{%
\special{pn 8}%
\special{ar 4395 4606 200 200 4.7123890 1.5707963}%
}}%
%
{\color[named]{Black}{%
\special{pn 8}%
\special{pa 4410 4805}%
\special{pa 4395 4806}%
\special{fp}%
\special{sh 1}%
\special{pa 4395 4806}%
\special{pa 4463 4822}%
\special{pa 4448 4802}%
\special{pa 4460 4782}%
\special{pa 4395 4806}%
\special{fp}%
}}%
\put(19.9500,-42.0600){\makebox(0,0){{\color[named]{Black}{$L_-$}}}}%
%
{\color[named]{Black}{%
\special{pn 8}%
\special{pa 3005 5206}%
\special{pa 805 5206}%
\special{fp}%
\special{sh 1}%
\special{pa 805 5206}%
\special{pa 872 5226}%
\special{pa 858 5206}%
\special{pa 872 5186}%
\special{pa 805 5206}%
\special{fp}%
}}%
%
{\color[named]{Black}{%
\special{pn 8}%
\special{pa 3005 6406}%
\special{pa 805 6406}%
\special{fp}%
\special{sh 1}%
\special{pa 805 6406}%
\special{pa 872 6426}%
\special{pa 858 6406}%
\special{pa 872 6386}%
\special{pa 805 6406}%
\special{fp}%
}}%
%
{\color[named]{Black}{%
\special{pn 8}%
\special{pa 2055 5856}%
\special{pa 2205 6006}%
\special{fp}%
\special{sh 1}%
\special{pa 2205 6006}%
\special{pa 2172 5945}%
\special{pa 2167 5968}%
\special{pa 2144 5973}%
\special{pa 2205 6006}%
\special{fp}%
}}%
%
{\color[named]{Black}{%
\special{pn 8}%
\special{pa 1655 5856}%
\special{pa 1805 6006}%
\special{fp}%
\special{sh 1}%
\special{pa 1805 6006}%
\special{pa 1772 5945}%
\special{pa 1767 5968}%
\special{pa 1744 5973}%
\special{pa 1805 6006}%
\special{fp}%
}}%
%
{\color[named]{Black}{%
\special{pn 8}%
\special{pa 1255 5856}%
\special{pa 1405 6006}%
\special{fp}%
\special{sh 1}%
\special{pa 1405 6006}%
\special{pa 1372 5945}%
\special{pa 1367 5968}%
\special{pa 1344 5973}%
\special{pa 1405 6006}%
\special{fp}%
}}%
%
{\color[named]{Black}{%
\special{pn 8}%
\special{pa 1805 6006}%
\special{pa 2205 5606}%
\special{fp}%
\special{sh 1}%
\special{pa 2205 5606}%
\special{pa 2144 5639}%
\special{pa 2167 5644}%
\special{pa 2172 5667}%
\special{pa 2205 5606}%
\special{fp}%
}}%
%
{\color[named]{Black}{%
\special{pn 8}%
\special{pa 1405 6006}%
\special{pa 1805 5606}%
\special{fp}%
\special{sh 1}%
\special{pa 1805 5606}%
\special{pa 1744 5639}%
\special{pa 1767 5644}%
\special{pa 1772 5667}%
\special{pa 1805 5606}%
\special{fp}%
}}%
%
{\color[named]{Black}{%
\special{pn 8}%
\special{pa 1005 6006}%
\special{pa 1405 5606}%
\special{fp}%
\special{sh 1}%
\special{pa 1405 5606}%
\special{pa 1344 5639}%
\special{pa 1367 5644}%
\special{pa 1372 5667}%
\special{pa 1405 5606}%
\special{fp}%
}}%
%
{\color[named]{Black}{%
\special{pn 8}%
\special{pa 805 6006}%
\special{pa 1005 6006}%
\special{fp}%
\special{sh 1}%
\special{pa 1005 6006}%
\special{pa 938 5986}%
\special{pa 952 6006}%
\special{pa 938 6026}%
\special{pa 1005 6006}%
\special{fp}%
}}%
%
{\color[named]{Black}{%
\special{pn 8}%
\special{pa 805 5606}%
\special{pa 1005 5606}%
\special{fp}%
\special{sh 1}%
\special{pa 1005 5606}%
\special{pa 938 5586}%
\special{pa 952 5606}%
\special{pa 938 5626}%
\special{pa 1005 5606}%
\special{fp}%
}}%
%
{\color[named]{Black}{%
\special{pn 8}%
\special{pa 1005 5606}%
\special{pa 1155 5756}%
\special{fp}%
}}%
%
{\color[named]{Black}{%
\special{pn 8}%
\special{pa 1805 5606}%
\special{pa 1955 5756}%
\special{fp}%
}}%
%
{\color[named]{Black}{%
\special{pn 8}%
\special{pa 1405 5606}%
\special{pa 1555 5756}%
\special{fp}%
}}%
%
{\color[named]{Black}{%
\special{pn 8}%
\special{ar 3005 5406 200 200 4.7123890 1.5707963}%
}}%
%
{\color[named]{Black}{%
\special{pn 8}%
\special{pa 3020 5207}%
\special{pa 3005 5206}%
\special{fp}%
\special{sh 1}%
\special{pa 3005 5206}%
\special{pa 3070 5230}%
\special{pa 3058 5210}%
\special{pa 3073 5190}%
\special{pa 3005 5206}%
\special{fp}%
}}%
%
{\color[named]{Black}{%
\special{pn 8}%
\special{ar 3005 6206 200 200 4.7123890 1.5707963}%
}}%
%
{\color[named]{Black}{%
\special{pn 8}%
\special{pa 3020 6405}%
\special{pa 3005 6406}%
\special{fp}%
\special{sh 1}%
\special{pa 3005 6406}%
\special{pa 3073 6422}%
\special{pa 3058 6402}%
\special{pa 3070 6382}%
\special{pa 3005 6406}%
\special{fp}%
}}%
\put(6.0500,-58.0600){\makebox(0,0){{\color[named]{Black}{$L_+$}}}}%
%
{\color[named]{Black}{%
\special{pn 8}%
\special{pa 6205 5206}%
\special{pa 4005 5206}%
\special{fp}%
\special{sh 1}%
\special{pa 4005 5206}%
\special{pa 4072 5226}%
\special{pa 4058 5206}%
\special{pa 4072 5186}%
\special{pa 4005 5206}%
\special{fp}%
}}%
%
{\color[named]{Black}{%
\special{pn 8}%
\special{pa 6205 6406}%
\special{pa 4005 6406}%
\special{fp}%
\special{sh 1}%
\special{pa 4005 6406}%
\special{pa 4072 6426}%
\special{pa 4058 6406}%
\special{pa 4072 6386}%
\special{pa 4005 6406}%
\special{fp}%
}}%
%
{\color[named]{Black}{%
\special{pn 8}%
\special{pa 5655 5856}%
\special{pa 5805 6006}%
\special{fp}%
\special{sh 1}%
\special{pa 5805 6006}%
\special{pa 5772 5945}%
\special{pa 5767 5968}%
\special{pa 5744 5973}%
\special{pa 5805 6006}%
\special{fp}%
}}%
%
{\color[named]{Black}{%
\special{pn 8}%
\special{pa 5255 5856}%
\special{pa 5405 6006}%
\special{fp}%
\special{sh 1}%
\special{pa 5405 6006}%
\special{pa 5372 5945}%
\special{pa 5367 5968}%
\special{pa 5344 5973}%
\special{pa 5405 6006}%
\special{fp}%
}}%
%
{\color[named]{Black}{%
\special{pn 8}%
\special{pa 4855 5856}%
\special{pa 5005 6006}%
\special{fp}%
\special{sh 1}%
\special{pa 5005 6006}%
\special{pa 4972 5945}%
\special{pa 4967 5968}%
\special{pa 4944 5973}%
\special{pa 5005 6006}%
\special{fp}%
}}%
%
{\color[named]{Black}{%
\special{pn 8}%
\special{pa 4455 5856}%
\special{pa 4605 6006}%
\special{fp}%
\special{sh 1}%
\special{pa 4605 6006}%
\special{pa 4572 5945}%
\special{pa 4567 5968}%
\special{pa 4544 5973}%
\special{pa 4605 6006}%
\special{fp}%
}}%
%
{\color[named]{Black}{%
\special{pn 8}%
\special{pa 5405 6006}%
\special{pa 5805 5606}%
\special{fp}%
\special{sh 1}%
\special{pa 5805 5606}%
\special{pa 5744 5639}%
\special{pa 5767 5644}%
\special{pa 5772 5667}%
\special{pa 5805 5606}%
\special{fp}%
}}%
%
{\color[named]{Black}{%
\special{pn 8}%
\special{pa 5005 6006}%
\special{pa 5405 5606}%
\special{fp}%
\special{sh 1}%
\special{pa 5405 5606}%
\special{pa 5344 5639}%
\special{pa 5367 5644}%
\special{pa 5372 5667}%
\special{pa 5405 5606}%
\special{fp}%
}}%
%
{\color[named]{Black}{%
\special{pn 8}%
\special{pa 4605 6006}%
\special{pa 5005 5606}%
\special{fp}%
\special{sh 1}%
\special{pa 5005 5606}%
\special{pa 4944 5639}%
\special{pa 4967 5644}%
\special{pa 4972 5667}%
\special{pa 5005 5606}%
\special{fp}%
}}%
%
{\color[named]{Black}{%
\special{pn 8}%
\special{pa 4205 6006}%
\special{pa 4605 5606}%
\special{fp}%
\special{sh 1}%
\special{pa 4605 5606}%
\special{pa 4544 5639}%
\special{pa 4567 5644}%
\special{pa 4572 5667}%
\special{pa 4605 5606}%
\special{fp}%
}}%
%
{\color[named]{Black}{%
\special{pn 8}%
\special{pa 4005 6006}%
\special{pa 4205 6006}%
\special{fp}%
\special{sh 1}%
\special{pa 4205 6006}%
\special{pa 4138 5986}%
\special{pa 4152 6006}%
\special{pa 4138 6026}%
\special{pa 4205 6006}%
\special{fp}%
}}%
%
{\color[named]{Black}{%
\special{pn 8}%
\special{pa 4005 5606}%
\special{pa 4205 5606}%
\special{fp}%
\special{sh 1}%
\special{pa 4205 5606}%
\special{pa 4138 5586}%
\special{pa 4152 5606}%
\special{pa 4138 5626}%
\special{pa 4205 5606}%
\special{fp}%
}}%
%
{\color[named]{Black}{%
\special{pn 8}%
\special{pa 4205 5606}%
\special{pa 4355 5756}%
\special{fp}%
}}%
%
{\color[named]{Black}{%
\special{pn 8}%
\special{pa 5405 5606}%
\special{pa 5555 5756}%
\special{fp}%
}}%
%
{\color[named]{Black}{%
\special{pn 8}%
\special{pa 5005 5606}%
\special{pa 5155 5756}%
\special{fp}%
}}%
%
{\color[named]{Black}{%
\special{pn 8}%
\special{pa 4605 5606}%
\special{pa 4755 5756}%
\special{fp}%
}}%
%
{\color[named]{Black}{%
\special{pn 8}%
\special{ar 6205 5406 200 200 4.7123890 1.5707963}%
}}%
%
{\color[named]{Black}{%
\special{pn 8}%
\special{pa 6220 5207}%
\special{pa 6205 5206}%
\special{fp}%
\special{sh 1}%
\special{pa 6205 5206}%
\special{pa 6270 5230}%
\special{pa 6258 5210}%
\special{pa 6273 5190}%
\special{pa 6205 5206}%
\special{fp}%
}}%
%
{\color[named]{Black}{%
\special{pn 8}%
\special{ar 6205 6206 200 200 4.7123890 1.5707963}%
}}%
%
{\color[named]{Black}{%
\special{pn 8}%
\special{pa 6220 6405}%
\special{pa 6205 6406}%
\special{fp}%
\special{sh 1}%
\special{pa 6205 6406}%
\special{pa 6273 6422}%
\special{pa 6258 6402}%
\special{pa 6270 6382}%
\special{pa 6205 6406}%
\special{fp}%
}}%
\put(38.0500,-58.0600){\makebox(0,0){{\color[named]{Black}{$L_0$}}}}%
%
{\color[named]{Black}{%
\special{pn 8}%
\special{pa 2200 5600}%
\special{pa 3000 5600}%
\special{fp}%
\special{sh 1}%
\special{pa 3000 5600}%
\special{pa 2933 5580}%
\special{pa 2947 5600}%
\special{pa 2933 5620}%
\special{pa 3000 5600}%
\special{fp}%
\special{pa 2200 6000}%
\special{pa 3000 6000}%
\special{fp}%
\special{sh 1}%
\special{pa 3000 6000}%
\special{pa 2933 5980}%
\special{pa 2947 6000}%
\special{pa 2933 6020}%
\special{pa 3000 6000}%
\special{fp}%
}}%
%
{\color[named]{Black}{%
\special{pn 8}%
\special{ar 6000 5450 250 250 0.6435011 2.4980915}%
}}%
%
{\color[named]{Black}{%
\special{pn 8}%
\special{pa 6191 5612}%
\special{pa 6200 5600}%
\special{fp}%
\special{sh 1}%
\special{pa 6200 5600}%
\special{pa 6144 5641}%
\special{pa 6168 5643}%
\special{pa 6176 5665}%
\special{pa 6200 5600}%
\special{fp}%
}}%
%
{\color[named]{Black}{%
\special{pn 8}%
\special{ar 6000 6150 250 250 3.7850938 5.6396842}%
}}%
%
{\color[named]{Black}{%
\special{pn 8}%
\special{pa 6191 5988}%
\special{pa 6200 6000}%
\special{fp}%
\special{sh 1}%
\special{pa 6200 6000}%
\special{pa 6176 5935}%
\special{pa 6168 5957}%
\special{pa 6144 5959}%
\special{pa 6200 6000}%
\special{fp}%
}}%
%
{\color[named]{Black}{%
\special{pn 8}%
\special{pa 2200 7805}%
\special{pa 2600 7405}%
\special{fp}%
\special{pa 2600 7405}%
\special{pa 2800 7405}%
\special{fp}%
\special{pa 2800 7005}%
\special{pa 3200 7405}%
\special{fp}%
\special{pa 2800 7005}%
\special{pa 2200 7005}%
\special{fp}%
\special{pa 2800 7405}%
\special{pa 2950 7255}%
\special{fp}%
\special{pa 3200 7005}%
\special{pa 3050 7155}%
\special{fp}%
\special{pa 3600 7405}%
\special{pa 3750 7255}%
\special{fp}%
\special{pa 4000 7005}%
\special{pa 3850 7155}%
\special{fp}%
\special{pa 3600 7005}%
\special{pa 4000 7405}%
\special{fp}%
\special{pa 4000 7405}%
\special{pa 4200 7405}%
\special{fp}%
\special{pa 4200 7405}%
\special{pa 4600 7405}%
\special{fp}%
\special{pa 4600 7005}%
\special{pa 4000 7005}%
\special{fp}%
\special{pa 4000 7805}%
\special{pa 4600 7805}%
\special{fp}%
\special{pa 4600 8205}%
\special{pa 4000 8205}%
\special{fp}%
\special{pa 4000 8205}%
\special{pa 2200 8205}%
\special{fp}%
\special{pa 2350 7555}%
\special{pa 2200 7405}%
\special{fp}%
\special{pa 2450 7655}%
\special{pa 2600 7805}%
\special{fp}%
\special{pa 2600 7805}%
\special{pa 4000 7805}%
\special{fp}%
}}%
%
{\color[named]{Black}{%
\special{pn 8}%
\special{ar 4600 7205 200 200 4.7123890 1.5707963}%
}}%
%
{\color[named]{Black}{%
\special{pn 8}%
\special{ar 4600 8005 200 200 4.7123890 1.5707963}%
}}%
\put(34.0000,-70.0500){\makebox(0,0){{\color[named]{Black}{$\cdots$}}}}%
\put(34.0000,-74.0500){\makebox(0,0){{\color[named]{Black}{$\cdots$}}}}%
\put(34.0000,-68.0000){\makebox(0,0){{\color[named]{Black}{trivial}}}}%
\end{picture}}%

%% file: proof_of_main_odd_minus.tex
{\unitlength 0.1in%
\begin{picture}(29.4000,16.8000)(2.6000,-20.3500)%
%
{\color[named]{Black}{%
\special{pn 8}%
\special{pa 1200 600}%
\special{pa 800 1800}%
\special{fp}%
\special{pa 1200 600}%
\special{pa 1600 1800}%
\special{fp}%
}}%
%
{\color[named]{Black}{%
\special{pn 8}%
\special{pa 1000 1200}%
\special{pa 1270 810}%
\special{fp}%
}}%
%
{\color[named]{Black}{%
\special{pn 8}%
\special{pa 800 1800}%
\special{pa 1539 1617}%
\special{fp}%
}}%
%
{\color[named]{Black}{%
\special{pn 8}%
\special{pa 1540 1620}%
\special{pa 950 1350}%
\special{fp}%
}}%
%
{\color[named]{Black}{%
\special{pn 8}%
\special{pa 1000 1200}%
\special{pa 1540 1620}%
\special{fp}%
}}%
\put(12.3000,-11.7000){\makebox(0,0){{\color[named]{Black}{$\vdots$}}}}%
\put(10.0000,-16.0000){\makebox(0,0){{\color[named]{Black}{$\vdots$}}}}%
\put(14.0000,-18.0000){\makebox(0,0){{\color[named]{Black}{$\vdots$}}}}%
%
{\color[named]{Black}{%
\special{pn 8}%
\special{pa 2800 600}%
\special{pa 2400 1800}%
\special{fp}%
\special{pa 2800 600}%
\special{pa 3200 1800}%
\special{fp}%
}}%
%
{\color[named]{Black}{%
\special{pn 8}%
\special{pa 2600 1200}%
\special{pa 2870 810}%
\special{fp}%
}}%
%
{\color[named]{Black}{%
\special{pn 8}%
\special{pa 2400 1800}%
\special{pa 3140 1620}%
\special{fp}%
}}%
%
{\color[named]{Black}{%
\special{pn 8}%
\special{pa 3140 1620}%
\special{pa 2550 1350}%
\special{fp}%
}}%
%
{\color[named]{Black}{%
\special{pn 8}%
\special{pa 2600 1200}%
\special{pa 3140 1620}%
\special{fp}%
}}%
\put(28.3000,-11.7000){\makebox(0,0){{\color[named]{Black}{$\vdots$}}}}%
\put(26.0000,-16.0000){\makebox(0,0){{\color[named]{Black}{$\vdots$}}}}%
\put(30.0000,-18.0000){\makebox(0,0){{\color[named]{Black}{$\vdots$}}}}%
%
{\color[named]{Black}{%
\special{pn 20}%
\special{pa 1200 600}%
\special{pa 1550 1650}%
\special{fp}%
}}%
%
{\color[named]{Black}{%
\special{pn 20}%
\special{pa 2800 600}%
\special{pa 3140 1620}%
\special{fp}%
\special{pa 3140 1620}%
\special{pa 2400 1800}%
\special{fp}%
}}%
\put(12.0000,-4.2000){\makebox(0,0){{\color[named]{Black}{$\cfrac{p}{q}$}}}}%
\put(28.0000,-4.2000){\makebox(0,0){{\color[named]{Black}{$\cfrac{p}{q}$}}}}%
\put(8.2000,-11.8000){\makebox(0,0){{\color[named]{Black}{$\cfrac{1}{1}$}}}}%
\put(6.2000,-17.8000){\makebox(0,0){{\color[named]{Black}{$\cfrac{1}{1}$}}}}%
\put(24.2000,-11.8000){\makebox(0,0){{\color[named]{Black}{$\cfrac{1}{1}$}}}}%
\put(28.0000,-21.0000){\makebox(0,0){{\color[named]{Black}{$a_{n-1}$ is odd.}}}}%
\put(12.0000,-21.0000){\makebox(0,0){{\color[named]{Black}{$a_{n-1}$ is even.}}}}%
\end{picture}}%

%% file: proof_of_main_odd_plus.tex
{\unitlength 0.1in%
\begin{picture}(40.1000,16.8500)(7.9000,-20.4000)%
%
{\color[named]{Black}{%
\special{pn 8}%
\special{pa 1200 600}%
\special{pa 800 1800}%
\special{fp}%
\special{pa 1200 600}%
\special{pa 1600 1800}%
\special{fp}%
}}%
%
{\color[named]{Black}{%
\special{pn 8}%
\special{pa 1000 1200}%
\special{pa 1270 810}%
\special{fp}%
}}%
%
{\color[named]{Black}{%
\special{pn 8}%
\special{pa 800 1800}%
\special{pa 1539 1617}%
\special{fp}%
}}%
%
{\color[named]{Black}{%
\special{pn 8}%
\special{pa 1540 1620}%
\special{pa 950 1350}%
\special{fp}%
}}%
%
{\color[named]{Black}{%
\special{pn 8}%
\special{pa 1000 1200}%
\special{pa 1540 1620}%
\special{fp}%
}}%
\put(12.3000,-11.7000){\makebox(0,0){{\color[named]{Black}{$\vdots$}}}}%
\put(10.0000,-16.0000){\makebox(0,0){{\color[named]{Black}{$\vdots$}}}}%
\put(14.0000,-18.0000){\makebox(0,0){{\color[named]{Black}{$\vdots$}}}}%
%
{\color[named]{Black}{%
\special{pn 8}%
\special{pa 2800 600}%
\special{pa 2400 1800}%
\special{fp}%
\special{pa 2800 600}%
\special{pa 3200 1800}%
\special{fp}%
}}%
%
{\color[named]{Black}{%
\special{pn 8}%
\special{pa 2600 1200}%
\special{pa 2870 810}%
\special{fp}%
}}%
%
{\color[named]{Black}{%
\special{pn 8}%
\special{pa 2400 1800}%
\special{pa 3140 1620}%
\special{fp}%
}}%
%
{\color[named]{Black}{%
\special{pn 8}%
\special{pa 3140 1620}%
\special{pa 2550 1350}%
\special{fp}%
}}%
%
{\color[named]{Black}{%
\special{pn 8}%
\special{pa 2600 1200}%
\special{pa 3140 1620}%
\special{fp}%
}}%
\put(28.3000,-11.7000){\makebox(0,0){{\color[named]{Black}{$\vdots$}}}}%
\put(26.0000,-16.0000){\makebox(0,0){{\color[named]{Black}{$\vdots$}}}}%
\put(30.0000,-18.0000){\makebox(0,0){{\color[named]{Black}{$\vdots$}}}}%
\put(12.0000,-4.2000){\makebox(0,0){{\color[named]{Black}{$\cfrac{p}{q}$}}}}%
\put(28.0000,-4.2000){\makebox(0,0){{\color[named]{Black}{$\cfrac{p}{q}$}}}}%
\put(14.0000,-7.8000){\makebox(0,0){{\color[named]{Black}{$\cfrac{1}{1}$}}}}%
\put(24.2000,-13.6000){\makebox(0,0){{\color[named]{Black}{$\cfrac{1}{1}$}}}}%
\put(30.0000,-7.8000){\makebox(0,0){{\color[named]{Black}{$\cfrac{1}{1}$}}}}%
\put(28.0000,-21.0000){\makebox(0,0){{\color[named]{Black}{$a_{n}$ and $a_{n-1}$ are even.}}}}%
\put(12.0000,-21.0000){\makebox(0,0){{\color[named]{Black}{$a_{n}$ is odd.}}}}%
%
{\color[named]{Black}{%
\special{pn 8}%
\special{pa 4400 605}%
\special{pa 4000 1805}%
\special{fp}%
\special{pa 4400 605}%
\special{pa 4800 1805}%
\special{fp}%
}}%
%
{\color[named]{Black}{%
\special{pn 8}%
\special{pa 4200 1205}%
\special{pa 4470 815}%
\special{fp}%
}}%
%
{\color[named]{Black}{%
\special{pn 8}%
\special{pa 4000 1805}%
\special{pa 4740 1625}%
\special{fp}%
}}%
%
{\color[named]{Black}{%
\special{pn 8}%
\special{pa 4740 1625}%
\special{pa 4150 1355}%
\special{fp}%
}}%
%
{\color[named]{Black}{%
\special{pn 8}%
\special{pa 4200 1205}%
\special{pa 4740 1625}%
\special{fp}%
}}%
\put(44.3000,-11.7500){\makebox(0,0){{\color[named]{Black}{$\vdots$}}}}%
\put(42.0000,-16.0500){\makebox(0,0){{\color[named]{Black}{$\vdots$}}}}%
\put(46.0000,-18.0500){\makebox(0,0){{\color[named]{Black}{$\vdots$}}}}%
\put(44.0000,-4.2500){\makebox(0,0){{\color[named]{Black}{$\cfrac{p}{q}$}}}}%
\put(46.0000,-7.8500){\makebox(0,0){{\color[named]{Black}{$\cfrac{1}{1}$}}}}%
\put(44.0000,-21.0500){\makebox(0,0){{\color[named]{Black}{$a_{n}$ is even and $a_{n-1}$ is odd.}}}}%
\put(16.2000,-15.8000){\makebox(0,0){{\color[named]{Black}{$\cfrac{1}{1}$}}}}%
%
{\color[named]{Black}{%
\special{pn 20}%
\special{pa 1200 600}%
\special{pa 800 1800}%
\special{fp}%
}}%
\put(40.2000,-13.6000){\makebox(0,0){{\color[named]{Black}{$\cfrac{1}{1}$}}}}%
\put(38.4000,-17.8000){\makebox(0,0){{\color[named]{Black}{$\cfrac{1}{1}$}}}}%
%
{\color[named]{Black}{%
\special{pn 20}%
\special{pa 2800 600}%
\special{pa 2600 1200}%
\special{fp}%
\special{pa 2600 1200}%
\special{pa 3140 1620}%
\special{fp}%
\special{pa 3140 1620}%
\special{pa 2400 1800}%
\special{fp}%
}}%
%
{\color[named]{Black}{%
\special{pn 20}%
\special{pa 4400 600}%
\special{pa 4200 1200}%
\special{fp}%
\special{pa 4200 1200}%
\special{pa 4740 1620}%
\special{fp}%
\special{pa 4740 1620}%
\special{pa 4770 1710}%
\special{fp}%
}}%
\end{picture}}%

%% file: proof_of_main_even_plus.tex
{\unitlength 0.1in%
\begin{picture}(24.0000,16.8500)(8.5000,-20.3500)%
%
{\color[named]{Black}{%
\special{pn 8}%
\special{pa 1300 595}%
\special{pa 1700 1795}%
\special{fp}%
\special{pa 1300 595}%
\special{pa 900 1795}%
\special{fp}%
}}%
%
{\color[named]{Black}{%
\special{pn 8}%
\special{pa 1500 1195}%
\special{pa 1230 805}%
\special{fp}%
}}%
%
{\color[named]{Black}{%
\special{pn 8}%
\special{pa 1700 1795}%
\special{pa 961 1612}%
\special{fp}%
}}%
%
{\color[named]{Black}{%
\special{pn 8}%
\special{pa 960 1615}%
\special{pa 1550 1345}%
\special{fp}%
}}%
%
{\color[named]{Black}{%
\special{pn 8}%
\special{pa 1500 1195}%
\special{pa 960 1615}%
\special{fp}%
}}%
\put(12.7000,-11.6500){\makebox(0,0){{\color[named]{Black}{$\vdots$}}}}%
\put(15.0000,-15.9500){\makebox(0,0){{\color[named]{Black}{$\vdots$}}}}%
\put(11.0000,-17.9500){\makebox(0,0){{\color[named]{Black}{$\vdots$}}}}%
%
{\color[named]{Black}{%
\special{pn 8}%
\special{pa 2850 600}%
\special{pa 3250 1800}%
\special{fp}%
\special{pa 2850 600}%
\special{pa 2450 1800}%
\special{fp}%
}}%
%
{\color[named]{Black}{%
\special{pn 8}%
\special{pa 3050 1200}%
\special{pa 2780 810}%
\special{fp}%
}}%
%
{\color[named]{Black}{%
\special{pn 8}%
\special{pa 3250 1800}%
\special{pa 2510 1620}%
\special{fp}%
}}%
%
{\color[named]{Black}{%
\special{pn 8}%
\special{pa 2510 1620}%
\special{pa 3100 1350}%
\special{fp}%
}}%
%
{\color[named]{Black}{%
\special{pn 8}%
\special{pa 3050 1200}%
\special{pa 2510 1620}%
\special{fp}%
}}%
\put(28.2000,-11.7000){\makebox(0,0){{\color[named]{Black}{$\vdots$}}}}%
\put(30.5000,-16.0000){\makebox(0,0){{\color[named]{Black}{$\vdots$}}}}%
\put(26.5000,-18.0000){\makebox(0,0){{\color[named]{Black}{$\vdots$}}}}%
%
{\color[named]{Black}{%
\special{pn 20}%
\special{pa 1300 595}%
\special{pa 950 1645}%
\special{fp}%
}}%
%
{\color[named]{Black}{%
\special{pn 20}%
\special{pa 2850 600}%
\special{pa 2510 1620}%
\special{fp}%
\special{pa 2510 1620}%
\special{pa 3250 1800}%
\special{fp}%
}}%
\put(13.0000,-4.1500){\makebox(0,0){{\color[named]{Black}{$\cfrac{p}{q}$}}}}%
\put(28.5000,-4.2000){\makebox(0,0){{\color[named]{Black}{$\cfrac{p}{q}$}}}}%
\put(16.8000,-11.7500){\makebox(0,0){{\color[named]{Black}{$\cfrac{1}{1}$}}}}%
\put(18.8000,-17.7500){\makebox(0,0){{\color[named]{Black}{$\cfrac{1}{1}$}}}}%
\put(32.3000,-11.8000){\makebox(0,0){{\color[named]{Black}{$\cfrac{1}{1}$}}}}%
\put(28.5000,-21.0000){\makebox(0,0){{\color[named]{Black}{$a_{n-1}$ is odd.}}}}%
\put(13.0000,-20.9500){\makebox(0,0){{\color[named]{Black}{$a_{n-1}$ is even.}}}}%
\end{picture}}%

%% file: proof_of_main_even_minus.tex
{\unitlength 0.1in%
\begin{picture}(28.2000,16.9000)(19.9000,-20.3000)%
%
{\color[named]{Black}{%
\special{pn 8}%
\special{pa 4410 595}%
\special{pa 4810 1795}%
\special{fp}%
\special{pa 4410 595}%
\special{pa 4010 1795}%
\special{fp}%
}}%
%
{\color[named]{Black}{%
\special{pn 8}%
\special{pa 4610 1195}%
\special{pa 4340 805}%
\special{fp}%
}}%
%
{\color[named]{Black}{%
\special{pn 8}%
\special{pa 4810 1795}%
\special{pa 4071 1612}%
\special{fp}%
}}%
%
{\color[named]{Black}{%
\special{pn 8}%
\special{pa 4070 1615}%
\special{pa 4660 1345}%
\special{fp}%
}}%
%
{\color[named]{Black}{%
\special{pn 8}%
\special{pa 4610 1195}%
\special{pa 4070 1615}%
\special{fp}%
}}%
\put(43.8000,-11.6500){\makebox(0,0){{\color[named]{Black}{$\vdots$}}}}%
\put(46.1000,-15.9500){\makebox(0,0){{\color[named]{Black}{$\vdots$}}}}%
\put(42.1000,-17.9500){\makebox(0,0){{\color[named]{Black}{$\vdots$}}}}%
%
{\color[named]{Black}{%
\special{pn 8}%
\special{pa 2800 585}%
\special{pa 3200 1785}%
\special{fp}%
\special{pa 2800 585}%
\special{pa 2400 1785}%
\special{fp}%
}}%
%
{\color[named]{Black}{%
\special{pn 8}%
\special{pa 3000 1185}%
\special{pa 2730 795}%
\special{fp}%
}}%
%
{\color[named]{Black}{%
\special{pn 8}%
\special{pa 3200 1785}%
\special{pa 2460 1605}%
\special{fp}%
}}%
%
{\color[named]{Black}{%
\special{pn 8}%
\special{pa 2460 1605}%
\special{pa 3050 1335}%
\special{fp}%
}}%
%
{\color[named]{Black}{%
\special{pn 8}%
\special{pa 3000 1185}%
\special{pa 2460 1605}%
\special{fp}%
}}%
\put(27.7000,-11.5500){\makebox(0,0){{\color[named]{Black}{$\vdots$}}}}%
\put(30.0000,-15.8500){\makebox(0,0){{\color[named]{Black}{$\vdots$}}}}%
\put(26.0000,-17.8500){\makebox(0,0){{\color[named]{Black}{$\vdots$}}}}%
\put(44.1000,-4.1500){\makebox(0,0){{\color[named]{Black}{$\cfrac{p}{q}$}}}}%
\put(28.0000,-4.0500){\makebox(0,0){{\color[named]{Black}{$\cfrac{p}{q}$}}}}%
\put(42.1000,-7.8000){\makebox(0,0){{\color[named]{Black}{$\cfrac{1}{1}$}}}}%
\put(47.9000,-13.5000){\makebox(0,0){{\color[named]{Black}{$\cfrac{1}{1}$}}}}%
\put(26.1000,-7.8500){\makebox(0,0){{\color[named]{Black}{$\cfrac{1}{1}$}}}}%
\put(28.0000,-20.8500){\makebox(0,0){{\color[named]{Black}{$a_{n}$ is odd.}}}}%
\put(44.1000,-20.9500){\makebox(0,0){{\color[named]{Black}{$a_{n}$ is even.}}}}%
\put(23.5000,-15.6000){\makebox(0,0){{\color[named]{Black}{$\cfrac{1}{1}$}}}}%
%
{\color[named]{Black}{%
\special{pn 20}%
\special{pa 4600 1195}%
\special{pa 4400 595}%
\special{fp}%
\special{pa 4600 1195}%
\special{pa 4070 1615}%
\special{fp}%
}}%
%
{\color[named]{Black}{%
\special{pn 20}%
\special{pa 2800 600}%
\special{pa 3200 1800}%
\special{fp}%
}}%
\end{picture}}%

%% file: example_of_main_27.tex
{\unitlength 0.1in%
\begin{picture}(17.4000,20.3500)(4.6000,-26.0000)%
%
{\color[named]{Black}{%
\special{pn 8}%
\special{pa 1600 800}%
\special{pa 2200 2600}%
\special{fp}%
\special{pa 2200 2600}%
\special{pa 1000 2600}%
\special{fp}%
\special{pa 1000 2600}%
\special{pa 1600 800}%
\special{fp}%
\special{pa 1400 1400}%
\special{pa 1900 1700}%
\special{fp}%
\special{pa 1900 1700}%
\special{pa 1000 2600}%
\special{fp}%
\special{pa 1000 2600}%
\special{pa 2050 2150}%
\special{fp}%
}}%
\put(16.0000,-6.3000){\makebox(0,0){{\color[named]{Black}{$\cfrac{2}{7}$}}}}%
\put(12.0000,-14.0000){\makebox(0,0){{\color[named]{Black}{$\cfrac{1}{4}$}}}}%
\put(21.0000,-17.0000){\makebox(0,0){{\color[named]{Black}{$\cfrac{1}{3}$}}}}%
\put(22.5000,-21.5000){\makebox(0,0){{\color[named]{Black}{$\cfrac{1}{2}$}}}}%
\put(24.0000,-26.3000){\makebox(0,0){{\color[named]{Black}{$\cfrac{1}{1}$}}}}%
\put(8.0000,-26.3000){\makebox(0,0){{\color[named]{Black}{$\cfrac{0}{1}$}}}}%
%
{\color[named]{Black}{%
\special{pn 20}%
\special{pa 1600 800}%
\special{pa 1000 2600}%
\special{fp}%
}}%
\end{picture}}%

%% file: example_of_main5.tex
{\unitlength 0.1in%
\begin{picture}(39.4000,22.6500)(8.6000,-24.0500)%
%
{\color[named]{Black}{%
\special{pn 8}%
\special{pa 2200 2400}%
\special{pa 1600 1800}%
\special{fp}%
\special{pa 1600 1800}%
\special{pa 1930 1590}%
\special{fp}%
}}%
%
{\color[named]{Black}{%
\special{pn 8}%
\special{pa 1800 1200}%
\special{pa 1400 2400}%
\special{fp}%
\special{pa 1400 2400}%
\special{pa 2200 2400}%
\special{fp}%
\special{pa 2200 2400}%
\special{pa 1800 1200}%
\special{fp}%
}}%
%
{\color[named]{Black}{%
\special{pn 20}%
\special{pa 1800 1200}%
\special{pa 1600 1800}%
\special{fp}%
\special{pa 1600 1800}%
\special{pa 2200 2400}%
\special{fp}%
}}%
\put(18.0000,-10.2000){\makebox(0,0){{\color[named]{Black}{$\cfrac{3}{5}$}}}}%
\put(21.1000,-15.7000){\makebox(0,0){{\color[named]{Black}{$\cfrac{2}{3}$}}}}%
\put(14.2000,-17.8000){\makebox(0,0){{\color[named]{Black}{$\cfrac{1}{2}$}}}}%
\put(12.2000,-23.8000){\makebox(0,0){{\color[named]{Black}{$\cfrac{0}{1}$}}}}%
\put(23.8000,-23.8000){\makebox(0,0){{\color[named]{Black}{$\cfrac{1}{1}$}}}}%
\put(18.0000,-15.3000){\makebox(0,0){{\color[named]{Black}{$+1$}}}}%
\put(19.0000,-19.0000){\makebox(0,0){{\color[named]{Black}{$-1$}}}}%
\put(17.0000,-22.0000){\makebox(0,0){{\color[named]{Black}{$+1$}}}}%
%
{\color[named]{Black}{%
\special{pn 8}%
\special{pa 4160 815}%
\special{pa 3890 675}%
\special{fp}%
}}%
%
{\color[named]{Black}{%
\special{pn 8}%
\special{pa 3520 1605}%
\special{pa 4320 1205}%
\special{fp}%
}}%
%
{\color[named]{Black}{%
\special{pn 8}%
\special{pa 4000 405}%
\special{pa 3200 2405}%
\special{fp}%
}}%
%
{\color[named]{Black}{%
\special{pn 8}%
\special{pa 4800 2405}%
\special{pa 4000 405}%
\special{fp}%
}}%
\put(49.5000,-21.9500){\makebox(0,0)[rt]{{\color[named]{Black}{$\cfrac{1}{1}$}}}}%
\put(30.5000,-21.9500){\makebox(0,0)[lt]{{\color[named]{Black}{$\cfrac{0}{1}$}}}}%
\put(47.2000,-18.9500){\makebox(0,0){{\color[named]{Black}{$\cfrac{1}{2}$}}}}%
\put(34.0000,-16.0500){\makebox(0,0){{\color[named]{Black}{$\cfrac{1}{3}$}}}}%
\put(44.5000,-12.0500){\makebox(0,0){{\color[named]{Black}{$\cfrac{2}{5}$}}}}%
\put(42.7000,-7.7500){\makebox(0,0){{\color[named]{Black}{$\cfrac{3}{8}$}}}}%
\put(37.5000,-6.5500){\makebox(0,0){{\color[named]{Black}{$\cfrac{4}{11}$}}}}%
\put(40.0000,-2.0500){\makebox(0,0){{\color[named]{Black}{$\cfrac{7}{19}$}}}}%
%
{\color[named]{Black}{%
\special{pn 8}%
\special{pa 3200 2405}%
\special{pa 4800 2405}%
\special{fp}%
}}%
%
{\color[named]{Black}{%
\special{pn 8}%
\special{pa 3200 2405}%
\special{pa 4600 1905}%
\special{fp}%
}}%
%
{\color[named]{Black}{%
\special{pn 8}%
\special{pa 4600 1905}%
\special{pa 3530 1595}%
\special{fp}%
}}%
%
{\color[named]{Black}{%
\special{pn 8}%
\special{pa 3530 1595}%
\special{pa 4160 795}%
\special{fp}%
}}%
%
{\color[named]{Black}{%
\special{pn 20}%
\special{pa 4000 400}%
\special{pa 4158 801}%
\special{fp}%
\special{pa 4158 801}%
\special{pa 3530 1594}%
\special{fp}%
\special{pa 3530 1594}%
\special{pa 4598 1903}%
\special{fp}%
\special{pa 4598 1903}%
\special{pa 4800 2400}%
\special{fp}%
}}%
\put(44.0000,-22.0000){\makebox(0,0){{\color[named]{Black}{$+1$}}}}%
\put(37.0000,-19.5000){\makebox(0,0){{\color[named]{Black}{$+1$}}}}%
\put(42.0000,-15.5000){\makebox(0,0){{\color[named]{Black}{$-1$}}}}%
\put(41.0000,-11.0000){\makebox(0,0){{\color[named]{Black}{$+1$}}}}%
\put(39.2000,-9.0000){\makebox(0,0){{\color[named]{Black}{$-1$}}}}%
\put(40.0000,-6.5000){\makebox(0,0){{\color[named]{Black}{$+1$}}}}%
\end{picture}}%